\documentclass[12pt]{article}
\usepackage{amsmath}
\usepackage{graphicx}
\usepackage{enumerate}
\usepackage{natbib}
\usepackage{url} 
\usepackage{xcolor}
\usepackage{enumitem}
\newcommand{\blind}{1}

\DeclareMathOperator{\Cov}{Cov}
\DeclareMathOperator{\sign}{sign}
\usepackage{amsthm}
\newtheorem{theorem}{Theorem}[section]
\newtheorem{lemma}[theorem]{Lemma}
\newtheorem{corollary}[theorem]{Corollary}

\newtheorem{proposition}[theorem]{Proposition}
\newtheorem{remark}[theorem]{Remark}
\newtheorem{assumption}[theorem]{Assumption}
\newtheorem{example}[theorem]{Example}

\usepackage{float}
\usepackage{mathrsfs}
\usepackage{amsfonts}
\usepackage{bbm}
\usepackage{bm}
\usepackage{amsmath}
\usepackage{amssymb}
\usepackage{graphicx}
\usepackage{tabularx}
\usepackage{float}

\begin{document}

\def\spacingset#1{\renewcommand{\baselinestretch}%
{#1}\small\normalsize} \spacingset{1}


\if1\blind
{
  \title{\bf Asymptotic Behavior of Adversarial Training Estimator under $\ell_\infty$-Perturbation}
  \author{Yiling Xie\\
    School of Industrial and Systems Engineering,\\  Georgia Institute of Technology,\\
    and \\
    Xiaoming Huo \\
    School of Industrial and Systems Engineering,\\  Georgia Institute of Technology.}
  \maketitle
} \fi

\if0\blind
{
  \bigskip
  \bigskip
  \bigskip
  \begin{center}
    {\LARGE\bf Asymptotic Behavior of Adversarial Training Estimator under $\ell_\infty$-Perturbation}
\end{center}
  \medskip
} \fi

\bigskip
\sloppy
\spacingset{1.9}
\begin{abstract}
Adversarial training has been proposed to protect machine learning models against adversarial attacks. 
This paper focuses on adversarial training under $\ell_\infty$-perturbation, which has recently attracted much research attention.
The asymptotic behavior of the adversarial training estimator is investigated in the generalized linear model.
The results imply that the asymptotic distribution of the adversarial training estimator under $\ell_\infty$-perturbation could put a positive probability mass at $0$ when the true parameter is $0$, providing a theoretical guarantee of the associated sparsity-recovery ability.
Alternatively, a two-step procedure is proposed---adaptive adversarial training, which could further improve the performance of adversarial training under $\ell_\infty$-perturbation. 
Specifically, the proposed procedure could achieve asymptotic variable-selection consistency and unbiasedness.
Numerical experiments are conducted to show the sparsity-recovery ability of adversarial training under $\ell_\infty$-perturbation and to compare the empirical performance between classic adversarial training and adaptive adversarial training. 
\end{abstract} 
\section{Introduction}\label{introductionsection}
Modern machine-learning models are
susceptible to adversarial attacks on their inputs. The adversarial training procedure has been proposed to remedy this issue by minimizing the empirical worst-case loss under a given magnitude of perturbation for each observation \citep{Goodfellow2015,madry2018towards}.
 In this paper, we consider the adversarial training problem formulated as the following optimization problem:
\begin{equation}\label{intro}\min_{\beta}\mathbb{E}_{\left( \bm{X},Y\right)\sim\mathbb{P}_n}\left[ \max_{\left\Vert\Delta\right\Vert\leq \delta} L\left(\bm{X}+\Delta,Y,\beta\right)\right],\end{equation}
where $\Vert\Delta\Vert$ denotes the norm of the perturbation $\Delta$, $\delta$ denotes the perturbation magnitude, $\bm{X}\in\mathbb{R}^d$ and $Y\in\mathbb{R}$ denote the input variable and response variable, $L$ is the loss function, $\mathbb{P}_n$ is the empirical distribution of $(\bm{X}, Y)$, and  $\mathbb{P}_n$  is constructed by $n$ independent and identically distributed samples drawn from the ground-truth data generating distribution $P_\ast$.

If the norm of the perturbation $\Vert\cdot\Vert$ is chosen as $\ell_p$ norm, we call the problem \eqref{intro} adversarial training under $\ell_p$-perturbation.
 Compared with other types of norms of the perturbation, adversarial training under $\ell_\infty$-perturbation is of particular interest.
 Intuitively, applying $\ell_\infty$-perturbation imposes perturbations uniformly in all components of the input variable.
This perturbation choice is widely used in the literature \citep{madry2018towards,schmidt2018adversarially,yin2019rademacher,chen2020more,guo2020connections,xing2021generalization}. 
For example, 
\citet{yin2019rademacher} studies the generalization ability of the adversarial training under $\ell_\infty$-perturbation, and
\citet{schmidt2018adversarially,chen2020more} investigate the generalization gap between
standard and adversarial robust models under $\ell_\infty$-perturbation. 
In addition, the $\ell_\infty$-perturbed adversarial training has been shown to enjoy desirable properties. 
For instance, \cite{xie2024highdimensional} proves that  $\ell_\infty$-perturbed adversarially-trained estimator can achieve the minimax optimal convergence rate in the high-dimensional linear regression. 
\cite{ribeiro2023regularization} points out that
the adversarial training under $\ell_\infty$-perturbation could produce sparse solutions in the linear regression.
However, the existing literature does not provide a precise statistical analysis to explain this sparsity-recovery phenomenon. Our paper aims to provide a rigorous theoretical foundation for this sparsity-recovery ability within the context of the generalized linear model, where linear regression is a special case.

In this paper, we first study the asymptotic behavior of the adversarial training estimator under the generalized linear model.
We derive the associated asymptotic distribution based on the regularization effect of the adversarial training procedure.
The results have two main implications.
Firstly, $\ell_\infty$-perturbation could help recover sparsity asymptotically while other types of perturbation could not. 
Secondly, the asymptotic sparsity-recovery ability of the $\ell_\infty$-perturbation can only be achieved when the perturbation magnitude is of the order $1/\sqrt{n}$, i.e., $\delta = \eta/n^{\gamma}$, where $\eta > 0$ and $\gamma = 1/2$.
More specifically, if we choose $\gamma=1/2$, the asymptotic distribution of the $\ell_\infty$-perturbed adversarial-trained estimator has a positive probability at $0$  when the true parameter is $0$.

To further improve the performance of adversarial training, we propose a novel two-step procedure---adaptive adversarial training.
In this approach, we first compute the empirical risk minimization estimator and then use the empirical risk minimization estimator as the weights added to the $\ell_\infty$-perturbation.
The estimator obtained from the proposed procedure can achieve variable-selection consistency in the asymptotic sense, while the classic adversarial training estimator can not.
Moreover, the adaptive adversarial training estimator is asymptotically unbiased when $1/2<\gamma<1$.

Our theoretical results are validated by numerical experiments. 
Specifically, we run adversarial training and adaptive adversarial training in linear and logistic regression on both synthetic and real-world datasets. 
The associated estimator path performs different patterns when the order of $\ell_\infty$-perturbation is chosen as $0<\gamma<1/2,\gamma=1/2,\gamma>1/2$, respectively, among which the choice $\gamma=1/2$ has the best performance in general.
We also compare the performance of classic adversarial training with that of adaptive adversarial training. 
It is observed that the proposed adaptive adversarial training has superior practical performance in estimation accuracy and sparsity-recovery.
 
\subsection{Related Work}
In the recent literature of adversarial training, \citet{javanmard2022precise, taheri2023asymptotic,hassani2024curse} investigate the asymptotic accuracy of adversarial training under the high-dimensional regime. 
However, our work focuses on the $\ell_\infty$-perturbed adversarial training estimator, analyzing its asymptotic distribution and relevant properties, such as sparsity recovery and variable-selection consistency, in fixed-dimensional settings.
In addition to the aforementioned literature, several other studies, such as \cite{dobriban2023provable, trillos2023multimarginal}, have explored the statistical properties of adversarially-trained classification problems, including Bayes-optimality and geometric structure, which differ from the focus of this paper.

Some studies have found that the adversarial training procedure is equivalent to $\infty$-Wasserstein distributionally robust optimization \citep{staib2017distributionally,gao2024wasserstein}. While there is extensive research on finite-order Wasserstein distributionally robust optimization \citep{blanchet2022confidence, gao2022distributionally,blanchet2023statistical}, the statistical properties of adversarial training estimator have not been thoroughly examined. 
Our results are consistent with the associated theories in finite-order Wasserstein distributionally robust optimization. However, we emphasize the statistical insights derived from the corresponding asymptotic distribution and propose potential improvements from a statistical perspective.

Adversarial training under $\ell_\infty$-perturbation is related to the well-known statistical procedure LASSO \citep{ribeiro2023regularization,xie2024highdimensional}. 
The regularization effect of adversarial training will be demonstrated to depend on the $\ell_1$-norm of the parameter.
In this regard, it will also be shown that the asymptotic distribution of the adversarial training estimator under $\ell_\infty$-perturbation follows a formulation that aligns with that of the LASSO estimator \citep{fu2000asymptotics}, while incorporating distinct gradient term specific to adversarial robustness.
Additionally, the proposed adaptive adversarial training connects with the adaptive LASSO \citep{zou2006adaptive}.
Specifically, adaptive adversarial training applies weights in the $\ell_\infty$-perturbation, while adaptive LASSO applies weights in the $\ell_1$-norm regularization term. 
Both weighting strategies lead to similar properties in the resulting estimator, achieving variable-selection consistency and reducing bias.
\subsection{Notations and Definitions}
$[\bm{x}]_i$ denotes the $i$th component of vector $\bm{x}$. 
$[\bm{x}]_\mathcal{I}$ denotes the sub-vector of $\bm{x}$ with $\mathcal{I}$ as the index set.
$\bm{1}$ denotes the vector of all ones.
$\bm{I}_d$ denotes the $d$-dimensional identity matrix.
$\sign([\bm{x}]_i)$ denotes the sign of $[\bm{x}]_i$ (where we let $\sign(0)=0$),
 and accordingly we let $\sign(\bm{x})=(\sign([\bm{x}]_1),...,\sign([\bm{x}]_d))^\top$.
If $\bm{x}=(x_1,...,x_d)^\top$ and $\bm{y}=(y_1,...,y_d)^\top$, then we have the following denotations: $\vert\bm{x}\vert=(\vert x_1\vert,...,\vert x_d\vert)^\top$, $\bm{x}^{-1}=(1/x_1,...,1/x_d)^\top$, $\bm{x}\otimes\bm{y}=(x_1y_1,...,x_dy_d)^\top$, $\langle\bm{x},\bm{y}\rangle=\bm{x}^\top\bm{y}=x_1y_1+...+x_dy_d$.
We use $\Vert\bm{x}\Vert_p$ to denote the $\ell_p$-norm of the vector $\bm{z}$, i.e., $\Vert\bm{x}\Vert_p^p=\sum_{j=1}^d\vert x_j \vert^p, 1\leq p<\infty, \Vert\bm{x}\Vert_\infty=\max_{1\leq j\leq d} \vert x_j\vert.$
For a given general  norm $\Vert\cdot\Vert$, we denote the associated dual norm by $\Vert\cdot\Vert_\ast$, which is defined by $\Vert \bm{y}\Vert_\ast=\max_{\Vert \bm{x}\Vert\leq 1} \bm{x}^{\top}\bm{y}$. 
For a matrix $M$, we use $\Vert M\Vert$ to denote the matrix norm of $M$ induced by the vector norm $\Vert \cdot\Vert$ , i.e., $\|M\|
= \max_{\bm{x} \neq 0} \|M\bm{x}\|/\|\bm{x}\|.$
``$\Rightarrow$'' denotes ``converge in distribution'' and ``$\to_p$'' denotes ``converge in probability''. 
``$O_p(1)$'' denotes that the term is ``bounded in probability''.
``$o_p(1)$'' denotes  that the term  ``converges to $0$ in probability''.
If we denote a term $R$ as ``$o(\delta^k)$'' where $k>0$, it means that $\lim_{\delta\to 0}R/\delta^k=0$.
 For the function $g(\bm{x},y,\beta)$, we use $\nabla_{\beta}g(\bm{x},y,\beta^\ast)$ to denote the value of $\nabla_{\beta}g(\bm{x},y,\beta)$ at $\beta=\beta^\ast$, i.e., $\nabla_{\beta}g(\bm{x},y,\beta^\ast)=\nabla_{\beta}g(\bm{x},y,\beta)\big\vert_{\beta=\beta^\ast}$. A similar notation also applies to the second-order gradients.
$\mathbb{I}(\cdot)$ denotes the indicator function.
\subsection{Organization of this Paper}
The remainder of this paper is organized as follows. 
In Section \ref{reg}, we introduce the regularization effect of adversarial training. 
In Section \ref{asymsection}, we investigate the asymptotic behavior of the adversarial training estimator under $\ell_\infty$-perturbation.
In Section \ref{adaptivesection}, we propose adaptive adversarial training and analyze its desirable statistical properties.
Numerical experiments are conducted and analyzed in Section \ref{numsection}.
We discuss some future work in Section \ref{dis}.
The proofs are relegated to the Supplementary Material whenever possible.
\section{Regularization Effect of Adversarial Training}\label{reg}
This section discusses the regularization effect of the adversarial training procedure.

In the rest of the paper, we adopt the convention that a matrix-valued function $g:\mathbb{R}^d\to \mathbb{R}^{d_1\times d_2}$ is uniformly continuous if each entry of $g$, as a function from $\mathbb{R}^d$ to $\mathbb{R}$, is uniformly continuous.  We state the regularization effect of adversarial training in the following proposition.
\begin{proposition}[Regularization Effect]\label{prop1}
If the function $h:\mathbb{R}^d\to\mathbb{R}$ is differentiable, $\nabla h$ is uniformly continuous, and $\mathbb{E}_{\bm{Z}\sim P}\left[ \Vert\nabla h(\bm{Z})\Vert_\ast\right]<\infty$,
then we have that
\begin{equation}\label{regularization}\mathbb{E}_{\bm{Z}\sim P}\left[\max _{\Vert\Delta\Vert\leq \delta}h(\bm{Z}+\Delta)\right]=\mathbb{E}_ {\bm{Z}\sim P}\left[ h(\bm{Z})\right]+ \delta \mathbb{E}_{\bm{Z}\sim P}\left[ \Vert\nabla h(\bm{Z})\Vert_\ast\right]+o(\delta)\end{equation}
as $\delta\to 0$.
\end{proposition}

\begin{remark}
Compared with the adversarial training problem \eqref{intro} introduced in Section \ref{introductionsection}, Proposition \ref{prop1} focuses on a more general loss function $h$ and decision variable $\bm{Z}$.
\end{remark}
Proposition \ref{prop1} demonstrates that the adversarial training worst-case loss, $\mathbb{E}_{\bm{Z} \sim P}\left[\max_{\Vert\Delta\Vert \leq \delta} h(\bm{Z} + \Delta)\right]$, can be reformulated as the sum of the classical non-adversarial loss, $\mathbb{E}_{\bm{Z} \sim P}\left[h(\bm{Z})\right]$, the expectation of the dual norm of the gradient, $\mathbb{E}_{\bm{Z} \sim P}\left[\Vert\nabla h(\bm{Z})\Vert_\ast\right]$, and a higher-order residual, $o(\delta)$.
Regardless of the higher-order residual, the term $\delta \mathbb{E}_{\bm{Z} \sim P}\left[\Vert\nabla h(\bm{Z})\Vert_\ast\right]$ can be interpreted as a regularization imposed on the non-adversarial loss. 
This additional gradient regularization term is caused by the computation of the worst-case loss.
 Since a smaller gradient can lead to solutions that are less affected by input perturbations, the gradient regularization term can be considered as a mechanism that enhances robustness against adversarial input attacks. We refer to this phenomenon as the regularization effect of adversarial training.
The regularization effect has also been identified in the context of the distributionally robust optimization problem \citep{gao2024wasserstein}.
In particular, we will focus on the regularization effect of adversarial training within the generalized linear model,  as illustrated in the next section.
Moreover, the absence of the ``sup" in the reformulation \eqref{regularization} enables the analysis of the asymptotic distribution of the adversarial training estimator.

\section{Asymptotic Behavior in Generalized Linear Model}\label{asymsection}
This section analyzes the asymptotic behavior of the adversarial training estimator under $\ell_\infty$-perturbation in the generalized linear model.

In the generalized linear model, the response variable $Y$ is assumed to follow some distribution in the exponential family.
Common examples include the Bernoulli distribution for 
\(Y \in \{-1,1\}\) in logistic regression, the Poisson distribution for 
\(Y \in \{0,1,2,\ldots\}\) in Poisson regression, and the normal distribution 
for \(Y \in \mathbb{R}\) in linear regression.
The relationship between the response variable \(Y\) and the input variable \(\bm{X}\) is characterized 
through a link function \(G\). Specifically, if we denote the (nonzero) 
ground-truth parameter by \(\beta^* \in \mathbb{R}^d\), then we have that $G(\mathbb{E}[Y|\bm{X}=\bm{x}])=\langle \bm{x},\beta^{\ast}\rangle.$
Examples of \(G\) include the logit function (logistic regression), the log 
function (Poisson regression), and the identity function (linear regression).
In practice, the ground-truth parameter \(\beta^*\) is estimated using maximum likelihood estimation, resulting in a loss function, denoted with slight abuse of notation, as $L(\bm{x},y,\beta)=L(\langle\bm{x},\beta\rangle,y)$.

The following corollary could be immediately obtained by Proposition \ref{prop1} and shows the regularization effect of the adversarial training procedure under $\ell_\infty$-perturbation in the generalized linear model.
\begin{corollary}\label{lemma}
If the function $L(f,y):\mathbb{R}\times \mathbb{R}\to\mathbb{R}$ is differentiable w.r.t. the first argument,  and $L^{\prime}(f,y)$ is uniformly continuous w.r.t. the first argument, where $L^\prime(f,y)$ denotes the first-order derivative of $f$ taken w.r.t. the first argument, then we have that 
\begin{equation}\label{regularizationeffect}\mathbb{E}_{\mathbb{P}_n}\left[\max _{\Vert\Delta\Vert_{\infty}\leq \delta}L(\langle \bm{X}+\Delta,\beta\rangle,Y)\right]=\mathbb{E}_{\mathbb{P}_n}\left[ L(\langle\bm{X},\beta\rangle,Y)\right]+\delta \Vert\beta\Vert_1\mathbb{E}_{\mathbb{P}_n}\left[ \left\vert L^\prime(\langle\bm{X},\beta\rangle,Y)\right\vert\right]+o(\delta)\end{equation}
 holds as $\delta\to 0$ for $\beta\in B$, where $B$ is some compact subset of $\mathbb{R}^d$.
\end{corollary}
Corollary \ref{lemma} implies that the empirical average worst-case loss in adversarial training,  $\mathbb{E}_{\mathbb{P}_n}\left[\max _{\Vert\Delta\Vert_{\infty}\leq \delta}L(\langle \bm{X}+\Delta,\beta\rangle,Y)\right]$, can be reformulated as the sum of the empirical average loss, $\mathbb{E}_{\mathbb{P}_n}\left[ L(\langle\bm{X},\beta\rangle,Y)\right]$, the regularization term,
 $\delta \Vert\beta\Vert_1\mathbb{E}_{\mathbb{P}_n}\left[ \left\vert L^\prime(\langle\bm{X},\beta\rangle,Y)\right\vert\right]$, and a high-order residual, $o(\delta)$. 
 The regularization term is the product of the perturbation magnitude $\delta$, the parameter $\ell_1$-norm $\Vert\beta\Vert_1$ and the empirical average of the absolute value of the loss gradient $\mathbb{E}_{\mathbb{P}_n}\left[ \left\vert L^\prime(\langle\bm{X},\beta\rangle,Y)\right\vert\right]$. 
Notably, the gradient regularization can improve the robustness towards adversarial input perturbations \citep{drucker1992improving} while the $\ell_1$-norm regularization can reduce the model complexity and improve the robustness towards overfitting \citep{hastie2015statistical}.
Hence, the regularization of the $\ell_\infty$-pertured adversarial training brings robustness against both adversarial perturbations and overfitting.

Taking advantage of the regularization effect of adversarial training, we could obtain the asymptotic behavior of the adversarial training estimator in the generalized linear model. In particular, the associated estimator is defined by the following:
\begin{equation}\label{advestimator}\beta^n\in\arg\min_{\beta\in B}\mathbb{E}_{\mathbb{P}_n}\left[\max _{\Vert\Delta\Vert_{\infty}\leq \delta_n}L(\langle \bm{X}+\Delta,\beta\rangle,Y)\right],\end{equation}
where $B$ is a convex compact subset of $\mathbb{R}^d$, and $\delta_n=\eta/n^\gamma$, $\eta,\gamma>0$.

We first clarify several assumptions before introducing our main theorem.
     \begin{assumption}\label{assume1main} In problem \eqref{advestimator}, for the loss function $L$, the ground-truth parameter $\beta^\ast\not=\bm{0}$, and true underlying distribution $P_\ast$, we assume the following conditions hold:
    \begin{itemize}
        \item The ground-truth parameter $\beta^\ast$ is an interior point of $B$.
        \item The loss function $L(f,y)$ is twice differentiable w.r.t. the first argument.
        
        \item The function $L^{\prime\prime}(f,y)$ is uniformly continuous w.r.t. the first argument, where  $L^{\prime\prime}(f,y)$ denotes the second-order derivative of $L$ taken w.r.t. the first argument.
        
        \item The loss function $L(\langle \bm{x},\beta\rangle,y)$ is convex w.r.t. $\beta\in B$.
        \item  The matrix $H=\mathbb{E}_{P_\ast}\left[ \nabla_{\beta}^2 L(\langle \bm{X},\beta^\ast\rangle,Y)\right]$ is nonsingular.
        \item The equation $P_\ast\left(L^\prime(\langle \bm{X},\beta^\ast\rangle,Y)=0\right)=0$ holds.
        \item The first-order condition $\mathbb{E}_{P_\ast}[\nabla_{\beta} L(\langle \bm{X},\beta^\ast\rangle,Y)]=\bm{0}$ holds.
        \item The inequalities $\mathbb{E}_{P_\ast}\left[ \left\vert L^\prime(\langle\bm{X},\beta^\ast\rangle,Y)\right\vert\right]<\infty$, 
$\mathbb{E}_{P_\ast}\left[ \Vert\nabla_\beta L(\langle \bm{X},\beta^\ast\rangle,Y)\Vert^2_2\right]<\infty$,\\
$\mathbb{E}_{P_\ast}\left[\Vert\nabla_{\beta}L^\prime(\langle\bm{X},\beta^\ast\rangle,Y)\Vert_2\right]<\infty$, and $\mathbb{E}_{P_\ast}\left[\Vert \nabla_{\beta}^2L(\langle \bm{X},\beta^\ast\rangle,Y)\Vert_2\right]<\infty$ hold.
    \end{itemize}
\end{assumption}
In particular, we claim that conditions in Assumption \ref{assume1main} could be satisfied by the linear regression and the logistic regression in the following propositions.
\begin{proposition}\label{assumelinear}
    In linear regression,  if there does not exist nonzero vector $\alpha$ such that $P_\ast(\alpha ^\top \bm{X}=0)=1$, the ground-truth parameter $\beta^\ast\not=\bm{0}$ is an interior point of $B$, and inequalities $\mathbb{E}_{P_\ast}[\Vert \bm{X}\Vert^2_2\mathrm{Var}_{P_\ast}(Y\vert \bm{X})]<\infty$, $\mathbb{E}_{P_\ast}[\mathrm{Var}_{P_\ast}(Y\vert \bm{X})]<\infty$,  $\mathbb{E}_{P_\ast}[\Vert \bm{X}\Vert_2^2]<\infty$ hold, then Assumption \ref{assume1main} is satisfied.
\end{proposition}
\begin{proposition}\label{assumelog}
    In logistic regression,  if there does not exist nonzero vector $\alpha$ such that $P_\ast(\alpha ^\top \bm{X}=0)=1$, the ground-truth parameter $\beta^\ast\not=\bm{0}$ is an interior point of $B$, and $\mathbb{E}_{P_\ast}[\Vert \bm{X}\Vert^2_2]<\infty$, then Assumption \ref{assume1main} is satisfied.
\end{proposition}
Equipped with Assumption \ref{assume1main}, we could derive the following theorem, which depicts the asymptotic behavior of adversarial training estimator under different choices of the perturbation magnitude $\delta_n$.
\begin{theorem}[Asymptotic Behavior] \label{asyminfty}
Under Assumption \ref{assume1main}, for the adversarial training estimator $\beta^n$ defined in \eqref{advestimator}, if the perturbation magnitude is chosen as $\delta_n=\eta/n^{\gamma}$, $\eta,\gamma>0$,
we have the following convergence as $n\to\infty$: 
 \[ n^{\min\{1/2, \gamma\}}\left( \beta^n-\beta^\ast\right)\Rightarrow \arg\min_{\bm{u}} V(\bm{u}).\]
\begin{itemize}
\item If $\gamma>1/2$, then
\[V(\bm{u})=-\mathbf{G}^\top \bm{u}+\frac{1}{2}\bm{u}^\top H\bm{u};  \]
    \item If  $\gamma=1/2$, then
\begin{equation}\label{functionV}
\begin{aligned}V(\bm{u})=&(-\mathbf{G}+\eta\Vert\beta^\ast\Vert_1 K)^\top \bm{u}+\frac{1}{2}\bm{u}^\top H\bm{u} \\
&+\eta\left(\sum_{j=1}^d \left([\bm{u}]_j \sign([\beta^\ast]_j) \mathbb{I}([\beta^\ast]_j\not=0)+\vert [\bm{u}]_j\vert \mathbb{I}([\beta^\ast]_j=0)\right)\right) r;
\end{aligned}
\end{equation}
\item If  $0<\gamma<1/2$, then
\[\hspace{-2em}V(\bm{u})=\eta \Vert\beta^\ast\Vert_1  K^\top \bm{u}+\frac{1}{2}\bm{u}^\top H\bm{u} +\eta\left(\sum_{j=1}^d\left([\bm{u}]_j \sign([\beta^\ast]_j) \mathbb{I}([\beta^\ast]_j\not=0)+\vert [\bm{u}]_j\vert \mathbb{I}([\beta^\ast]_j=0)\right)\right) r; \]
\end{itemize}
where  $\mathbf{G}\sim \mathcal{N}(\bm{0},\Sigma)$ with covariance matrix $\Sigma=\Cov_{P_\ast}(\nabla_\beta L(\langle \bm{X},\beta^\ast\rangle,Y))$, $r=\mathbb{E}_{P_\ast}\left[ \left\vert L^\prime(\langle\bm{X},\beta^\ast\rangle,Y)\right\vert\right]$, $K=\mathbb{E}_{P_\ast}\left[\sign(L^\prime (\langle \bm{X},\beta^\ast\rangle,Y)) \nabla_\beta  L^\prime (\langle \bm{X},\beta^\ast\rangle,Y)\right]$, and $H=\mathbb{E}_{P_\ast}\left[ \nabla_{\beta}^2L(\langle \bm{X},\beta^\ast\rangle,Y)\right]$.
\end{theorem}

\begin{remark}
We focus on the generalized linear model and the associated loss resulting from the maximum likelihood estimation. 
However, our results are applicable to any case that satisfies Assumption \ref{assume1main}, as stated in Theorem \ref{asyminfty}.
\end{remark}

\begin{remark}\label{remarklp}
Theorem \ref{asyminfty} analyzes the asymptotic behavior of adversarial training estimator under $\ell_\infty$-perturbation. 
If we consider $\ell_p$-perturbation with $p\in(1,\infty)$, under certain regularity conditions, 
the resulting estimator $\beta^n$ with perturbation magnitude $\delta_n=\eta/n^{\gamma}$, $\eta,\gamma>0$ has the following asymptotic distribution as $n\to\infty$:    
 \begin{itemize}[itemsep=-2ex]
        \item If $\gamma>1/2$, then 
        \begin{equation*} \sqrt{n}(\beta^n-\beta^\ast)\Rightarrow H^{-1}\mathbf{G};\end{equation*}
        \item If $\gamma=1/2$, then 
        \begin{equation*}\begin{aligned}
             \sqrt{n}(\beta^n-\beta^\ast)
             \Rightarrow H^{-1}\left(-\eta \nabla_{\beta}\left( \Vert\beta^\ast\Vert_q\mathbb{E}_{P_\ast}[\vert L^\prime (\langle \bm{X},\beta^\ast\rangle,Y)\vert]\right)+\mathbf{G}\right);\end{aligned}\end{equation*}
        \item If $0<\gamma<1/2$, then 
        \begin{equation*}\begin{aligned} n^{\gamma}(\beta^n-\beta^\ast)
        \Rightarrow&-\eta H^{-1}\nabla_{\beta}\left( \Vert\beta^\ast\Vert_q\mathbb{E}_{P_\ast}[\vert L^\prime (\langle \bm{X},\beta\rangle,Y)\vert]\right),\end{aligned}\end{equation*}
    \end{itemize}
    where $1/p+1/q=1$. We provide the formal statement and its corresponding proof in Appendix \ref{pperturbation}.
\end{remark}

Theorem \ref{asyminfty} classifies the asymptotic behavior of adversarial training into three cases based on the order of the perturbation magnitude. Recall that the perturbation magnitude is defined by $\delta_n=\eta/n^\gamma$. 
If $\gamma>1/2$, the asymptotic behavior of the adversarial training estimator is identical to that of the empirical risk minimization estimator, indicating that the influence of hedging against adversarial attacks vanishes. 
If $\gamma=1/2$, the adversarial training estimator converges to a non-degenerate random variable at the rate of $\sqrt{n}$. The terms $\Vert\beta^\ast\Vert_1 K$, $r$ and $\sum_{j=1}^d( [u]_j \sign([\beta^\ast]_j) \mathbb{I}([\beta^\ast]_j\not=0)+\vert [u]_j\vert \mathbb{I}([\beta^\ast]_j=0))$ resulted by the regularization term $\delta \Vert\beta\Vert_1\mathbb{E}_{\mathbb{P}_n}\left[ \left\vert L^\prime(\langle\bm{X},\beta\rangle,Y)\right\vert\right]$ in the reformulation \eqref{regularizationeffect}. These terms capture the impact of achieving adversarial robustness on the asymptotic behavior of the resulting estimator.
If $0<\gamma<1/2$, the adversarial training estimator converges to a nonrandom quantity at a slower convergence rate $n^{\gamma}$. 

\subsection{Sparsity-recovery Ability}
It has been noted in the literature that the adversarial training under $\ell_\infty$-perturbation could help recover sparsity \citep{ribeiro2023regularization}.
The following proposition, based on the asymptotic behavior of the associated estimator $\beta^n$ in Theorem \ref{asyminfty}, provides a theoretical guarantee for the sparsity-recovery ability of the adversarial training via $\ell_\infty$-perturbation under the $1/\sqrt{n}$-order of perturbation magnitude.
\begin{proposition}[Sparsity-recovery Ability]\label{positiveprob}
   Under Assumption \ref{assume1main} and the setting $\delta_n=\eta/\sqrt{n}$, $\eta>0$, for the adversarial training estimator $\beta^n$ defined in \eqref{advestimator}, if we suppose $[\beta^\ast]_1,...,[\beta^\ast]_s\not=0$ and $[\beta^\ast]_{s+1}=...=[\beta^\ast]_d=0$, then the asymptotic distribution of $(\beta^n)_2$ have a positive probability mass at $\bm{0}$, where $(\beta^n)_2=([\beta^n]_{s+1},...,[\beta^n]_d)^\top$.
\end{proposition}
We discuss the implications of Proposition \ref{positiveprob} as follows.

We first discuss the superiority of $\ell_\infty$-perturbed adversarial training compared with other types of perturbation. 
Proposition \ref{positiveprob} indicates that the asymptotic distribution of the adversarial training estimator under $\ell_\infty$-perturbation can have a positive probability mass at $0$ if the underlying parameter equals $0$.
In contrast, as shown in Remark \ref{remarklp}, the adversarial training estimator under $\ell_p$-perturbation with $p\in(1,\infty)$ converges in distribution to the normal distribution when $\gamma\geq1/2$ and converges in distribution to a nonrandom quantity when $\gamma<1/2$.
These results demonstrate that the asymptotic distribution of the adversarial training estimator under $\ell_p$-perturbation with $p\in(1,\infty)$ can not have a positive probability mass at $0$, even if the corresponding true parameter equals to $0$. 
This indicates that the adversarial training under $\ell_\infty$-perturbation has a better theoretical performance for obtaining sparse solutions when the sparsity of the parameter is known.

 We then discuss the choice of the order of the perturbation magnitude. 
 As shown in Theorem \ref{asyminfty}, under $\ell_\infty$-perturbation, the asymptotic distribution of the adversarial training estimator is normal when $\gamma>1/2$, and the adversarial training estimator converges in distribution to a nonrandom quantity when $0<\gamma<1/2$. 
Therefore,  the asymptotic distribution of the adversarial training estimator with the setting $\gamma>1/2$ and $0<\gamma<1/2$ can not have positive probability at $0$ when the underlying parameter equals $0$. In other words, only the choice $\gamma=1/2$, i.e., $\delta_n=\eta/\sqrt{n}$,  could promise the sparsity-recovery ability of $\ell_\infty$-perturbed adversarial training.
Also, although the root-$n$ consistency could be achieved with both $\gamma = 1/2$ and $\gamma > 1/2$, the sparsity-recovery advantage can only be realized when $\gamma=1/2$.
Hence,  $\gamma=1/2$ is the optimal choice for sparsity recovery.

\section{Adaptive Adversarial Training}\label{adaptivesection}
In this section, we propose a novel two-step approach, called adaptive adversarial training, and analyze its superior asymptotic statistical properties compared with classic adversarial training procedure under $\ell_\infty$-perturbation.

We focus on the generalized linear model, and our new approach has the following two steps:
\begin{itemize}
    \item \textbf{Step 1:} Solve the problem 
    \begin{equation*}\widehat{\beta}^n\in\arg\min_{\beta\in B} \mathbb{E}_{\mathbb{P}_n}\left[ L(\langle\bm{X},\beta\rangle,Y)\right],\end{equation*}
    where $B$ is a convex compact subset of $\mathbb{R}^d$.
    \item \textbf{Step 2:} Solve the problem 
    \begin{equation}\label{step2} \widetilde{\beta}^n\in\arg\min_{\beta\in B} \mathbb{E}_{\mathbb{P}_n}\left[\max _{\Vert \widehat{\beta}^n\otimes\Delta\Vert_{\infty}\leq \delta_n}L(\langle\bm{X}+\Delta,\beta\rangle,Y)\right],\end{equation}
     where $\delta_n=\eta/n^\gamma$,  $\eta,\gamma>0$, and output the estimator $\widetilde{\beta}^n$.
     Recall $\widehat{\beta}^n\otimes\Delta=([\widehat{\beta}^n]_1[\Delta]_1,...,[\widehat{\beta}^n]_d[\Delta]_d)$.
\end{itemize}

In other words, the weight added to the perturbation $\Delta$ is chosen as $\widehat{\beta}^n$. 
The $i$th weight component is chosen as the $i$th component of the classic empirical risk minimization estimator, i.e., $[\widehat{\beta}^n]_i$.
In other words, the proposed adaptive technique imposes more perturbation along the components where the underlying ground truth is close to $0$, making the associated adaptive estimator approach $0$.
It will be shown later that this adjustment for $\ell_\infty$-perturbed adversarial training leads to desirable statistical properties.

To better understand our proposed adaptive adversarial training, we investigate the regularization effect of the adaptive adversarial training in the following proposition.

\begin{proposition}[Regularization Effect of Adaptive Technique]\label{prop2}
 If the function $h:\mathbb{R}^d\to\mathbb{R}$ is differentiable, $\nabla h$ is uniformly continuous, and $\mathbb{E}_{\bm{Z}\sim P}\left[ \Vert \bm{w}^{-1}\otimes\nabla h(\bm{Z})\Vert_\ast\right]<\infty$ with $\bm{\omega}\in\mathbb{R}^d$ satisfying $[\bm{\omega}]_i\not = 0, 1\leq i\leq d$, then we have that
\[\mathbb{E}_{\bm{Z}\sim P}\left[\max _{\Vert \bm{w}\otimes\Delta\Vert\leq \delta}h(\bm{Z}+\Delta)\right]=\mathbb{E}_ {\bm{Z}\sim P}\left[ h(\bm{Z})\right]+\delta \mathbb{E}_{\bm{Z}\sim P}\left[ \Vert \bm{w}^{-1}\otimes\nabla h(\bm{Z})\Vert_\ast\right]+o(\delta)\]
as $\delta\to 0$.
\end{proposition}
Proposition \ref{prop2} implies that the influence of adding weight in adaptive adversarial training is reflected in the regularization term $\delta \mathbb{E}_{\bm{Z}\sim P}\left[ \Vert \bm{w}^{-1}\otimes\nabla h(\bm{Z})\Vert_\ast\right]$. 
We can immediately derive the following corollary from Proposition \ref{prop2}, which characterizes the scenario of the generalized linear model.
\begin{corollary}\label{lemmaadaptive}
If the function $L(f,y)$ is differentiable w.r.t. the first argument, $L^{\prime}(f,y)$ is uniformly continuous w.r.t. the first argument, for $\bm{\omega}\in\mathbb{R}^d$ satisfying $[\bm{\omega}]_i\not = 0, 1\leq i\leq d$, we have that
\begin{equation}\label{adaptivereg}\begin{aligned}&\mathbb{E}_{\mathbb{P}_n}\left[\max _{\Vert\bm{w}\otimes\Delta\Vert_{\infty}\leq \delta}L(\langle\bm{X}+\Delta,\beta\rangle,Y)\right]\\
&=\mathbb{E}_{\mathbb{P}_n}\left[ L(\langle\bm{X},\beta\rangle,Y)\right]+\delta \Vert\bm{w}^{-1}\otimes\beta\Vert_1\mathbb{E}_{\mathbb{P}_n}\left[ \left\vert L^\prime(\langle\bm{X},\beta\rangle,Y)\right\vert\right]+o(\delta),\end{aligned}\end{equation}
holds  as $\delta\to 0$ for $\beta\in B$, where $B$ is some compact subset of $\mathbb{R}^d$.
\end{corollary}
Corollary \ref{lemmaadaptive} indicates that having weight on the $\ell_\infty$-perturbation is asymptotically equivalent to having weight on the $\ell_1$-penalty in the regularization term. Then, we can obtain the desirable properties of the proposed adaptive adversarial training in the following subsection.
\subsection{Statistical Properties of 
Adaptive Adversarial Training}
This subsection discusses the properties of adaptive adversarial training and compares it with classic adversarial training. 

We first clarify some notations.
We let $\mathcal{A}=\{j\vert [\beta^\ast]_j\not=0\}$
denote the index set of non-zero components of the ground-truth parameter $\beta^\ast\in\mathbb{R}^d$. 
Without loss of generality, we also assume that $\mathcal{A}=\{1,...,s\}$. In this regard, for any $\mathbb{R}^{d\times d}$ matrix $M$, we let $M=\left(\begin{array}{cc}
    M_{11}&M_{12}  \\
    M_{21}&M_{22}
\end{array}\right),$
where $M_{11}\in\mathbb{R}^{s\times s}$, $M_{22}\in\mathbb{R}^{(d-s)\times (d-s)}$, $M_{12}\in\mathbb{R}^{r\times (d-s)}$,
and $M_{21}\in\mathbb{R}^{(d-s)\times s}$.

We discuss the asymptotic behavior of the adaptive adversarial training estimator in the following theorem.
\begin{theorem}[Asymptotic Behavior of Adaptive Adversarial Training Estimator]\label{unbiasedness}
Under Assumption \ref{assume1main}, if we choose the magnitude of perturbation as $\delta=\eta/n^{\gamma}$, $\eta>0, 0<\gamma<1$, the estimator defined in \eqref{step2} has the following convergence as $n\to\infty$:
 \begin{itemize}
 \item If $1/2<\gamma<1$, then
 \[ \sqrt{n}\left( [\widetilde{\beta}^n]_{\mathcal{A}}-[\beta^\ast]_{\mathcal{A}}\right)\Rightarrow H_{11}^{-1}[\mathbf{G}]_{\mathcal{A}},\quad \sqrt{n} [\widetilde{\beta}^n]_{\mathcal{A}^c}\to_p \bm{0};\]\item If $\gamma=1/2$, then
\[\sqrt{n}\left( [\widetilde{\beta}^n]_{\mathcal{A}}-[\beta^\ast]_{\mathcal{A}}\right)\Rightarrow H_{11}^{-1}(-\eta r[{\beta^\ast}^{-1}]_{\mathcal{A}}-\eta s[K]_{\mathcal{A}}+[\mathbf{G}]_{\mathcal{A}}),\quad \sqrt{n} [\widetilde{\beta}^n]_{\mathcal{A}^c}\to_p \bm{0};\]
 \item If $0<\gamma<1/2$, then 
  \[n^{\gamma}\left( [\widetilde{\beta}^n]_{\mathcal{A}}-[\beta^\ast]_{\mathcal{A}}\right)\Rightarrow H_{11}^{-1}(-\eta r[{\beta^\ast}^{-1}]_{\mathcal{A}}-\eta s[K]_{\mathcal{A}}),\quad n^{\gamma} [\widetilde{\beta}^n]_{\mathcal{A}^c}\to_p \bm{0},\]
 \end{itemize}
 where $[\mathbf{G}]_{\mathcal{A}}\sim \mathcal{N}(\bm{0},\Sigma_{11})$,  $\Sigma=\Cov_{P_\ast}(\nabla_\beta L(\langle \bm{X},\beta^\ast\rangle,Y))$,
$H=\mathbb{E}_{P_\ast}\left[ \nabla_{\beta}^2 L(\langle \bm{X},\beta^\ast\rangle,Y)\right]$,
$K=\mathbb{E}_{P_\ast}\left[\sign(L^\prime (\langle \bm{X},\beta^\ast\rangle,Y)) \nabla_\beta  L^\prime (\langle \bm{X},\beta^\ast\rangle,Y)\right]$,
$r=\mathbb{E}_{P_\ast}\left[ \left\vert L^\prime(\langle\bm{X},\beta^\ast\rangle,Y)\right\vert\right]$, and $s$ denotes the number of nonzero components in $\beta^\ast$, i.e., the cardinality of $\mathcal{A}$.
\end{theorem}
\begin{remark}
   We require $\gamma<1$ to avoid the indeterminate form of limits in Theorem \ref{unbiasedness}.
\end{remark}
\begin{remark}
In the adaptive adversarial training framework introduced in \eqref{step2}, we set the weight vector to be the empirical risk minimizer. However, as implied in the proof of Theorem \ref{unbiasedness}, any sequence of weight vectors that converges in probability to $\beta^\ast$ is sufficient.
\end{remark}
\begin{remark}
To avoid degenerate cases, we assume that each component of the weight vector $\widehat{\beta}^n$—the empirical risk minimizer in the adaptive adversarial training step \eqref{step2}—is nonzero. If any component happens to be zero, a small offset, such as $1/n$, can be added to ensure non-degeneracy. Since $1/n \to 0$ as $n \to \infty$, this adjustment does not affect the final asymptotics.
\end{remark}
Based on the asymptotic distribution obtained in Theorem \ref{unbiasedness}, it can be observed that the adaptive adversarial training procedure achieves asymptotic variable-selection consistency, while the classic adversarial training estimator does not. The following proposition states the variable-selection inconsistency of the adversarial training estimator.

\begin{proposition}[Variable-selection Inconsistency of Adversarial Training]\label{inconsistency}
Suppose Assumption \ref{assume1main} holds.
For the adversarial training estimator $\beta^n$ defined in \eqref{advestimator}, if the perturbation magnitude is chosen as $\delta_n=\eta/n^{\gamma}$, $\eta, \gamma>0$, we have the following
   \[\lim\inf_n P_\ast(\mathcal{A}_n=\mathcal{A})\leq c<1,\]
   where $c$ is some constant,
 and $\mathcal{A}_n=\{j\vert [\beta^n]_j\not=0\}$.
\end{proposition}

Recall that we have shown in Proposition \ref{positiveprob} that the asymptotic distribution of the adversarial training estimator under $\ell_\infty$-perturbation could put a positive probability to $0$, which we claim as the theoretical evidence for the sparsity-recovery ability.
However, Proposition \ref{inconsistency} demonstrates that the index set of nonzero components identified by the adversarial training procedure under $\ell_\infty$-perturbation is wrong with a positive probability. 
Thus, the classic adversarial training procedure has limitations in asymptotic variable-selection consistency.
In contrast, adaptive adversarial training could achieve asymptotic variable-selection consistency, seeing the following theorem.

\begin{theorem}[Variable-selection Consistency of Adaptive Adversarial Training]\label{consistency}
Suppose Assumption \ref{assume1main} holds.
For the adaptive adversarial training estimator $\widetilde{\beta}^n$ defined in \eqref{step2}, if the perturbation magnitude is chosen as $\delta_n=\eta/n^{\gamma}$ and $ 0<\gamma<1$, we have the following
     \[\lim_n P_\ast(\widetilde{\mathcal{A}}_n=\mathcal{A})=1,\]
     where
     $\widetilde{\mathcal{A}}_n=\{j\vert [\widetilde{\beta}^n]_j\not=0\}$.\end{theorem}

In addition to achieving variable-selection consistency, the proposed adaptive adjustment technique also ensures unbiasedness. Theorem \ref{asyminfty} shows that the asymptotic distribution of the adversarial training estimator with the preferred setting $\gamma=1/2$ has a bias, indicating that this estimator may not be reliable for parameter estimation. In contrast, Theorem \ref{unbiasedness} demonstrates that the proposed adaptive adversarial training estimator can be asymptotically unbiased when $1/2<\gamma<1$, without compromising asymptotic variable-selection consistency.

To explore more properties of the adaptive adversarial training, we take the linear regression as an example in the following corollary.

\begin{corollary}[Asymptotic Behavior of Adaptive Adversarial Training in Linear Regression]\label{asymcorolaarylinear}
In the linear regression where $L=(\langle\bm{x},\beta\rangle-y)^2/2$, we assume that  the distribution of response variable $Y$ conditional on input variable $\bm{X}$ is normal, i.e., $Y|\bm{X}=\bm{x}\sim\mathcal{N}(\langle \bm{x},\beta^\ast\rangle,\sigma^2)$, $0<\sigma<\infty$, and $\bm{X}\sim\mathcal{N}(\bm{0},\bm{I}_d)$.
    If the perturbation magnitude is chosen as $\delta_n=\eta/n^{\gamma}$, $\eta,\gamma>0$,
    for the adversarial training estimator $\beta^n$ defined in \eqref{step2},
    we have the following convergence as $n\to\infty$:
    \begin{itemize}
 \item If $1/2<\gamma<1$, then
 \[ \sqrt{n}\left( [\widetilde{\beta}^n]_{\mathcal{A}}-[\beta^\ast]_{\mathcal{A}}\right)\Rightarrow [\mathbf{G}]_{\mathcal{A}},\quad \sqrt{n} [\widetilde{\beta}^n]_{\mathcal{A}^c}\to_p \bm{0};\]
 \item If $\gamma=1/2$, then
\[\sqrt{n}\left( [\widetilde{\beta}^n]_{\mathcal{A}}-[\beta^\ast]_{\mathcal{A}}\right)\Rightarrow -\eta \sigma\sqrt{\frac{2}{\pi}}[{\beta^\ast}^{-1}]_{\mathcal{A}}+[\mathbf{G}]_{\mathcal{A}},\quad \sqrt{n} [\widetilde{\beta}^n]_{\mathcal{A}^c}\to_p \bm{0};\]
 \item If $0<\gamma<1/2$, then 
  \[n^{\gamma}\left( [\widetilde{\beta}^n]_{\mathcal{A}}-[\beta^\ast]_{\mathcal{A}}\right)\Rightarrow -\eta \sigma\sqrt{\frac{2}{\pi}}[{\beta^\ast}^{-1}]_{\mathcal{A}},\quad n^{\gamma} [\widetilde{\beta}^n]_{\mathcal{A}^c}\to_p \bm{0},\]
 \end{itemize}
where $\mathbf{G}\sim \mathcal{N}(\bm{0}, \sigma^2\bm{I}_d)$.
\end{corollary}
Under the settings in Corollary \ref{asymcorolaarylinear} and with the perturbation magnitude order $1/2<\gamma<1$, the variance of the asymptotic distribution of the adaptive adversarial training is $\sigma^2 s$, where $s$ is the number of nonzero components in $\beta^\ast$. In contrast, for the empirical risk minimization estimator, the variance of the associated  asymptotic distribution is $\sigma^2 d$, where $d$ is the dimension of the variable. Thus, when there are zero components in $\beta^\ast$, the adaptive adversarial training can achieve variance reduction compared to empirical risk minimization.
This indicates a reduced asymptotic mean squared error, as both the empirical risk minimization estimator and adaptive adversarial training under the aforementioned settings are asymptotically unbiased.

In summary, the proposed two-step procedure improves the statistical properties of adversarial training in the following ways: it achieves asymptotic variable-selection consistency, as discussed in Theorem \ref{consistency}, realizes asymptotic unbiasedness, as shown in Theorem \ref{unbiasedness}, and reduces asymptotic variance compared to empirical risk minimization under certain settings, as demonstrated in Corollary \ref{asymcorolaarylinear}.

\section{Numerical Experiments}\label{numsection}
In this section, we investigate the numerical performance of $\ell_\infty$-perturbed adversarial training and the proposed adaptive adversarial training. 
We will first give details of calculating the estimators obtained by the aforementioned two procedures in Section \ref{tractablereformulate}, and then introduce our experimental settings and results in Section \ref{experiments} and Section \ref{experimentsreal}.
\subsection{Tractable Reformulation}\label{tractablereformulate}
A major challenge of implementing adversarial training is to solve the inner maximization problem inside the empirical expectation \eqref{intro}.
To overcome this challenge, \citet{ribeiro2023regularization} provides the tractable reformulation when the loss function is based on the linear prediction. We restate the result in the following theorem. 
\begin{theorem}[Theorem 4 in \citet{ribeiro2023regularization}]\label{tractable}
   If $L(f,y)$ is convex and lower-semicontinuous w.r.t the first argument, we have that
    \[\max_{\Vert\Delta\Vert\leq \delta}L(\langle \bm{x}+\Delta,\beta\rangle,y)=\max_{s\in\{-1,1\}}L\left(\langle \bm{x},\beta\rangle+\delta s\Vert \beta\Vert_{\ast},y\right).\]
\end{theorem}
Based on Theorem \ref{tractable}, we develop the tractable reformulation for the adaptive adversarial training in the following corollary.
\begin{corollary}\label{tractable2}
   If $L(f,y)$ is convex and lower-semicontinuous w.r.t the first argument, for $\bm{\omega}\in\mathbb{R}^d$ satisfying $[\bm{\omega}]_i\not = 0, 1\leq i\leq d$, we have that
    \[\max_{\Vert\bm{w}\otimes\Delta\Vert\leq \delta}L(\langle \bm{x}+\Delta,\beta\rangle,y)=\max_{s\in\{-1,1\}}L\left(\langle \bm{x},\beta\rangle+\delta s\Vert \bm{w}^{-1}\otimes\beta\Vert_{\ast},y\right).\]
\end{corollary}
We mainly focus on linear regression and logistic regression in our experiments. Accordingly, we could immediately obtain the reformulations of the adversarial training and the adaptive adversarial training under $\ell_\infty$-perturbation for the linear regression and logistic regression, shown in the following example.
\begin{example}\label{example}
    Consider the adversarial training worst-case loss in the linear regression and logistic regression, we have that
    \[\max_{\Vert\Delta\Vert_\infty\leq \delta}(\langle \bm{x}+\Delta,\beta\rangle-y)^2=\left(\vert\langle \bm{x},\beta\rangle-y\vert+\delta \Vert \beta\Vert_1\right)^2.\]
    \[\max_{\Vert\Delta\Vert_\infty\leq \delta}\log \left(1+\exp(-y\langle \bm{x}+\Delta,\beta\rangle)\right)=\log \left(1+\exp(-y\langle \bm{x},\beta\rangle+\delta\Vert\beta\Vert_1)\right).\]
Consider the adaptive adversarial training worst-case loss in the linear regression and logistic regression, for $\bm{\omega}\in\mathbb{R}^d$ satisfying $[\bm{\omega}]_i\not = 0, 1\leq i\leq d$, we have that
        \[\max_{\Vert\bm{w}\otimes\Delta\Vert_\infty\leq \delta}(\langle \bm{x}+\Delta,\beta\rangle-y)^2=\left(\vert\langle \bm{x},\beta\rangle-y\vert+\delta \Vert \bm{w}^{-1}\otimes\beta\Vert_1\right)^2.\]
    \[\max_{\Vert\bm{w}\otimes\Delta\Vert_\infty\leq \delta}\log \left(1+\exp(-y\langle \bm{x}+\Delta,\beta\rangle)\right)=\log \left(1+\exp(-y\langle \bm{x},\beta\rangle+\delta\Vert\bm{w}^{-1}\otimes\beta\Vert_1)\right).\]
\end{example}
We implement the adversarial training and the adaptive adversarial training using the reformulations in Example \ref{example}.
More implementation details are elaborated in the next subsections.
\subsection{Synthetic-data Numerical Experiments}\label{experiments}
In this subsection, we describe how to carry out our numerical experiments on synthetic data and show our numerical results.
\subsubsection{Experimental Setting}\label{expsetting}
We consider the linear regression and the logistic regression. 

For the linear regression, we assume the input variable $\bm{X}$ follows the 10-dimensional standard normal distribution, and the response variable $Y$ follows the normal distribution, where we set $Y|\bm{X}=\bm{x}\sim \mathcal{N}(\langle\bm{x},\beta^\ast\rangle,0.01)$ and  $\beta^\ast=(1,0.8,-0.4,0,0,0,0,0,0,0)$. The first three components of the ground-truth parameter $\beta^\ast$ are nonzero while others are zero.

For the logistic regression, we assume the input variable $\bm{X}$ follows the 10-dimensional standard normal distribution, and the response variable $Y$ follows the Bernoulli distribution. More specifically, we have  $P_\ast(Y=1|\bm{X}=\bm{x})=1/(1+e^{-\langle \bm{x},\beta^{\ast}\rangle})$ and $\beta^\ast=(0.3,0.1,0,0,$ $0,0,0,0,0,0)$, where the first two components are nonzero where others are zero.

We conduct adversarial training and adaptive adversarial training for both linear regression and logistic regression using different sample sizes. Given a sample size of \( n \), we select the first \( n \) samples from the synthetic data generated by the respective regression models.
The asymptotic behavior of the adversarial training estimator varies depending on the order of the perturbation magnitude, which is parameterized as \( \delta = \eta / n^\gamma \). To explore these effects, we set \( \gamma = 1/3, 1/2, 1 \) and use \( \eta = 1 \) as the default value. 
For adaptive adversarial training, which achieves asymptotic variable-selection consistency and unbiasedness when \( 1/2 < \gamma < 1 \), we specifically choose \( \gamma = 2/3 \).

\subsubsection{Experimental Results}
\begin{figure}[htbp]
    \centering
    \begin{minipage}{0.5\textwidth}
        \centering
        \includegraphics[width=\textwidth]{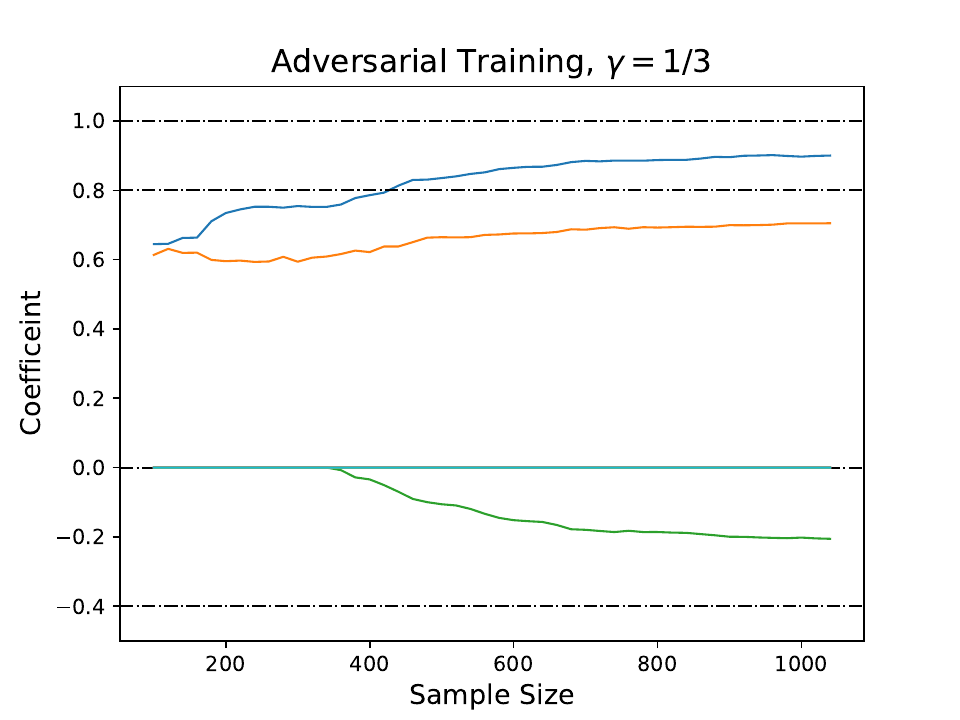} 
    \end{minipage}\hfill
    \begin{minipage}{0.5\textwidth}
        \centering
        \includegraphics[width=\textwidth]{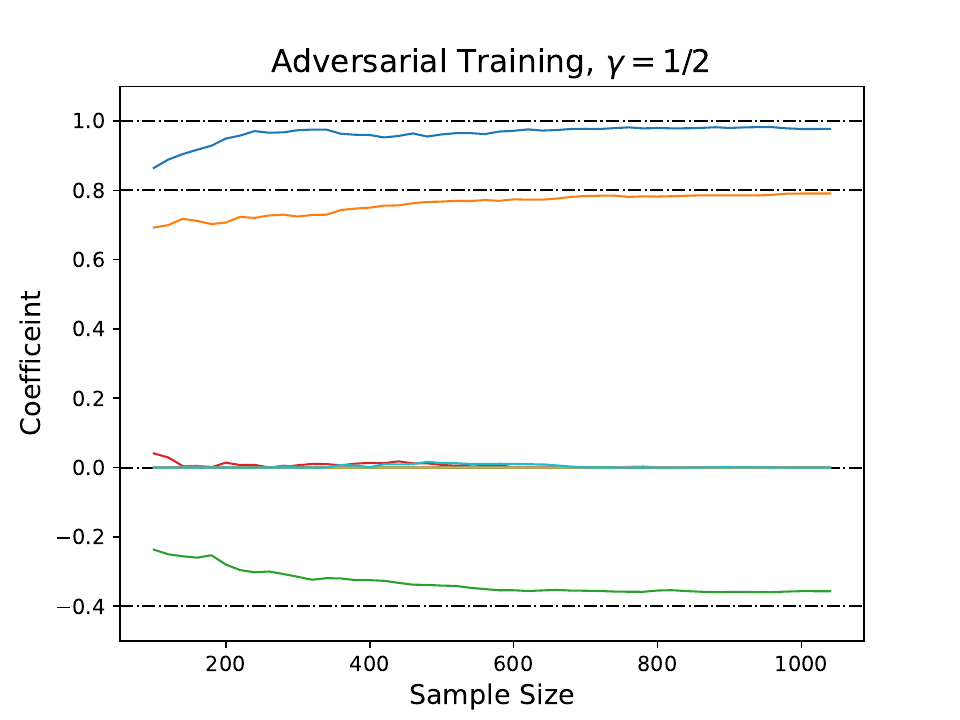} 
    \end{minipage}
        \begin{minipage}{0.5\textwidth}
        \centering
        \includegraphics[width=\textwidth]{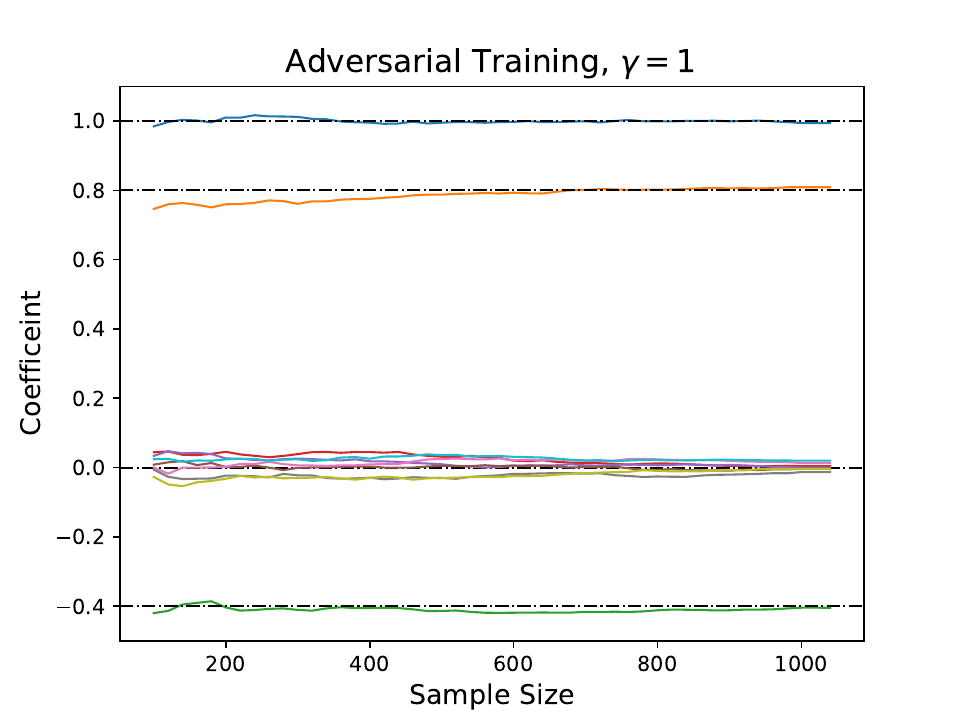} 
    \end{minipage}\hfill
            \begin{minipage}{0.5\textwidth}
        \centering
        \includegraphics[width=\textwidth]{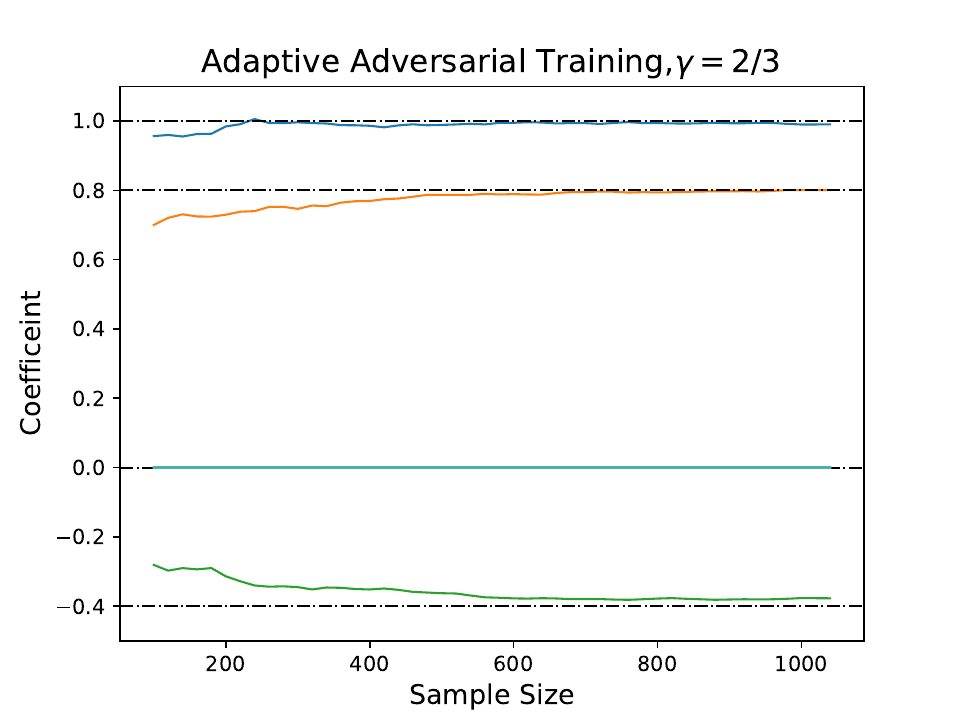} 
    \end{minipage}
    \caption{Coefficient Path in the Linear Regression}
    \label{linearfig}
\end{figure}

\begin{table}[htbp]
\begin{center}
\begin{tabular}{c|c c c c c c} 
 \hline
  & $n=100$ & $n=200$& $n=400$ & $n=600$ & $n=800$ & $n=1000$\\ [0.5ex] 
\hline\hline
 AT, $\gamma=1/3$ & 0.5665& 0.5218& 0.4597& 0.3087& 0.2645& 0.2424\\ 

 AT, $\gamma=1/2$ &0.2411 & 0.1607& 0.0996& 0.0606& 0.0525& 0.0497\\

 AT, $\gamma=1$ & $\bm{0.0895}$ & $\bm{0.0816}$& 0.0739& 0.0564& 0.0466& 0.0304\\

 Adaptive AT, $\gamma=2/3$ &0.1612& 0.1123& $\bm{0.0587}$& $\bm{0.0254}$& $\bm{0.0234}$& $\bm{0.0254}$\\ [1ex]
 \hline
\end{tabular}
\end{center}
\caption{Estimation Error $\Vert\widetilde{\beta}_n-\beta^\ast\Vert_2$ in the Linear Regression}
\label{lineartable}
\end{table}

\begin{figure}[ht]
    \centering
    \begin{minipage}{0.5\textwidth}
        \centering
        \includegraphics[width=\textwidth]{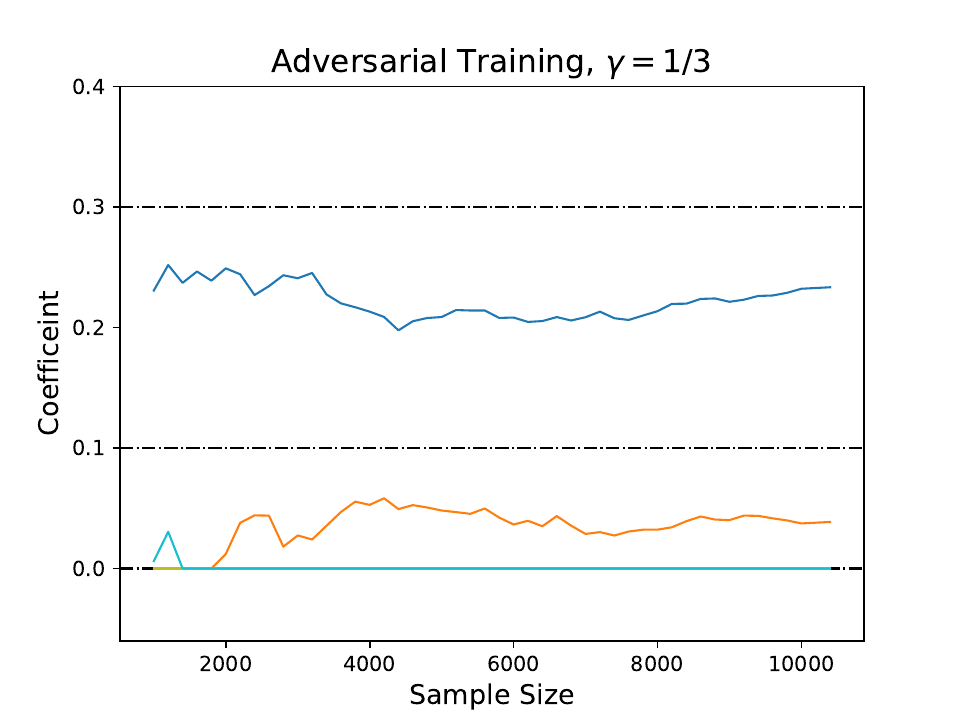} 
    \end{minipage}\hfill
    \begin{minipage}{0.5\textwidth}
        \centering
        \includegraphics[width=\textwidth]{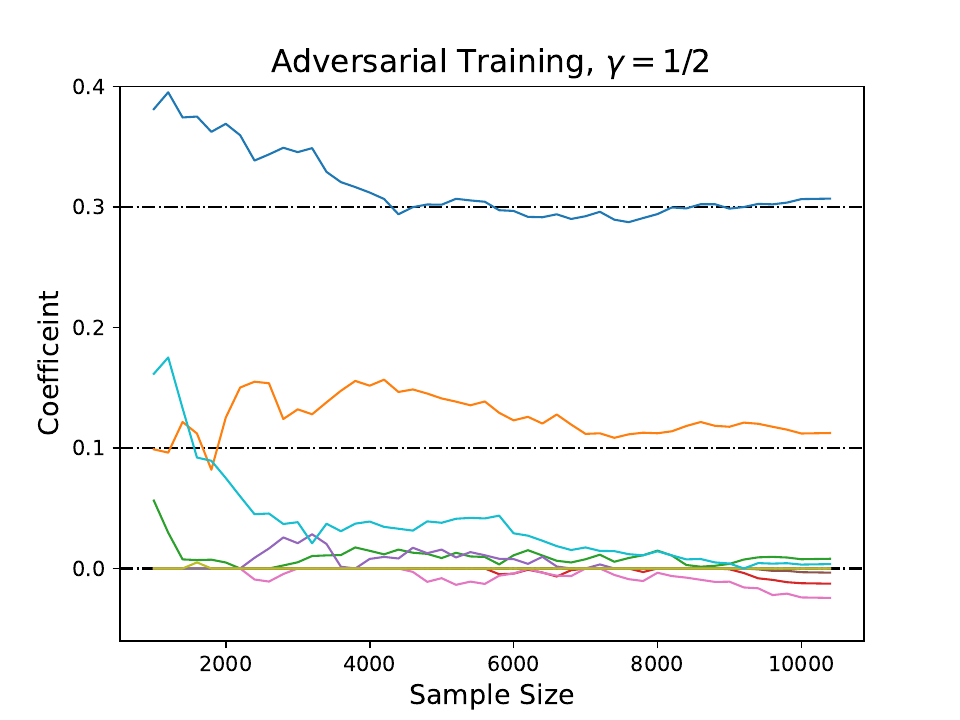} 
    \end{minipage}
        \begin{minipage}{0.5\textwidth}
        \centering
        \includegraphics[width=\textwidth]{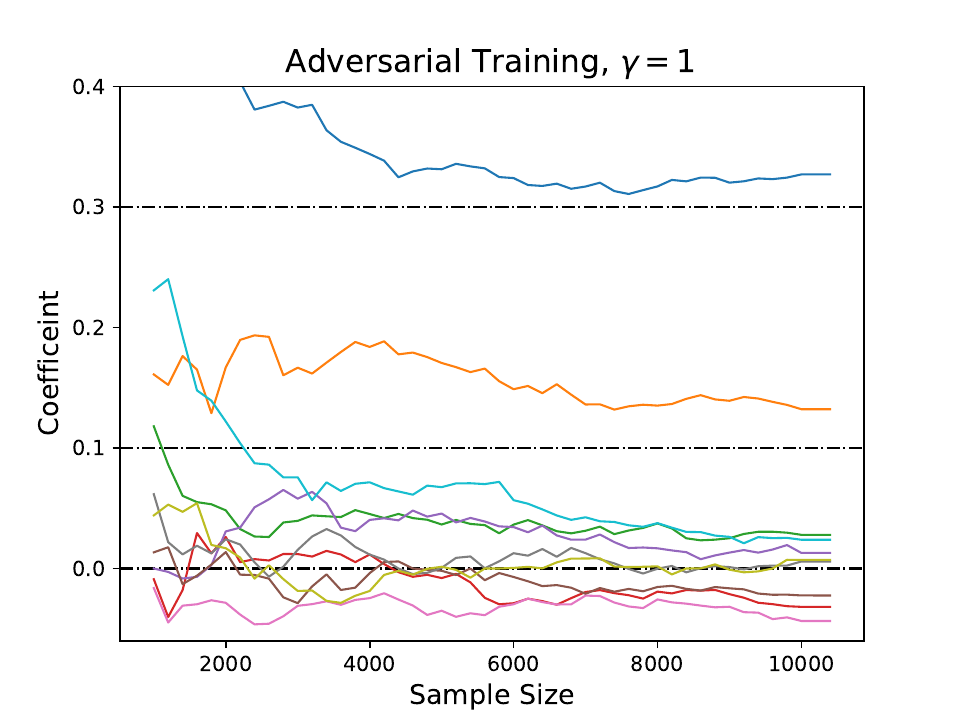} 
    \end{minipage}\hfill
            \begin{minipage}{0.5\textwidth}
        \centering
        \includegraphics[width=\textwidth]{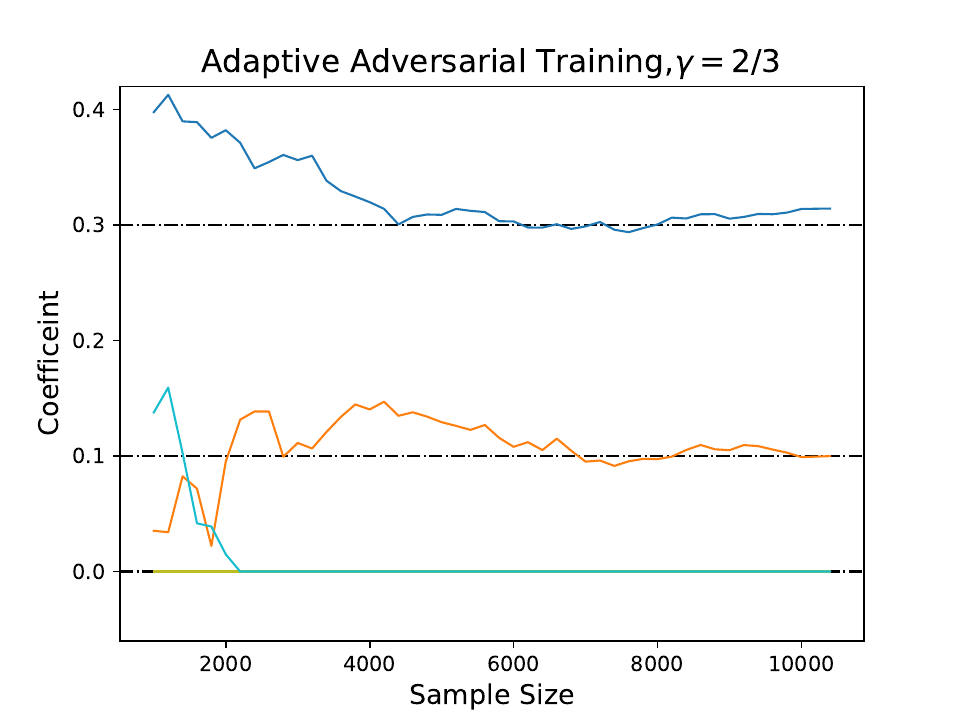} 
    \end{minipage}
    \caption{Coefficient Path in the Logistic Regression}
    \label{logfig}
\end{figure}

\begin{table}[htbp]
\begin{center}
\begin{tabular}{c|c c c c c c} 
 \hline
  & $n=1000$ & $n=2000$& $n=4000$ & $n=6000$ & $n=8000$ & $n=10000$\\ [0.5ex] 
\hline\hline
 AT, $\gamma=1/3$ & $\bm{0.1219}$&0.1017& 0.0989& 0.1116& 0.1099& 0.0924\\ 

 AT, $\gamma=1/2$ & 0.1893& 0.1051 &0.0679& 0.0401& 0.0249& 0.0315\\

 AT, $\gamma=1$ & 0.3156& 0.1978& 0.1377& 0.1029&0.0767&0.0823 \\

 Adaptive AT, $\gamma=2/3$&0.1809& $\bm{0.0836} $& $\bm{0.0449}$& $\bm{0.0085}$& $\bm{0.0028}$& $\bm{0.0139}$\\ [1ex]
 \hline
\end{tabular}

\end{center}
\caption{Estimation Error $\Vert\widetilde{\beta}_n-\beta^\ast\Vert_2$ in the Logistic Regression}
\label{logtable}
\end{table}

The experimental results are shown in Figure \ref{linearfig}, Figure \ref{logfig}, Table   \ref{lineartable}, and Table   \ref{logtable}.
The coefficient paths are presented in Figure \ref{linearfig} and Figure \ref{logfig}, showing the ability of sparsity-recovery and variable-selection under different settings.
The estimation error is recorded in Table   \ref{lineartable} and Table  \ref{logtable}, helping us compare the overall estimation accuracy for different procedures and settings. In the tables, we use ``AT'' to denote ``adversarial training'' to save space.

We analyze the numerical results as follows.
In Figures \ref{linearfig} and \ref{logfig}, the coefficient paths have different patterns depending on the choice of \( \gamma \). In adversarial training, when \( \gamma = 1/3 \), the coefficient paths indicate an estimation bias, suggesting that the estimations in this setting are not very accurate. When \( \gamma = 1 \), adversarial training provides desirable estimations for the nonzero components but fails to shrink the estimations for the zero components. When \( \gamma = 1/2 \), adversarial training not only shrinks the zero components effectively but also yields accurate estimations for the nonzero components. These observations align with the results in Theorem \ref{unbiasedness}. Thus, \( \gamma = 1/2 \) is the optimal choice, particularly for sparsity recovery in adversarial training.
In adaptive adversarial training, we achieve accurate estimations for both zero and nonzero components. Compared to the optimal scenario (\( \gamma = 1/2 \)) in standard adversarial training, Figures \ref{linearfig} and \ref{logfig} show that the estimation performances for nonzero components are similar, while adaptive adversarial training exhibits superior estimation performance for the zero components. Moreover, Tables \ref{lineartable} and \ref{logtable} indicate that the estimation error in adaptive adversarial training is the smallest in most cases, which are highlighted in bold.
In conclusion, adaptive adversarial training demonstrates a better ability for sparsity recovery and produces more accurate estimations. These observations are consistent with the results in Section \ref{adaptivesection}.

\subsection{Real-data Numerical Experiments}\label{experimentsreal}
In this subsection, we describe how to conduct empirical experiments on real-world datasets and show our numerical results.
\subsubsection{Experiment Setting}
We conduct adversarial training and adaptive adversarial training on two datasets: linear regression on the California housing dataset \citep{pedregosa2011scikit} and logistic regression on the Pulsar candidates dataset \citep{misc_htru2_372}.
California Housing dataset is loaded from scikit-learn datasets. Pulsar candidates dataset is loaded from UCI Machine Learning Repository. The California housing dataset contains 20640 samples with 8 numerical features, such as median income, house age, and average number of rooms. The response variable is the median house value in units of 100 thousand dollars. The Pulsar candidates dataset consists of 17898 samples with 8 numerical features extracted from radio signal observations, such as the mean, standard deviation, skewness, and kurtosis. The response variable is a binary label, where 0 represents non-pulsars and 1 represents pulsars.
We preprocess both datasets.
First, we standardize the input variables. We then split the data into a training set and a testing set, with the training set comprising 80\% of the instances. The adversarial training procedures are applied to the training set, and the estimation performance is evaluated on the testing set.
Per the settings of the adversarial training, we follow the same steps outlined in Section \ref{expsetting}.
\subsubsection{Experiment Results}
The numerical results of the California housing dataset and the Pulsar candidates dataset are reported in Figure \ref{realdatfigure}, Table \ref{realtable}, and Figure \ref{realdatfigure2}, Table \ref{realtable2}, respectively. Figure \ref{realdatfigure} and Figure  \ref{realdatfigure2} show the coefficient paths. Table  \ref{realtable} and Table \ref{realtable2} record the average of the squared prediction error, 
where the prediction is obtained by applying the coefficients estimated from the training samples to the testing samples.

Notice that both the California housing dataset and the Pulsar candidates dataset have 8 input variables.  
In Figures \ref{realdatfigure} and \ref{realdatfigure2}, the number of coefficient paths is fewer than 8, suggesting that the (adaptive) adversarial training procedures exhibit empirical sparsity-recovery abilities in real-world datasets.  
Additionally, Tables \ref{realtable} and \ref{realtable2} show that adaptive adversarial training generally achieves lower prediction errors on the testing samples compared to classic adversarial training.  
These observations highlight the sparsity-recovery ability of adversarial training and underscore the superiority of adaptive adversarial training, aligning with our theoretical results.  

\begin{figure}[htbp]
    \centering
    \begin{minipage}{0.5\textwidth}
        \centering
        \includegraphics[width=\textwidth]{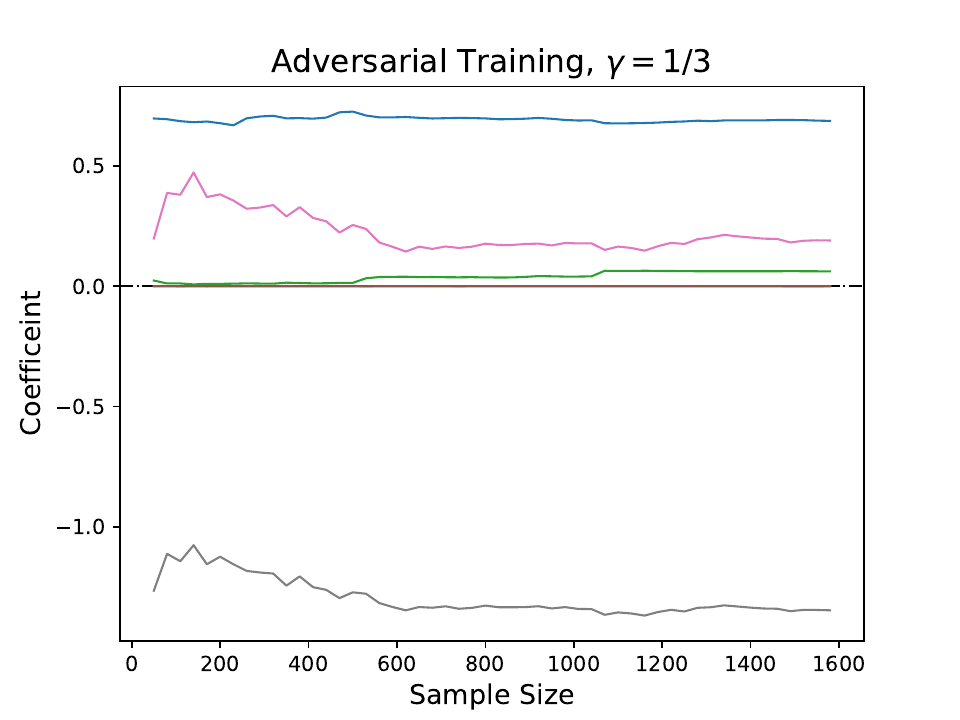} 
    \end{minipage}\hfill
    \begin{minipage}{0.5\textwidth}
        \centering
        \includegraphics[width=\textwidth]{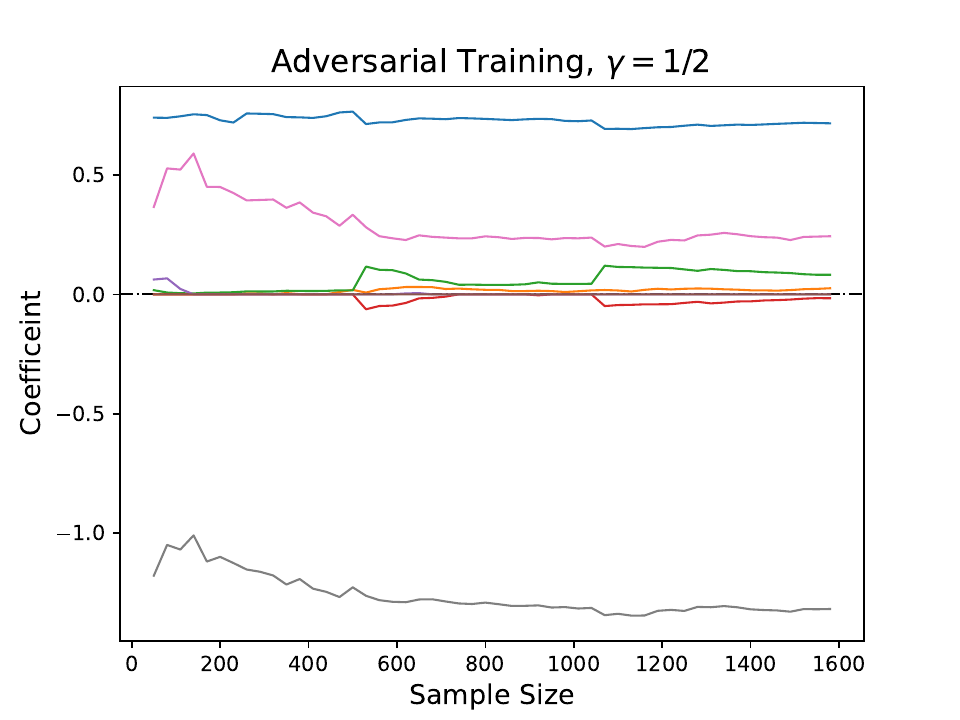} 
    \end{minipage}
        \begin{minipage}{0.5\textwidth}
        \centering
        \includegraphics[width=\textwidth]{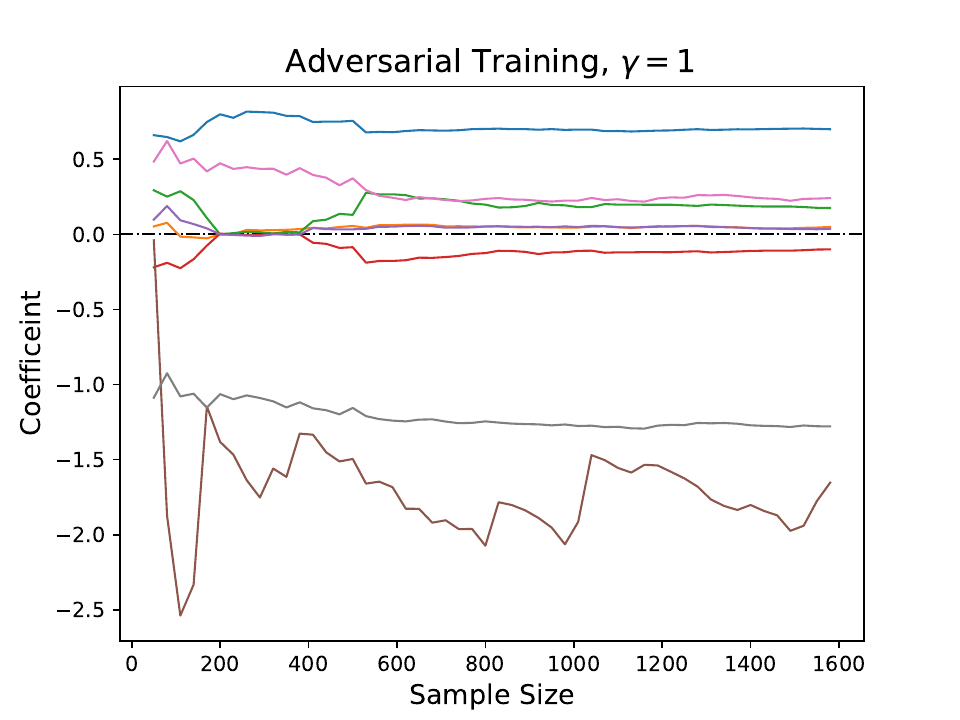} 
    \end{minipage}\hfill
            \begin{minipage}{0.5\textwidth}
        \centering
        \includegraphics[width=\textwidth]{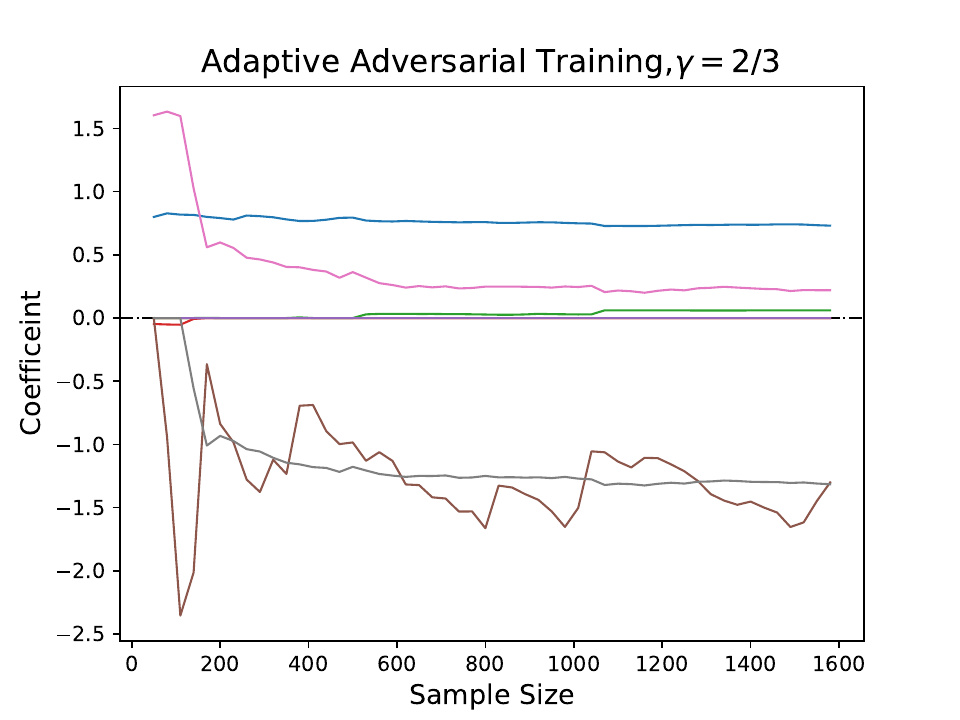} 
    \end{minipage}
    \caption{Coefficient Path in the Linear Regression}
    \label{realdatfigure}
\end{figure}

\begin{table}[htbp]
\begin{center}
\begin{tabular}{c|c c c c c c} 
 \hline
  & $n=100$ & $n=300$& $n=600$ & $n=900$ & $n=1200$ & $n=1500$\\ [0.5ex] 
\hline\hline
 AT, $\gamma=1/3$ & 0.3556& 0.3512& 0.3505& 0.3471& 0.3453& 0.3438\\ 

 AT, $\gamma=1/2$ & $\bm{0.3552}$& 0.3444& 0.3436&0.3403& 0.3426& 0.3407\\

 AT, $\gamma=1$ &0.3798& $\bm{0.3234}$& 0.3312& 0.3247&0.3254&0.3233 \\

 Adaptive AT, $\gamma=2/3$ & 0.6381& 0.3295& $\bm{0.3257}$& $\bm{0.3225}$& $\bm{0.3247}$& $\bm{0.3212}$\\ [1ex]
 \hline
\end{tabular}
\end{center}
\caption{Average of Squared Prediction Error in the Linear Regression}
\label{realtable}
\end{table}

\begin{figure}[htbp]
    \centering
    \begin{minipage}{0.5\textwidth}
        \centering
        \includegraphics[width=\textwidth]{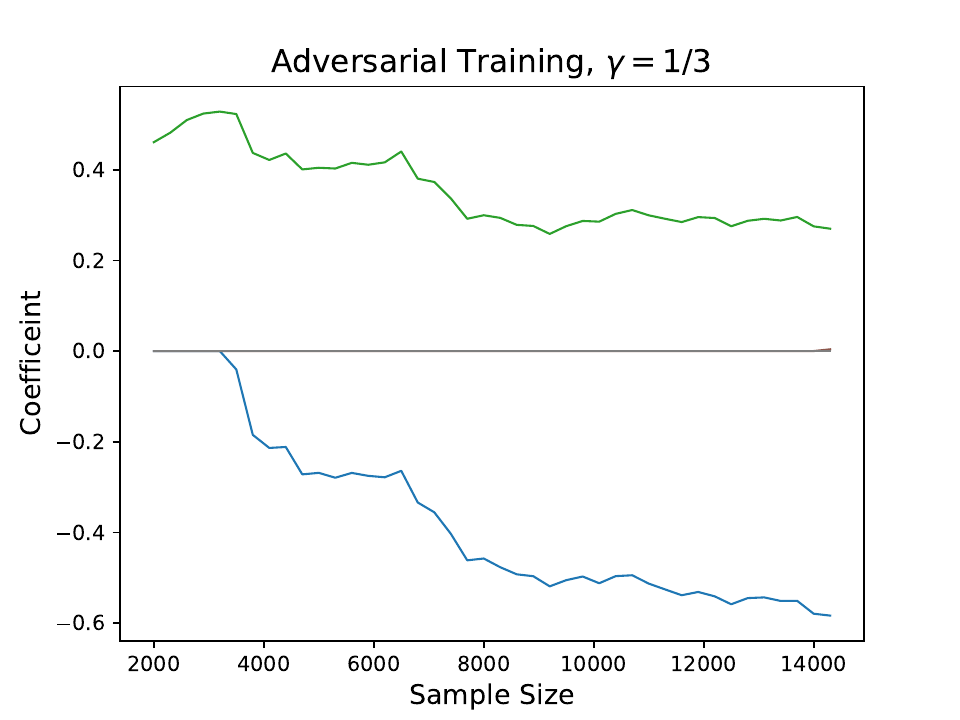} 
    \end{minipage}\hfill
    \begin{minipage}{0.5\textwidth}
        \centering
        \includegraphics[width=\textwidth]{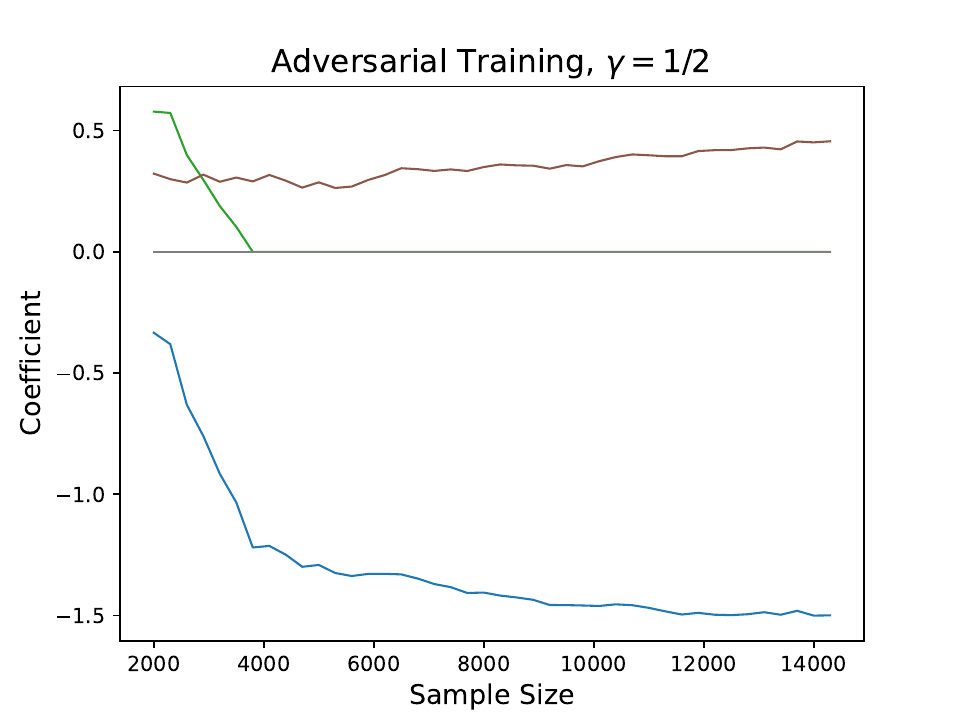} 
    \end{minipage}
        \begin{minipage}{0.5\textwidth}
        \centering
        \includegraphics[width=\textwidth]{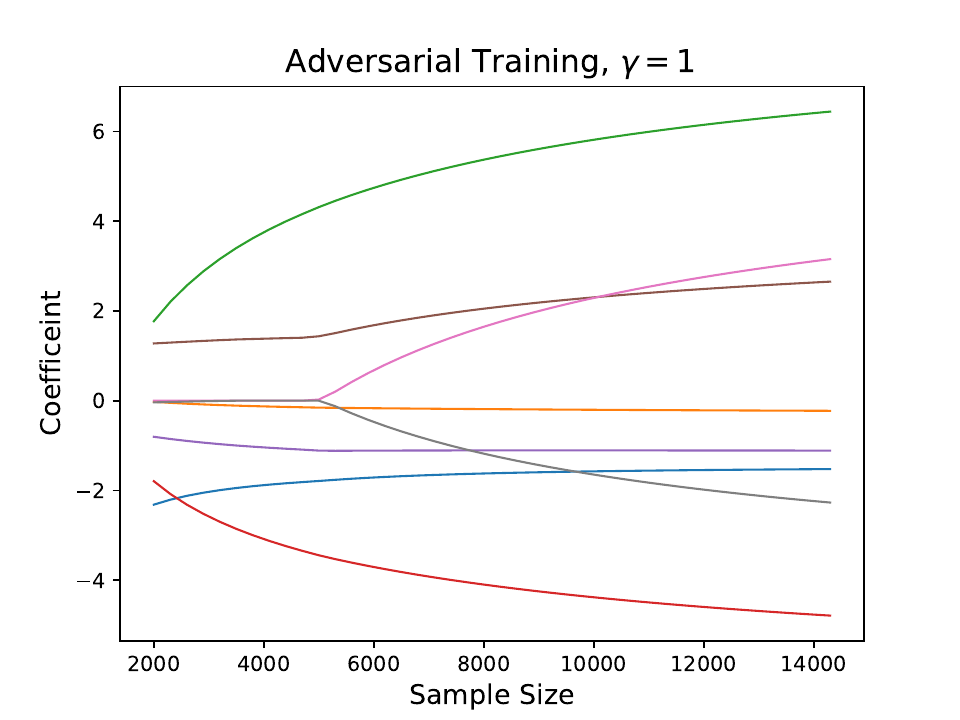} 
    \end{minipage}\hfill
            \begin{minipage}{0.5\textwidth}
        \centering
        \includegraphics[width=\textwidth]{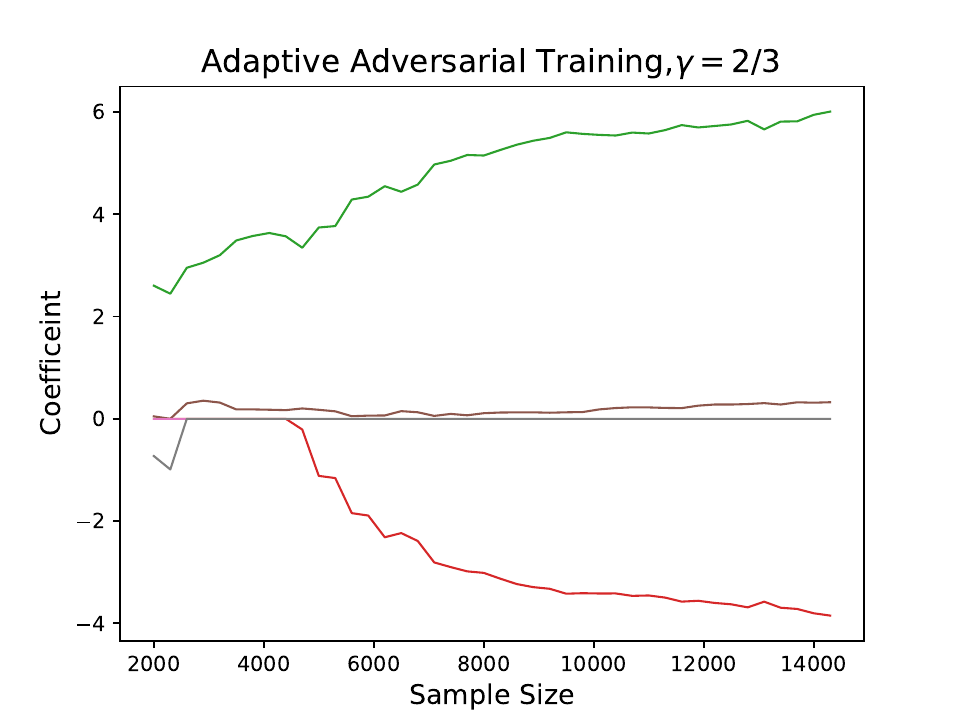} 
    \end{minipage}
    \caption{Coefficient Path in the Logistic Regression}
    \label{realdatfigure2}
\end{figure}

\begin{table}[htbp]
\begin{center}
\begin{tabular}{c|c c c c c c} 
 \hline
  & $n=4000$ & $n=6000$& $n=8000$ & $n=10000$ & $n=12000$ & $n=15000$\\ [0.5ex] 
\hline\hline
 AT, $\gamma=1/3$ & 0.2263& $\bm{0.2550}$& 0.2922& 0.2969&0.2978& 0.3078\\ 

 AT, $\gamma=1/2$ & 0.3137& 0.3204& 0.3187& 0.3170& 0.3151& 0.3134\\

 AT, $\gamma=1$ & 0.3254& 0.3436&0.3494&0.3439&0.3436& 0.3383 \\

 Adaptive AT, $\gamma=2/3$ & $\bm{0.1609}$ & 0.2693& $\bm{0.2785}$& $\bm{0.2821}$& $\bm{0.2838}$& $\bm{0.2849}$\\ [1ex]
 \hline
\end{tabular}
\end{center}
\caption{Average of Squared Prediction Error in the Logistic Regression}
\label{realtable2}
\end{table}

\section{Discussions}\label{dis}
In this paper, we leverage tools from statistics (in particular, regularized statistical learning) and optimization (in particular, distributionally robust optimization) to characterize the asymptotic behavior of the powerful machine learning model---adversarial training. 
We provide a theoretical explanation for the sparsity-recovery ability of the adversarial training procedure through the lens of large-sample asymptotics.  
Furthermore, we propose adaptive adversarial training, which incorporates the empirical risk minimization estimator as weights for the adversarial perturbation.  
In this way, adaptive adversarial training yields an asymptotically unbiased estimator that could produce sparse solutions.

There are several potential future directions. 
One may notice that there is intense literature delivering the non-asymptotic analysis for $\ell_1$-norm regularized estimator using the primal-dual witness approach \citep{wainwright2009sharp,honorio2014unified}. A promising future direction is to apply these techniques to provide a more precise non-asymptotic characterization of adversarial training under $\ell_\infty$-perturbation.

Moreover, there are two primary frameworks for analyzing statistical asymptotics. The first is the fixed-dimension framework, assuming that the sample size becomes large while keeping the variable dimension fixed. 
The fixed-dimension setting is applied in works such as \cite{fu2000asymptotics, blanchet2022confidence} and in this paper.
The second is the high-dimension framework, assuming that the sample size and the variable dimension become increasingly large in a fixed ratio. The high-dimension framework is utilized in works such as \cite{wang2017asymptotics, sur2019modern,zhao2022asymptotic}.
While the asymptotic accuracy of adversarial training in high-dimensional regimes has been well studied in the literature \citep{javanmard2022precise, taheri2023asymptotic, hassani2024curse}, how to reduce the bias in the adversarial-trained estimator under the high-dimension regime has not been investigated. Therefore, a promising future direction is to explore the bias-reduction techniques for the adversarial-trained estimator in high-dimensional settings, for instance, investigating whether the proposed adaptive technique in this paper remains effective under the high-dimension framework.

Finally, we discuss the robustness property of adversarial training. 
As mentioned in Section \ref{asymsection}, the adversarial training under $\ell_\infty$-perturbation achieves robustness against both adversarial input perturbations and overfitting due to its regularization reformulation.
In addition, our investigation reveals that adversarial training under $\ell_\infty$-perturbation enjoys sparsity-recovery ability when $\gamma=1/2$.
Recovering sparsity can be considered as pushing the estimated coefficients toward zero for components that are truly zero-valued, which in turn strengthens the robustness of the adversarial training estimator.
Therefore, under the criterion of robustness, particularly when the ground-truth model contains zero-valued components, $\gamma=1/2$ remains the optimal choice.
However, the $\ell_\infty$-perturbed adversarial training might be sensitive to outliers since adversarial training involves a min-max robust optimization problem \citep {jiang2024distributionally}. In this way, statistical outlier-robustness improvement for adversarial training can be considered as one of our future work. 
\bibliographystyle{apalike}
\bibliography{JASA.bib}

\begin{thebibliography}{}

\bibitem[Blanchet et~al., 2022]{blanchet2022confidence}
Blanchet, J., Murthy, K., and Si, N. (2022).
\newblock Confidence regions in {W}asserstein distributionally robust estimation.
\newblock {\em Biometrika}, 109(2):295--315.

\bibitem[Blanchet and Shapiro, 2023]{blanchet2023statistical}
Blanchet, J. and Shapiro, A. (2023).
\newblock Statistical limit theorems in distributionally robust optimization.
\newblock In {\em 2023 Winter Simulation Conference (WSC)}, pages 31--45. IEEE.

\bibitem[Chen et~al., 2020]{chen2020more}
Chen, L., Min, Y., Zhang, M., and Karbasi, A. (2020).
\newblock More data can expand the generalization gap between adversarially robust and standard models.
\newblock In {\em International Conference on Machine Learning}, pages 1670--1680. PMLR.

\bibitem[Dobriban et~al., 2023]{dobriban2023provable}
Dobriban, E., Hassani, H., Hong, D., and Robey, A. (2023).
\newblock Provable tradeoffs in adversarially robust classification.
\newblock {\em IEEE Transactions on Information Theory}, 69(12):7793--7822.

\bibitem[Drucker and Le~Cun, 1992]{drucker1992improving}
Drucker, H. and Le~Cun, Y. (1992).
\newblock Improving generalization performance using double backpropagation.
\newblock {\em IEEE Transactions on Neural Networks}, 3(6):991--997.

\bibitem[Fu and Knight, 2000]{fu2000asymptotics}
Fu, W. and Knight, K. (2000).
\newblock Asymptotics for lasso-type estimators.
\newblock {\em The Annals of Statistics}, 28(5):1356--1378.

\bibitem[Gao et~al., 2024]{gao2024wasserstein}
Gao, R., Chen, X., and Kleywegt, A.~J. (2024).
\newblock Wasserstein distributionally robust optimization and variation regularization.
\newblock {\em Operations Research}, 72(3):1177--1191.

\bibitem[Gao and Kleywegt, 2022]{gao2022distributionally}
Gao, R. and Kleywegt, A. (2022).
\newblock Distributionally robust stochastic optimization with {W}asserstein distance.
\newblock {\em Mathematics of Operations Research}, 48(2):603--655.

\bibitem[Geyer, 1994]{geyer1994asymptotics}
Geyer, C.~J. (1994).
\newblock On the asymptotics of constrained {M}-estimation.
\newblock {\em The Annals of Statistics}, 22(4):1993--2010.

\bibitem[Geyer, 1996]{geyer1996asymptotics}
Geyer, C.~J. (1996).
\newblock On the asymptotics of convex stochastic optimization.
\newblock {\em Unpublished manuscript}, 37.

\bibitem[Goodfellow et~al., 2015]{Goodfellow2015}
Goodfellow, I.~J., Shlens, J., and Szegedy, C. (2015).
\newblock Explaining and harnessing adversarial examples.
\newblock In {\em International Conference on Learning Representations}.

\bibitem[Guo et~al., 2020]{guo2020connections}
Guo, Y., Chen, L., Chen, Y., and Zhang, C. (2020).
\newblock On connections between regularizations for improving {DNN} robustness.
\newblock {\em IEEE Transactions on Pattern Analysis and Machine Intelligence}, 43(12):4469--4476.

\bibitem[Hassani and Javanmard, 2024]{hassani2024curse}
Hassani, H. and Javanmard, A. (2024).
\newblock The curse of overparametrization in adversarial training: Precise analysis of robust generalization for random features regression.
\newblock {\em The Annals of Statistics}, 52(2):441--465.

\bibitem[Hastie et~al., 2015]{hastie2015statistical}
Hastie, T., Tibshirani, R., and Wainwright, M. (2015).
\newblock {\em Statistical learning with sparsity: the lasso and generalizations}.
\newblock CRC press.

\bibitem[Honorio and Jaakkola, 2014]{honorio2014unified}
Honorio, J. and Jaakkola, T. (2014).
\newblock A unified framework for consistency of regularized loss minimizers.
\newblock In {\em International Conference on Machine Learning}, pages 136--144. PMLR.

\bibitem[Javanmard and Soltanolkotabi, 2022]{javanmard2022precise}
Javanmard, A. and Soltanolkotabi, M. (2022).
\newblock Precise statistical analysis of classification accuracies for adversarial training.
\newblock {\em The Annals of Statistics}, 50(4):2127--2156.

\bibitem[Jiang and Xie, 2024]{jiang2024distributionally}
Jiang, N. and Xie, W. (2024).
\newblock Distributionally favorable optimization: A framework for data-driven decision-making with endogenous outliers.
\newblock {\em SIAM Journal on Optimization}, 34(1):419--458.

\bibitem[Kato, 2009]{kato2009asymptotics}
Kato, K. (2009).
\newblock Asymptotics for argmin processes: Convexity arguments.
\newblock {\em Journal of Multivariate Analysis}, 100(8):1816--1829.

\bibitem[Lyon, 2017]{misc_htru2_372}
Lyon, R. (2017).
\newblock {HTRU2}.
\newblock UCI Machine Learning Repository.
\newblock {DOI}: https://doi.org/10.24432/C5DK6R.

\bibitem[Madry et~al., 2018]{madry2018towards}
Madry, A., Makelov, A., Schmidt, L., Tsipras, D., and Vladu, A. (2018).
\newblock Towards deep learning models resistant to adversarial attacks.
\newblock In {\em International Conference on Learning Representations}.

\bibitem[Pedregosa et~al., 2011]{pedregosa2011scikit}
Pedregosa, F., Varoquaux, G., Gramfort, A., Michel, V., Thirion, B., Grisel, O., Blondel, M., Prettenhofer, P., Weiss, R., Dubourg, V., et~al. (2011).
\newblock Scikit-learn: Machine learning in python.
\newblock {\em Journal of Machine Learning Research}, 12:2825--2830.

\bibitem[Ribeiro et~al., 2023]{ribeiro2023regularization}
Ribeiro, A.~H., Zachariah, D., Bach, F., and Sch{\"o}n, T.~B. (2023).
\newblock Regularization properties of adversarially-trained linear regression.
\newblock In {\em Thirty-seventh Conference on Neural Information Processing Systems}.

\bibitem[Schmidt et~al., 2018]{schmidt2018adversarially}
Schmidt, L., Santurkar, S., Tsipras, D., Talwar, K., and Madry, A. (2018).
\newblock Adversarially robust generalization requires more data.
\newblock In {\em Advances in Neural Information Processing Systems}, volume~31.

\bibitem[Staib and Jegelka, 2017]{staib2017distributionally}
Staib, M. and Jegelka, S. (2017).
\newblock Distributionally robust deep learning as a generalization of adversarial training.
\newblock In {\em NIPS workshop on Machine Learning and Computer Security}, volume~3, page~4.

\bibitem[Sur and Cand{\`e}s, 2019]{sur2019modern}
Sur, P. and Cand{\`e}s, E.~J. (2019).
\newblock A modern maximum-likelihood theory for high-dimensional logistic regression.
\newblock {\em Proceedings of the National Academy of Sciences}, 116(29):14516--14525.

\bibitem[Taheri et~al., 2023]{taheri2023asymptotic}
Taheri, H., Pedarsani, R., and Thrampoulidis, C. (2023).
\newblock Asymptotic behavior of adversarial training in binary linear classification.
\newblock {\em IEEE Transactions on Neural Networks and Learning Systems}, pages 1--9.

\bibitem[Trillos et~al., 2023]{trillos2023multimarginal}
Trillos, N.~G., Jacobs, M., and Kim, J. (2023).
\newblock The multimarginal optimal transport formulation of adversarial multiclass classification.
\newblock {\em Journal of Machine Learning Research}, 24(45):1--56.

\bibitem[Wainwright, 2009]{wainwright2009sharp}
Wainwright, M.~J. (2009).
\newblock Sharp thresholds for high-dimensional and noisy sparsity recovery using $l_1$-constrained quadratic programming (lasso).
\newblock {\em IEEE Transactions on Information Theory}, 55(5):2183--2202.

\bibitem[Wang and Fan, 2017]{wang2017asymptotics}
Wang, W. and Fan, J. (2017).
\newblock Asymptotics of empirical eigenstructure for high dimensional spiked covariance.
\newblock {\em Annals of Statistics}, 45(3):1342.

\bibitem[Xie and Huo, 2024a]{xie2024adjusted}
Xie, Y. and Huo, X. (2024a).
\newblock Adjusted {W}asserstein distributionally robust estimator in statistical learning.
\newblock {\em Journal of Machine Learning Research}, 25(148):1--40.

\bibitem[Xie and Huo, 2024b]{xie2024highdimensional}
Xie, Y. and Huo, X. (2024b).
\newblock High-dimensional (group) adversarial training in linear regression.
\newblock In {\em The Thirty-eighth Annual Conference on Neural Information Processing Systems}.

\bibitem[Xing et~al., 2021]{xing2021generalization}
Xing, Y., Song, Q., and Cheng, G. (2021).
\newblock On the generalization properties of adversarial training.
\newblock In {\em International Conference on Artificial Intelligence and Statistics}, pages 505--513. PMLR.

\bibitem[Yin et~al., 2019]{yin2019rademacher}
Yin, D., Kannan, R., and Bartlett, P. (2019).
\newblock Rademacher complexity for adversarially robust generalization.
\newblock In {\em International Conference on Machine Learning}, pages 7085--7094. PMLR.

\bibitem[Zhao et~al., 2022]{zhao2022asymptotic}
Zhao, Q., Sur, P., and Candes, E.~J. (2022).
\newblock The asymptotic distribution of the {MLE} in high-dimensional logistic models: Arbitrary covariance.
\newblock {\em Bernoulli}, 28(3):1835--1861.

\bibitem[Zou, 2006]{zou2006adaptive}
Zou, H. (2006).
\newblock The adaptive lasso and its oracle properties.
\newblock {\em Journal of the American Statistical Association}, 101(476):1418--1429.

\end{thebibliography}


\begin{thebibliography}{}

\bibitem[\protect\citeauthoryear{Campbell and Austin}{Campbell and
  Austin}{2002}]{Campbell02}
Campbell, J.~I. and S.~Austin (2002).
\newblock Effects of response time deadlines on adults' strategy choices for
  simple addition.
\newblock {\em Memory \& Cognition\/}~{\em 30\/}(6), 988--994.

\bibitem[\protect\citeauthoryear{Chi, Feltovich, and Glaser}{Chi
  et~al.}{1981}]{Chi81}
Chi, M.~T., P.~J. Feltovich, and R.~Glaser (1981).
\newblock Categorization and representation of physics problems by experts and
  novices.
\newblock {\em Cognitive science\/}~{\em 5\/}(2), 121--152.

\bibitem[\protect\citeauthoryear{Schubert, Denmark, Crandall, Grome, and
  Pappas}{Schubert et~al.}{2013}]{Schubert13}
Schubert, C.~C., T.~K. Denmark, B.~Crandall, A.~Grome, and J.~Pappas (2013).
\newblock Characterizing novice-expert differences in macrocognition: an
  exploratory study of cognitive work in the emergency department.
\newblock {\em Annals of emergency medicine\/}~{\em 61\/}(1), 96--109.

\end{thebibliography}

\newpage
\appendix
\section{Proofs}\label{proofsection}
\subsection{Proof of Proposition \ref{prop1}}
\begin{proof}
Since the function $h$ is differentiable, Taylor expansion implies that
\[h(\bm{z}+\delta\Delta)= h(\bm{z})+\delta \nabla h(\bm{\xi})^\top\Delta,\]
\begin{equation}\begin{aligned}\label{expansion}
h(\bm{z}+\delta\Delta)= h(\bm{z})+\delta \nabla h(\bm{z})^\top\Delta+ R(\bm{z},\delta\Delta),\end{aligned}\end{equation}
where $\xi$ lies on the line segment between $\bm{z}$ and $\bm{z}+\delta\Delta$.
Thus, the remainder $R(\bm{z},\delta\Delta)$ can be expressed as 
\[R(\bm{z},\delta\Delta)=\delta \left( \nabla h(\bm{\xi})-\nabla h(\bm{z})\right)^\top \Delta.\]

It follows from     \eqref{expansion} and the H\"older's inequality that 
\begin{equation}\begin{aligned}\label{split}
    \max _{\Vert\Delta\Vert\leq 1}\{ h(\bm{z}+\delta\Delta)- h(\bm{z})\}&= \max _{\Vert\Delta\Vert\leq 1}\{\delta\nabla h(\bm{z})^\top\Delta+R(\bm{z},\delta\Delta)\}\\
    &\leq \delta \max _{\Vert\Delta\Vert\leq 1}\nabla h(\bm{z})^\top\Delta +\max_{\Vert\Delta\Vert\leq 1} R(\bm{z},\delta\Delta)\\
    &\leq \delta \Vert\nabla h(\bm{z})\Vert_\ast +\max_{\Vert\Delta\Vert\leq 1} R(\bm{z},\delta\Delta) \end{aligned}\end{equation}

Notice we have that
\[ \vert R(\bm{z},\delta\Delta)\vert = \delta \vert\left( \nabla h(\bm{\xi})-\nabla h(\bm{z})\right)^\top \Delta\vert\leq \delta\Vert\nabla h(\bm{\xi})-\nabla h(\bm{z})\Vert_\ast \Vert\Delta\Vert,\]
resulting in 
\[ \frac{1}{\delta}\mathbb{E}_{\bm{Z}\sim P}\left[\max_{\Vert\Delta\Vert\leq 1} R(\bm{Z},\delta\Delta)\right] \leq \int_{\mathbb{R}^d}  \Vert\nabla h(\bm{\xi})-\nabla h(\bm{z})\Vert_\ast~dP(\bm{z}),\]
where $\xi$ lies on the line segment between $\bm{z}$ and $\bm{z}+\delta\Delta$.

Since $\nabla h$ is uniformly continuous, we can conclude that 
\begin{equation}\label{converge}\lim_{\delta\to 0}\frac{1}{\delta}\mathbb{E}_{\bm{Z}\sim P}\left[\max_{\Vert\Delta\Vert\leq 1} R(\bm{Z},\delta\Delta)\right] = 0.\end{equation}

In this way, it follows from \eqref{split} and \eqref{converge} that 
    \begin{equation*}
    \begin{aligned}&\mathbb{E}_{\bm{Z}\sim P}\left[\max _{\Vert\Delta\Vert\leq \delta}h(\bm{Z}+\Delta)\right]-\mathbb{E}_ {\bm{Z}\sim P}\left[ h(\bm{Z})\right]- \delta \mathbb{E}_{\bm{Z}\sim P}\left[ \Vert\nabla h(\bm{Z})\Vert_\ast\right]\\
=&\mathbb{E}_{\bm{Z}\sim P}\left[\max _{\Vert\Delta\Vert\leq \delta}h(\bm{Z}+\Delta)- h(\bm{Z})\right]-\delta\mathbb{E}_{\bm{Z}\sim P} \left[ \Vert \nabla h(\bm{Z})\Vert_\ast \right]\\
=& \mathbb{E}_{\bm{Z}\sim P}\left[\max _{\Vert\Delta\Vert\leq 1}h(\bm{Z}+\delta\Delta)- h(\bm{Z})\right]-\delta\mathbb{E}_{\bm{Z}\sim P} \left[ \Vert \nabla h(\bm{Z})\Vert_\ast \right]\\
\leq & \mathbb{E}_{\bm{Z}\sim P}\left[\max_{\Vert\Delta\Vert\leq 1} R(\bm{Z},\delta\Delta)\right]\\
= & o(\delta),
\end{aligned}\end{equation*}
which is equivalent to 
 \[\mathbb{E}_{\bm{Z}\sim P}\left[\max _{\Vert\Delta\Vert\leq \delta}h(\bm{Z}+\Delta)\right]=\mathbb{E}_ {\bm{Z}\sim P}\left[ h(\bm{Z})\right]+ \delta \mathbb{E}_{\bm{Z}\sim P}\left[ \Vert\nabla h(\bm{Z})\Vert_\ast\right]+o(\delta)\]
 as $\delta\to 0$.

\end{proof}
\subsection{Proof of Proposition \ref{assumelinear}}
\begin{proof}
      It follows from Lemma 20 and Lemma 23 in \citet{xie2024adjusted} that it suffices to show that
    $P_\ast\left(L^\prime(\langle \bm{X},\beta^\ast\rangle,Y)=0\right)=0,$
    \( L^{\prime\prime}(f,y) \) is uniformly continuous with respect to the first argument, and the following inequalities hold:
$\mathbb{E}_{P_\ast}\left[ \left\vert L^\prime(\langle\bm{X},\beta^\ast\rangle,Y)\right\vert\right]<\infty,$ 
    $\mathbb{E}_{P_\ast}\left[\Vert\nabla_{\beta}L^\prime(\langle\bm{X},\beta^\ast\rangle,Y)\Vert_2\right]<\infty,$
    $\mathbb{E}_{P_\ast}\left[\Vert \nabla_{\beta}^2L(\langle \bm{X},\beta^\ast\rangle,Y)\Vert_2\right]<\infty.$

    In linear regression, we have 
    $L^\prime(\langle \bm{x},\beta^\ast\rangle,y) = \langle \bm{x},\beta^\ast\rangle - y$, indicating
    $P_\ast\left(L^\prime(\langle \bm{X},\beta^\ast\rangle,Y)=0\right) = 0.$
    We also have \( L^{\prime\prime}(f,y) = 1 \), which is trivially uniformly continuous.

    Additionally, we establish the required expectation bounds:
    \[\mathbb{E}_{P_\ast}\left[\Vert\nabla_{\beta}L^\prime(\langle\bm{X},\beta^\ast\rangle,Y)\Vert_2\right] = \mathbb{E}_{P_\ast}\left[\Vert \bm{X}\Vert_2\right] \leq \sqrt{\mathbb{E}_{P_\ast}\left[\Vert \bm{X}\Vert_2^2 \right]} < \infty.
    \]
    \begin{equation*}
    \begin{aligned}
        \mathbb{E}_{P_\ast}\left[\left\vert L^\prime(\langle\bm{X},\beta^\ast\rangle,Y)\right\vert\right] &= \mathbb{E}_{P_\ast}\left[\vert\langle \bm{X},\beta^\ast\rangle-Y\vert\right] \\
        &= \mathbb{E}_{P_\ast}\left[\mathbb{E}_{P_\ast}\left[\left\vert\langle \bm{X},\beta^\ast\rangle-Y\right\vert \mid \bm{X} \right]\right] \\
        &= \sqrt{\frac{2}{\pi}} \mathbb{E}_{P_\ast}\left[\operatorname{Var}_{P_\ast}(Y\vert \bm{X})\right] < \infty.
    \end{aligned}
    \end{equation*}
    \[
    \mathbb{E}_{P_\ast}\left[\Vert \nabla_{\beta}^2L(\langle \bm{X},\beta^\ast\rangle,Y)\Vert_2\right] = \mathbb{E}_{P_\ast}\left[\Vert\bm{X}\bm{X}^\top\Vert_2\right] = \mathbb{E}_{P_\ast}\left[\Vert\bm{X}\Vert_2^2\right] < \infty.
    \]\end{proof}
\subsection{Proof of Proposition \ref{assumelog}}
\begin{proof}
    It follows from Lemma 29 and Lemma 22 in \citet{xie2024adjusted} that it suffices to show that
$P_\ast\left(L^\prime(\langle \bm{X},\beta^\ast\rangle,Y)=0\right)=0,$
    \( L^{\prime\prime}(f,y) \) is uniformly continuous with respect to the first argument, and the following inequalities hold:
    $\mathbb{E}_{P_\ast}\left[ \left\vert L^\prime(\langle\bm{X},\beta^\ast\rangle,Y)\right\vert\right]<\infty$,
    $
    \mathbb{E}_{P_\ast}\left[\Vert\nabla_{\beta}L^\prime(\langle\bm{X},\beta^\ast\rangle,Y)\Vert_2\right]<\infty,$
    $
    \mathbb{E}_{P_\ast}\left[\Vert \nabla_{\beta}^2L(\langle \bm{X},\beta^\ast\rangle,Y)\Vert_2\right]<\infty.$
    
    In logistic regression, we have 
    \[
    L^\prime(\langle \bm{x},\beta^\ast\rangle,y)=\frac{-y}{1+e^{y\langle\bm{x},\beta^\ast\rangle}},
    \]
    which implies that 
    \[
    P_\ast\left(L^\prime(\langle \bm{X},\beta^\ast\rangle,Y)=0\right)=0.
    \]

    Furthermore, we observe that 
    \[
    L^{\prime\prime}(f, y) = \frac{e^{yf}}{(1 + e^{yf})^2}.
    \]
    It can be shown that the derivative of \( L^{\prime\prime}(f, y) \) with respect to \( f \) is bounded, ensuring uniform continuity.

    Additionally, we establish the required expectations:
    \begin{equation*}
    \begin{aligned}
        \mathbb{E}_{P_\ast}\left[\left\vert L^\prime(\langle\bm{X},\beta^\ast\rangle,Y)\right\vert\right] &= \mathbb{E}_{P_\ast}\left[\left\vert\frac{-Y}{1+e^{Y\langle\bm{X},\beta^\ast\rangle}}\right\vert \right] < \infty.
    \end{aligned}
    \end{equation*}
    \[
    \mathbb{E}_{P_\ast}\left[\Vert\nabla_{\beta}L^\prime(\langle\bm{X},\beta^\ast\rangle,Y)\Vert_2\right] = \mathbb{E}_{P_\ast}\left[\frac{e^{Y\langle \bm{X},\beta^\ast\rangle}}{(1+e^{Y\langle \bm{X},\beta^\ast\rangle})^2} \Vert \bm{X}\Vert_2\right] < \mathbb{E}_{P_\ast}\left[\Vert \bm{X}\Vert_2 \right] \leq \sqrt{\mathbb{E}_{P_\ast}\left[\Vert \bm{X}\Vert_2^2 \right]} < \infty.
    \]
    \[
    \mathbb{E}_{P_\ast}\left[\Vert \nabla_{\beta}^2L(\langle \bm{X},\beta^\ast\rangle,Y)\Vert_2\right] < \mathbb{E}_{P_\ast}\left[\Vert\bm{X}\bm{X}^\top\Vert_2\right] = \mathbb{E}_{P_\ast}\left[\Vert\bm{X}\Vert_2^2\right] < \infty.
    \]
\end{proof}
\subsection{Proof of Theorem \ref{asyminfty}}
\subsubsection{Lemmas}
To prove Theorem \ref{asyminfty}, we need a more precise description of the regularization effect for adversarial training.
The more precise description is based on the following lemma.

\begin{lemma}\label{uniform}
    If $h$ is twice differentiable, $\nabla^2 h$ is uniformly continuous, $\mathbb{E}_{\bm{Z}\sim P}\left[ \Vert\nabla h(\bm{Z})\Vert_\ast\right]<\infty$, and $\mathbb{E}_{\bm{Z}\sim P}\left[ \Vert\nabla^2 h(\bm{Z})\Vert_\ast\right]<\infty$,  then we have that 
    \[ \mathbb{E}_{\bm{Z}\sim P}\left[\max_{\Vert \Delta\Vert\leq 1} h(\bm{Z}+\delta \Delta )- h(\bm{Z})\right]=\delta  \mathbb{E}_{\bm{Z}\sim P}\left[ \max_{\Vert\Delta\Vert\leq 1}\left\{\nabla h(\bm{Z})^\top \Delta+\frac{1}{2}\delta \Delta^\top \nabla^2h(\bm{Z}) \Delta\right\}\right] +o(\delta^2),\]    as $\delta\to 0$.
\end{lemma}

\begin{proof}
    It follows from the Taylor expansion that 
        \[ h(\bm{z}+\delta \Delta )- h(\bm{z})=\delta   \nabla h(\bm{z})^\top \Delta+\frac{1}{2}\delta^2 \Delta^\top \nabla^2h(\bm{\xi}) \Delta ,\]
        \begin{equation}\label{expansion2} h(\bm{z}+\delta \Delta )- h(\bm{z})=\delta   \nabla h(\bm{z})^\top \Delta+\frac{1}{2}\delta^2 \Delta^\top \nabla^2h(\bm{z}) \Delta +R^1 (\bm{z},\delta\Delta),\end{equation}
        where $\xi$ lies on the line segment between $\bm{z}$ and $\bm{z}+\delta\Delta$. Thus, the remainder $R^1(\bm{z},\delta\Delta)$ can be expressed as 
\[R^1(\bm{z},\delta\Delta)=\frac{1}{2}\delta^2 \Delta^\top \left( \nabla^2 h(\bm{\xi})-\nabla^2 h(\bm{z})\right) \Delta.\]

       Then, we have that
   \[\frac { \vert R^1(\bm{z},\delta\Delta)\vert}{\delta^2} =\frac{1}{2} \left\vert\Delta ^{\top} \left( \nabla ^2 h(\bm{z})- \nabla ^2h(\bm{\xi})\right)\Delta\right\vert\leq \frac{C}{2}\Vert  \nabla ^2 h(\bm{\xi})-\nabla^2 h(\bm{z})\Vert_{\ast} \Vert\Delta\Vert^2, \] 
   where $C>0$ is some constant depending on the dimension of $\Delta$, 
resulting in 
\[\frac{1}{\delta^2}\mathbb{E}_{\bm{Z}\sim P}\left[\max_{\Vert\Delta\Vert\leq 1} R^1(\bm{Z},\delta\Delta)\right] \leq \frac{C}{2}\int_{\mathbb{R}^d}  \Vert\nabla^2 h(\bm{\xi})-\nabla^2 h(\bm{z})\Vert_\ast ~dP(\bm{z}),\]
where $\xi$ lies on the line segment between $\bm{z}$ and $\bm{z}+\delta\Delta$.
   
   Due to the uniform continuity of $\nabla ^2h$, we have that 
\[\lim_{\delta\to 0} \frac{1}{\delta^2}\mathbb{E}_{\bm{Z}\sim P}\left[\max_{\Vert\Delta\Vert\leq 1} R^1(\bm{Z},\delta\Delta)\right] =0.\]

It follows from  \eqref{expansion2} that 
   \begin{equation*}\begin{aligned}
       &\mathbb{E}_{\bm{Z}\sim P}\left[\max_{\Vert \Delta\Vert\leq 1} h(\bm{Z}+\delta \Delta )- h(\bm{Z})\right]-\delta  \mathbb{E}_{\bm{Z}\sim P}\left[ \max_{\Vert\Delta\Vert\leq 1}\left\{\nabla h(\bm{Z})^\top \Delta+\frac{1}{2}\delta \Delta^\top \nabla^2h(\bm{Z}) \Delta\right\}\right]\\
       \leq&\mathbb{E}_{\bm{Z}\sim P}\left[\max_{\Vert\Delta\Vert\leq 1} R^1(\bm{Z},\delta\Delta)\right]\\
       =&o(\delta^2),\end{aligned}\end{equation*} 
       as $\delta\to 0$.
\end{proof}
Based on Lemma \ref{uniform}, we could obtain the second-order regularization effect formula.

\begin{lemma}\label{reglemma1} If the function $L(f,y)$ is twice differentiable w.r.t. the first argument, and $L^{\prime\prime}(f,y)$ is uniformly continuous w.r.t. the first argument, then we have that 
\begin{equation}
\begin{aligned}\label{regeffectglmprecise}
&\mathbb{E}_{\mathbb{P}_n}\left[\max _{\Vert\Delta\Vert_{p}\leq \delta}L(\langle \bm{X}+\Delta,\beta\rangle,Y)\right]\\
=&\mathbb{E}_{\mathbb{P}_n}\left[ L(\langle\bm{X},\beta\rangle,Y)\right]+\delta \Vert\beta\Vert_q\mathbb{E}_{\mathbb{P}_n}\left[ \vert L^{\prime}(\langle \bm{X},\beta\rangle,Y)\vert\right]+\frac{1}{2}\delta^2\Vert\beta\Vert_q^2 \mathbb{E}_{\mathbb{P}_n}[L^{\prime\prime}(\langle \bm{X},\beta\rangle,Y)]+
    o(\delta^2),
\end{aligned}\end{equation}
as $\delta\to 0$, where $p\in(1,\infty]$, and $1/p+1/q=1$.
\end{lemma}
\begin{proof}

Since the function $L(f,y)$ is twice differentiable w.r.t. the first argument and $L^{\prime\prime}(f,y)$ is uniformly continuous w.r.t. the first argument, it follows from Lemma \ref{uniform} that
    \begin{equation}
    \begin{aligned}\label{firstfirst}
    &\mathbb{E}_{\mathbb{P}_n}\left[\max _{\Vert\Delta\Vert_p\leq \delta}L(\langle \bm{X}+\Delta,\beta\rangle,Y)- L(\langle \bm{X} ,\beta\rangle,Y)\right]\\
    =&\mathbb{E}_{\mathbb{P}_n}\left[\max_{\Vert\Delta\Vert_p\leq 1} L(\langle \bm{X}+\delta\Delta,\beta\rangle,Y)- L(\langle \bm{X} ,\beta\rangle,Y)\right]\\
    =&\delta\mathbb{E}_{\mathbb{P}_n}\left[\max_{\Vert\Delta\Vert_p\leq 1}\left\{L^\prime(\langle \bm{X},\beta\rangle,Y)\Delta^\top\beta +\frac{1}{2}\delta  L^{\prime\prime}(\langle \bm{X},\beta\rangle,Y) (\Delta^\top \beta)^2\right\}\right] +o(\delta^2),
    \end{aligned}
    \end{equation}
    as $\delta\to 0$.

Then, we focus on the term $\max_{\Vert\Delta\Vert_p\leq 1}\left\{L^\prime(\langle \bm{X},\beta\rangle,Y)(\Delta^\top \beta)+\frac{1}{2}\delta L^{\prime\prime}(\langle \bm{X},\beta\rangle,Y)(\Delta^\top\beta)^2\right\}.$
If we let $t=\Delta^\top\beta$, the resulting optimization problem takes a quadratic form. Since $\delta\to0$,  the maximum can only be taken at the boundaries of the interval $\vert t\vert\leq \Vert \beta\Vert_q$. Then, we have that 
 \[\max_{\vert t\vert\leq \Vert \beta\Vert_q}\left\{\frac{1}{2}\delta L^{\prime\prime}(\langle \bm{X},\beta\rangle,Y)t^2+L^\prime(\langle \bm{X},\beta\rangle,Y)t\right\}=\frac{1}{2}\delta  L^{\prime\prime}(\langle \bm{X},\beta\rangle,Y)\Vert\beta\Vert_q^2+\vert L^\prime(\langle \bm{X},\beta\rangle,Y)\vert\Vert\beta\Vert_q.\]

In this way, problem \eqref{firstfirst} will have the following reformulation
\begin{equation}
\begin{aligned}\label{aa1}
&\mathbb{E}_{\mathbb{P}_n}\left[\max _{\Vert\Delta\Vert_p\leq \delta}L(\langle \bm{X}+\Delta,\beta\rangle,Y)- L(\langle \bm{X} ,\beta\rangle,Y)\right]
\\
=&\delta \mathbb{E}_{\mathbb{P}_n}\left[\max_{\Vert\Delta\Vert_p\leq 1}\left\{L^\prime(\langle \bm{X},\beta\rangle,Y)(\Delta^\top \beta)+\frac{1}{2}\delta L^{\prime\prime}(\langle \bm{X},\beta\rangle,Y)(\Delta^\top\beta)^2\right\} \right]+o(\delta^2)\\
=&\delta \mathbb{E}_{\mathbb{P}_n}\left[\vert L^\prime(\langle \bm{X},\beta\rangle,Y)\vert\Vert\beta\Vert_q+\frac{1}{2}\delta  L^{\prime\prime}(\langle \bm{X},\beta\rangle,Y)\Vert\beta\Vert_q^2\right]+o(\delta^2)\\
=&\delta\Vert\beta\Vert_q \mathbb{E}_{\mathbb{P}_n}\left[\vert L^\prime(\langle \bm{X},\beta\rangle,Y)\vert\right]+\frac{1}{2}\delta^2\Vert\beta\Vert_q^2\mathbb{E}_{\mathbb{P}_n}\left[L^{\prime\prime}(\langle \bm{X},\beta\rangle,Y)\right]+o(\delta^2),
\end{aligned}
\end{equation}
as $\delta\to 0$.
 
In conclusion,
    \eqref{regeffectglmprecise} holds.
\end{proof}

The following lemma is also needed to prove Theorem  \ref{asyminfty}.
\begin{lemma}\label{convexity}
If we denote
\[ \Psi_n(\beta)=\mathbb{E}_{\mathbb{P}_n}\left[\max _{\Vert\Delta\Vert_p\leq \delta}L(\langle\bm{X}+\Delta,\beta\rangle,Y)\right],\]
where $p\in(1,\infty]$, and $L(\langle\bm{x},\beta\rangle,y)$ is convex w.r.t $\beta$,
then $\Psi_n(\beta)$ is convex w.r.t. $\beta$.
\end{lemma}
\begin{proof}
    If we choose any $\beta_1, \beta_2\in B$ and $a_1\in(0,1)$, we have that
\begin{equation*}
\begin{aligned}
&a_1\Psi_n(\beta_1)+(1-a_1)\Psi_n(\beta_2)\\
=&\mathbb{E}_{\mathbb{P}_n}\left[a_1\max _{\Vert\Delta\Vert_p\leq \delta}L(\langle\bm{X}+\Delta,\beta_1\rangle,Y)+(1-a_1)\max _{\Vert\Delta\Vert_p\leq \delta}L(\langle\bm{X}+\Delta,\beta_2\rangle,Y)\right]\\
=&\mathbb{E}_{\mathbb{P}_n}\left[\max _{\Vert\Delta\Vert_p\leq \delta}a_1L(\langle\bm{X}+\Delta,\beta_1\rangle,Y)+\max _{\Vert\Delta\Vert_p\leq \delta}(1-a_1)L(\langle\bm{X}+\Delta,\beta_2\rangle,Y)\right]\\
\geq&\mathbb{E}_{\mathbb{P}_n}\left[\max _{\Vert\Delta\Vert_p\leq \delta}a_1 L(\langle\bm{X}+\Delta,\beta_1\rangle,Y)+(1-a_1)L(\langle\bm{X}+\Delta,\beta_2\rangle,Y)\right]\\
\overset{(a)}{\geq}&\mathbb{E}_{\mathbb{P}_n}\left[\max _{\Vert\Delta\Vert_p\leq \delta}L(\langle\bm{X}+\Delta,a_1\beta_1+(1-a_1)\beta_2\rangle,Y),\right]\\
=&\Psi_n(a_1\beta_1+(1-a_1)\beta_2),
\end{aligned}\end{equation*}
where (a) comes from the convexity of  $L(\langle\bm{x},\beta\rangle,y)$. 
Then, $\Psi_n(\beta)$ is convex due to the definition of convexity.
\end{proof}
\subsubsection{Proof}
We begin to prove Theorem \ref{asyminfty}.
\begin{proof}
\textbf{CASE 1: $\bm{\gamma=1/2}$}

We obtain from Lemma \ref{reglemma1} that 
\begin{equation*}
    \begin{aligned}
&\mathbb{E}_{\mathbb{P}_n}\left[\max _{\Vert\Delta\Vert_{\infty}\leq \delta_n}L(\langle \bm{X}+\Delta,\beta\rangle,Y)\right]\\
=&\mathbb{E}_{\mathbb{P}_n}\left[ L(\langle\bm{X},\beta\rangle,Y)\right]+\delta_n \Vert\beta\Vert_1\mathbb{E}_{\mathbb{P}_n}\left[ \left\vert L^\prime(\langle\bm{X},\beta\rangle,Y)\right\vert\right]+\frac{1}{2}\delta^2_n\Vert\beta\Vert_1^2 \mathbb{E}_{\mathbb{P}_n}[L^{\prime\prime}(\langle \bm{X},\beta\rangle,Y)]+
    o(\delta^2_n),
    \end{aligned}
\end{equation*}
based on which we further have that 
\begin{equation}
\begin{aligned}\label{defofv}
&V_n(\bm{u})\\
\overset{\vartriangle}{=}&n\left(\Psi_n(\beta^\ast+\frac{1}{\sqrt{n}}\bm{u})-\Psi_n(\beta^\ast)\right)\\
=&n\left( \mathbb{E}_{\mathbb{P}_n}\left[ L(\langle\bm{X},\beta^\ast+\frac{1}{\sqrt{n}}\bm{u}\rangle,Y)\right]-\mathbb{E}_{\mathbb{P}_n}\left[ L(\langle\bm{X},\beta^\ast\rangle,Y)\right]\right)\\
&+\eta\sqrt{n}\left(\left\Vert\beta^\ast+\frac{1}{\sqrt{n}}\bm{u}\right\Vert_1\mathbb{E}_{\mathbb{P}_n}\left[ \left\vert L^\prime(\langle\bm{X},\beta^\ast+\frac{1}{\sqrt{n}}\bm{u}\rangle,Y)\right\vert\right]-\left\Vert\beta^\ast\right\Vert_1\mathbb{E}_{\mathbb{P}_n}\left[ \left\vert L^\prime(\langle\bm{X},\beta^\ast\rangle,Y)\right\vert\right]\right)\\
&+\eta^2 \left(\frac{1}{2}\left\Vert\beta^\ast+\frac{1}{\sqrt{n}}\bm{u}\right\Vert_1^2 \mathbb{E}_{\mathbb{P}_n}\left[L^{\prime\prime}(\langle \bm{X},\beta^\ast+\frac{1}{\sqrt{n}}\bm{u}\rangle,Y)\right]-\frac{1}{2}\left\Vert\beta^\ast\right\Vert_1^2 \mathbb{E}_{\mathbb{P}_n}\left[L^{\prime\prime}(\langle \bm{X},\beta^\ast\rangle,Y)\right]\right)+o(1),
\end{aligned}
\end{equation}
as $n\to\infty$.

\textbf{We analyze the terms in \eqref{defofv} as follows.}

\textbf{For the first term in \eqref{defofv}}, due to the asymptotic theory of empirical risk minimization, we have that
\begin{equation}\label{firstequation}n\left( \mathbb{E}_{\mathbb{P}_n}\left[ L(\langle\bm{X},\beta^\ast+\frac{1}{\sqrt{n}}\bm{u}\rangle,Y)\right]-\mathbb{E}_{\mathbb{P}_n}\left[ L(\langle\bm{X},\beta^\ast\rangle,Y)\right]\right)\Rightarrow -\mathbf{G}^\top \bm{u}+\frac{1}{2}\bm{u}^\top H\bm{u}, \end{equation}
where $\mathbf{G}\sim \mathcal{N}(\bm{0},\Sigma)$ with covariance matrix $\Sigma=\Cov_{P_\ast}\left(\nabla_\beta L(\langle \bm{X},\beta^\ast\rangle,Y)\right)$ and 
$H=$\\ $\mathbb{E}_{P_\ast}\left[ \nabla_{\beta}^2 L(\langle \bm{X},\beta^\ast\rangle,Y)\right]$. We refer to Theorem 3, Proposition 6, and the associated proofs in \citet{blanchet2022confidence} for more detailed descriptions.

\textbf{For the second term in    \eqref{defofv}}, we have the following decomposition
\begin{equation}\label{decompose}
\begin{aligned}
&\sqrt{n}\left(\left\Vert\beta^\ast+\frac{1}{\sqrt{n}}\bm{u}\right\Vert_1\mathbb{E}_{\mathbb{P}_n}\left[ \left\vert L^\prime(\langle\bm{X},\beta^\ast+\frac{1}{\sqrt{n}}\bm{u}\rangle,Y)\right\vert\right]-\left\Vert\beta^\ast\right\Vert_1\mathbb{E}_{\mathbb{P}_n}\left[ \left\vert L^\prime(\langle\bm{X},\beta^\ast\rangle,Y)\right\vert\right]\right)\\
=&\sqrt{n}\left(\left\Vert\beta^\ast+\frac{1}{\sqrt{n}}\bm{u}\right\Vert_1-\left\Vert\beta^\ast\right\Vert_1\right)\mathbb{E}_{\mathbb{P}_n}\left[ \left\vert L^\prime(\langle\bm{X},\beta^\ast+\frac{1}{\sqrt{n}}\bm{u}\rangle,Y)\right\vert\right]\\
&+\sqrt{n}\left\Vert\beta^\ast\right\Vert_1\left(\mathbb{E}_{\mathbb{P}_n}\left[ \left\vert L^\prime(\langle\bm{X},\beta^\ast+\frac{1}{\sqrt{n}}\bm{u}\rangle,Y)\right\vert\right]-\mathbb{E}_{\mathbb{P}_n}\left[ \left\vert L^\prime(\langle\bm{X},\beta^\ast\rangle,Y)\right\vert\right]\right).
\end{aligned}
\end{equation}

For the first term in    \eqref{decompose}, we have that 
\begin{equation}\label{furesult}\sqrt{n}\left(\left\Vert\beta^\ast+\frac{1}{\sqrt{n}}\bm{u}\right\Vert_1-\left\Vert\beta^\ast\right\Vert_1\right)\to\sum_{j=1}^d\left([\bm{u}]_j \sign([\beta^\ast]_j) I([\beta^\ast]_j\not=0)+\vert [\bm{u}]_j\vert I([\beta^\ast]_j=0)\right),\end{equation}
and 
\[ \mathbb{E}_{\mathbb{P}_n}\left[ \left\vert L^\prime(\langle\bm{X},\beta^\ast+\frac{1}{\sqrt{n}}\bm{u}\rangle,Y)\right\vert\right]\to_p\mathbb{E}_{P_\ast}\left[ \left\vert L^\prime(\langle\bm{X},\beta^\ast\rangle,Y)\right\vert\right],\]
as $n\to\infty$.

For the second term in    \eqref{decompose}, it follows from Taylor expansion that,\\ 
if $L^\prime (\langle\bm{x},\beta^\ast\rangle,y)\not=0$,
\begin{equation}\label{taylorexpansion}
\begin{aligned}
&\sqrt{n}\left( \left\vert L^\prime(\langle\bm{x},\beta^\ast+\frac{1}{\sqrt{n}}\bm{u}\rangle,y)\right\vert- \left\vert L^\prime(\langle\bm{x},\beta^\ast\rangle,y)\right\vert\right)\\
= & \sign(L^\prime (\langle \bm{x},\beta^\ast\rangle,y)) \nabla_\beta  L^\prime (\langle \bm{x},\beta^\ast\rangle,y)^\top \bm{u} +o(1),
\end{aligned}
\end{equation}
and if $L^\prime (\langle\bm{x},\beta^\ast\rangle,y)=0$,
\begin{equation}\label{taylorexpansion2}
\begin{aligned}
&\sqrt{n}\left( \left\vert L^\prime(\langle\bm{x},\beta^\ast+\frac{1}{\sqrt{n}}\bm{u}\rangle,y)\right\vert- \left\vert L^\prime(\langle\bm{x},\beta^\ast\rangle,y)\right\vert\right)\\
= & \sqrt{n}\left\vert L^\prime(\langle\bm{x},\beta^\ast+\frac{1}{\sqrt{n}}\bm{u}\rangle,y)\right\vert\\
=& \left\vert \nabla_\beta  L^\prime (\langle \bm{x},\beta^\ast\rangle,y)^\top \bm{u}\right\vert+o(1),
\end{aligned}
\end{equation}
as $n\to\infty$.

In this way, we have that
\begin{equation}\label{usedinadaptive}
\begin{aligned}
&\sqrt{n}\left(\mathbb{E}_{\mathbb{P}_n}\left[ \left\vert L^\prime(\langle\bm{X},\beta^\ast+\frac{1}{\sqrt{n}}\bm{u}\rangle,Y)\right\vert\right]-\mathbb{E}_{\mathbb{P}_n}\left[ \left\vert L^\prime(\langle\bm{X},\beta^\ast\rangle,Y)\right\vert\right]\right)\\
=&\mathbb{E}_{\mathbb{P}_n} \left[ \mathbb{I}(L^\prime (\langle\bm{X},\beta^\ast\rangle,Y)=0) \left\vert \nabla_\beta  L^\prime (\langle \bm{X},\beta^\ast\rangle,Y)^\top \bm{u}\right\vert\right.\\
&\left. +  \mathbb{I}(L^\prime (\langle\bm{X},\beta^\ast\rangle,Y)\not=0 )\sign(L^\prime (\langle \bm{X},\beta^\ast\rangle,Y)) \nabla_\beta  L^\prime (\langle \bm{X},\beta^\ast\rangle,Y)^\top \bm{u}\right]+o(1), 
\end{aligned}\end{equation}
which indicates that 
\begin{equation}
\begin{aligned}\label{taylorexpansion3}
&\sqrt{n}\left(\mathbb{E}_{\mathbb{P}_n}\left[ \left\vert L^\prime(\langle\bm{X},\beta^\ast+\frac{1}{\sqrt{n}}\bm{u}\rangle,Y)\right\vert\right]-\mathbb{E}_{\mathbb{P}_n}\left[ \left\vert L^\prime(\langle\bm{X},\beta^\ast\rangle,Y)\right\vert\right]\right)\\
\to_p&\mathbb{E}_{P_\ast} \left[ \mathbb{I}(L^\prime (\langle\bm{X},\beta^\ast\rangle,Y)=0) \left\vert \nabla_\beta  L^\prime (\langle \bm{X},\beta^\ast\rangle,Y)^\top \bm{u}\right\vert\right.\\
&\left. +  \mathbb{I}(L^\prime (\langle\bm{X},\beta^\ast\rangle,Y)\not=0 )\sign(L^\prime (\langle \bm{X},\beta^\ast\rangle,Y)) \nabla_\beta  L^\prime (\langle \bm{X},\beta^\ast\rangle,Y)^\top \bm{u}\right]\\
\overset{(a)}{=}&\mathbb{E}_{P_\ast}\left[\sign(L^\prime (\langle \bm{X},\beta^\ast\rangle,Y)) \nabla_\beta  L^\prime (\langle \bm{X},\beta^\ast\rangle,Y)\right]^\top \bm{u},
\end{aligned}\end{equation}
as $n\to\infty$, where (a) comes from our assumption $P_\ast\left( L^\prime (\langle \bm{X},\beta^\ast\rangle,Y))=0\right)=0$.

Combining \eqref{decompose}-\eqref{taylorexpansion3}, we could obtain that, for the second term in \eqref{defofv}, the following holds
\begin{equation}\label{secondterm}
\begin{aligned}
&\sqrt{n}\left(\left\Vert\beta^\ast+\frac{1}{\sqrt{n}}\bm{u}\right\Vert_1\mathbb{E}_{\mathbb{P}_n}\left[ \left\vert L^\prime(\langle\bm{X},\beta^\ast+\frac{1}{\sqrt{n}}\bm{u}\rangle,Y)\right\vert\right]-\left\Vert\beta^\ast\right\Vert_1\mathbb{E}_{\mathbb{P}_n}\left[ \left\vert L^\prime(\langle\bm{X},\beta^\ast\rangle,Y)\right\vert\right]\right)\\
\to_p& \left(\sum_{j=1}^d\left([\bm{u}]_j \sign([\beta^\ast]_j) I([\beta^\ast]_j\not=0)+\vert [\bm{u}]_j\vert I([\beta^\ast]_j=0)\right)\right)\mathbb{E}_{P_\ast}\left[ \left\vert L^\prime(\langle\bm{X},\beta^\ast\rangle,Y)\right\vert\right]\\
&+\Vert\beta^\ast\Vert_1\mathbb{E}_{P_\ast}\left[\sign(L^\prime (\langle \bm{X},\beta^\ast\rangle,Y)) \nabla_\beta  L^\prime (\langle \bm{X},\beta^\ast\rangle,Y)\right]^\top \bm{u},
\end{aligned}
\end{equation}
as $n\to\infty$.

\textbf{For the third term in \eqref{defofv}}, we have that
\begin{equation}\label{third}
   \eta^2 \left(\frac{1}{2}\left\Vert\beta^\ast+\frac{1}{\sqrt{n}}\bm{u}\right\Vert_1^2 \mathbb{E}_{\mathbb{P}_n}\left[L^{\prime\prime}(\langle \bm{X},\beta^\ast+\frac{1}{\sqrt{n}}\bm{u}\rangle,Y)\right]-\frac{1}{2}\left\Vert\beta^\ast\right\Vert_1^2 \mathbb{E}_{\mathbb{P}_n}\left[L^{\prime\prime}(\langle \bm{X},\beta^\ast\rangle,Y)\right]\right)\to_p0,
\end{equation}
as $n\to\infty$.

 If we denote the term $\mathbb{E}_{P_\ast}\left[\sign(L^\prime (\langle \bm{X},\beta^\ast\rangle,Y)) \nabla_\beta  L^\prime (\langle \bm{X},\beta^\ast\rangle,Y)\right]$ by $K$ and the term $\mathbb{E}_{P_\ast}\left[ \left\vert L^\prime(\langle\bm{X},\beta^\ast\rangle,Y)\right\vert\right]$ by $r$, it follows from \eqref{firstequation}, \eqref{secondterm} and \eqref{third} that 
\[ V_n(\bm{u})\Rightarrow V(\bm{u}),\]
as $n\to\infty$, where 
\begin{equation*}
\begin{aligned}
V(\bm{u})=&(-\mathbf{G}+\eta\Vert\beta^\ast\Vert_1 K)^\top \bm{u}+\frac{1}{2}\bm{u}^\top H\bm{u} +\\
&\eta\left(\sum_{j=1}^d\left( [\bm{u}]_j \sign([\beta^\ast]_j) I([\beta^\ast]_j\not=0)+\vert [\bm{u}]_j\vert I([\beta^\ast]_j=0)\right)\right)r. 
\end{aligned}
\end{equation*}

Lemma \ref{convexity} shows $\Psi_n(\beta)$ is convex, demonstrating $V_n(\bm{u})$ is convex. Since $H$ is positive definite, $V(\bm{u})$ has a unique minimizer. It follows from \citet{geyer1996asymptotics,fu2000asymptotics,kato2009asymptotics} that \[ \sqrt{n}\left( \beta^n-\beta^\ast\right)\Rightarrow \arg\min_{\bm{u}} V(\bm{u}),\]
as $n\to\infty$.

\textbf{CASE 2: $\bm{\gamma>1/2}$}

In this case, we have that 
\begin{equation*}
\begin{aligned}
&V_n(\bm{u})\\
\overset{\vartriangle}{=}&n\left(\Psi_n(\beta^\ast+\frac{1}{\sqrt{n}}\bm{u})-\Psi_n(\beta^\ast)\right)\\
=&n\left( \mathbb{E}_{\mathbb{P}_n}\left[ L(\langle\bm{X},\beta^\ast+\frac{1}{\sqrt{n}}\bm{u}\rangle,Y)\right]-\mathbb{E}_{\mathbb{P}_n}\left[ L(\langle\bm{X},\beta^\ast\rangle,Y)\right]\right)\\
&+\eta n^{1/2-\gamma}\sqrt{n}\left(\left\Vert\beta^\ast+\frac{1}{\sqrt{n}}\bm{u}\right\Vert_1\mathbb{E}_{\mathbb{P}_n}\left[ \left\vert L^\prime(\langle\bm{X},\beta^\ast+\frac{1}{\sqrt{n}}\bm{u}\rangle,Y)\right\vert\right]-\left\Vert\beta^\ast\right\Vert_1\mathbb{E}_{\mathbb{P}_n}\left[ \left\vert L^\prime(\langle\bm{X},\beta^\ast\rangle,Y)\right\vert\right]\right)\\
&+\eta^2 \left(\frac{1}{2}\left\Vert\beta^\ast+\frac{1}{\sqrt{n}}\bm{u}\right\Vert_1^2 \mathbb{E}_{\mathbb{P}_n}\left[L^{\prime\prime}(\langle \bm{X},\beta^\ast+\frac{1}{\sqrt{n}}\bm{u}\rangle,Y)\right]-\frac{1}{2}\left\Vert\beta^\ast\right\Vert_1^2 \mathbb{E}_{\mathbb{P}_n}\left[L^{\prime\prime}(\langle \bm{X},\beta^\ast\rangle,Y)\right]\right)+o(1),
\end{aligned}
\end{equation*}
as $n\to\infty$.

In this way, due to $\gamma>1/2$, we have that 
\begin{equation*}\begin{aligned}
    \eta n^{1/2-\gamma}\sqrt{n}\left(\left\Vert\beta^\ast+\frac{1}{\sqrt{n}}\bm{u}\right\Vert_1\mathbb{E}_{\mathbb{P}_n}\left[ \left\vert L^\prime(\langle\bm{X},\beta^\ast+\frac{1}{\sqrt{n}}\bm{u}\rangle,Y)\right\vert\right]-\left\Vert\beta^\ast\right\Vert_1\mathbb{E}_{\mathbb{P}_n}\left[ \left\vert L^\prime(\langle\bm{X},\beta^\ast\rangle,Y)\right\vert\right]\right)\\
\to_p 0,
\end{aligned}
\end{equation*}
as $n\to\infty$.

Similar to the proof of \textbf{CASE 1}, it follows from \eqref{firstequation} that
\[ \sqrt{n}\left( \beta^n-\beta^\ast\right)\Rightarrow \arg\min_{\bm{u}} V(\bm{u}),\]
as $n\to\infty$, where \[V(\bm{u})=-\mathbf{G}^\top \bm{u}+\frac{1}{2}\bm{u}^\top H\bm{u}. \]

\textbf{CASE 3: $\bm{\gamma<1/2}$}

In this case, we have that 
\begin{equation}\label{defofv2}
\begin{aligned}
&V_n(\bm{u})\\
\overset{\vartriangle}{=}&n^{2\gamma}\left(\Psi_n(\beta^\ast+\frac{1}{n^\gamma}\bm{u})-\Psi_n(\beta^\ast)\right)\\
=&n^{2\gamma}\left( \mathbb{E}_{\mathbb{P}_n}\left[ L(\langle\bm{X},\beta^\ast+\frac{1}{n^\gamma}\bm{u}\rangle,Y)\right]-\mathbb{E}_{\mathbb{P}_n}\left[ L(\langle\bm{X},\beta^\ast\rangle,Y)\right]\right)\\
&+\eta n^\gamma\left(\left\Vert\beta^\ast+\frac{1}{n^\gamma}\bm{u}\right\Vert_1\mathbb{E}_{\mathbb{P}_n}\left[ \left\vert L^\prime(\langle\bm{X},\beta^\ast+\frac{1}{n^\gamma}\bm{u}\rangle,Y)\right\vert\right]-\left\Vert\beta^\ast\right\Vert_1\mathbb{E}_{\mathbb{P}_n}\left[ \left\vert L^\prime(\langle\bm{X},\beta^\ast\rangle,Y)\right\vert\right]\right)\\
&+\eta^2 \left(\frac{1}{2}\left\Vert\beta^\ast+\frac{1}{\sqrt{n}}\bm{u}\right\Vert_1^2 \mathbb{E}_{\mathbb{P}_n}\left[L^{\prime\prime}(\langle \bm{X},\beta^\ast+\frac{1}{\sqrt{n}}\bm{u}\rangle,Y)\right]-\frac{1}{2}\left\Vert\beta^\ast\right\Vert_1^2 \mathbb{E}_{\mathbb{P}_n}\left[L^{\prime\prime}(\langle \bm{X},\beta^\ast\rangle,Y)\right]\right)+o(1),
\end{aligned}
\end{equation}
as $n\to\infty$.

It follows from Proposition 6 in \citet{blanchet2022confidence}  that
\begin{equation*}\begin{aligned}
    n^{2\gamma}\left( \mathbb{E}_{\mathbb{P}_n}\left[ L(\langle\bm{X},\beta^\ast+\frac{1}{n^\gamma}\bm{u}\rangle,Y)\right]-\mathbb{E}_{\mathbb{P}_n}\left[ L(\langle\bm{X},\beta^\ast\rangle,Y)\right]\right)
\Rightarrow \frac{1}{2}\bm{u}^\top H\bm{u},
\end{aligned}\end{equation*}
as $n\to\infty$.

The second term and the third term in    \eqref{defofv2} have the same limits as in \textbf{CASE 1}, so we could conclude that 
\[ \sqrt{n}\left( \beta^n-\beta^\ast\right)\Rightarrow \arg\min_{\bm{u}} V(\bm{u}),\]
as $n\to\infty$,
where \[V(\bm{u})= \eta\Vert\beta^\ast\Vert_1 K^\top \bm{u}+\frac{1}{2}\bm{u}^\top H\bm{u} +\eta\left(\sum_{j=1}^d\left( [\bm{u}]_j \sign([\beta^\ast]_j) I([\beta^\ast]_j\not=0)+\vert [\bm{u}]_j\vert I([\beta^\ast]_j=0)\right)\right) r. \]
\end{proof}
\subsection{Proof of Proposition \ref{positiveprob}}
\begin{proof}
 Since $[\beta^\ast]_1,...,[\beta^\ast]_s$ are nonzero components and $[\beta^\ast]_{s+1},...,[\beta^\ast]_d$ are zero components, we have that 
\[V(\bm{u})=(-\mathbf{G}+\eta\Vert\beta^\ast\Vert_1 K)^\top u+\frac{1}{2}\bm{u}^\top H\bm{u} +\eta\left(\sum_{j=1}^s [\bm{u}]_j \sign([\beta^\ast]_j) +\sum_{j=s+1}^d \vert [\bm{u}]_j\vert\right)  r.\]

Next, we rewrite 
\begin{equation*} 
\mathbf{G}=\left(\begin{array}{c}
    \mathbf{G}_1  \\
     \mathbf{G}_2 
\end{array}\right),\quad K=\left(\begin{array}{c}
    K_1  \\
    K_2 
\end{array}\right), \quad H=\left(\begin{array}{cc}
    H_{11}&H_{12}  \\
    H_{21}&H_{22}
\end{array}\right),\quad \bm{u}=\left(\begin{array}{c}
    \bm{u}_1  \\
    \bm{u}_2 
\end{array}\right),
\end{equation*}
where $\mathbf{G}_1,K_1,\bm{u}_1\in\mathbb{R}^s$, $\mathbf{G}_2,K_2,\bm{u}_2\in\mathbb{R}^{d-s}$, $H_{11}\in\mathbb{R}^{s\times s}$, $H_{22}\in\mathbb{R}^{(d-s)\times (d-s)}$, $H_{12}\in\mathbb{R}^{r\times (d-s)}$,
and $H_{21}\in\mathbb{R}^{(d-s)\times s}$.

Notice that the probability of the asymptotic distribution $(\beta^n)_2$ lies at $\bm{0}$ equals to the probability of $V(\bm{u})$ being minimized at $\bm{u}_2=\bm{0}$. It suffices to show that $V(\bm{u})$ is minimized at $\bm{u}_2=\bm{0}$ with a positive probability.

If $V(\bm{u})$ is minimized at $\bm{u}_2=\bm{0}$, it follow from the first-order condition that
\[ \partial V(\bm{u})=\bm{0}.\]

That is to say,
\[ -\mathbf{G}_1+\eta\Vert\beta^\ast\Vert_1 K_1+H_{11}\bm{u}_1=-\eta r \mathrm{sign}(\beta^\ast),\]
and 
\[-\eta r\bm{1} \leq -\mathbf{G}_2+\eta\Vert\beta^\ast\Vert_1 K_2+H_{21}\bm{u}_1\leq \eta r\bm{1}.\]

In this sense, we have that 
\[ \bm{u}_1 = H_{11}^{-1}\left(\mathbf{G}_1-\eta\Vert\beta^\ast\Vert_1 K_1-\eta r  \mathrm{sign}(\beta^\ast)\right),\]
\[ -\eta r\bm{1} \leq H_{21}H_{11}^{-1}\left( \mathbf{G}_1-\eta\Vert\beta^\ast\Vert_1 K_1-\eta r \mathrm{sign}(\beta^\ast)\right)-\mathbf{G}_2+\eta\Vert\beta^\ast\Vert_1 K_2\leq \eta r\bm{1}.\]

In conclusion, the probability of $V(\bm{u})$ being minimized at $\bm{u}_2=\bm{0}$ is equal to the probability of the inequality
\[ -\eta r\bm{1} \leq H_{21}H_{11}^{-1}\left( \mathbf{G}_1-\eta \Vert\beta^\ast\Vert_1K_1-\eta r \mathrm{sign}(\beta^\ast)\right)-\mathbf{G}_2+\eta \Vert\beta^\ast\Vert_1K_2\leq \eta r\bm{1}\]
holds. Since $\mathbf{G}_1$ and $\mathbf{G}_2$ are normal random variables, $V(\bm{u})$ is minimized at $\bm{u}_2=\bm{0}$ with a positive probability.
\end{proof}
\subsection{Proof of Proposition \ref{prop2}}
\begin{proof}
Notably, we have that     \begin{equation*}
    \begin{aligned}&\max_{\Vert \bm{w}\otimes\Delta\Vert\leq 1}\nabla h(\bm{z})^\top \Delta\\ 
=&\max_{\Vert\Delta^\prime\Vert\leq 1} \nabla h(\bm{z}) ^\top(\bm{w}^{-1}\otimes\Delta^\prime)\\
=&\max_{\Vert\Delta^\prime\Vert\leq 1} (\bm{w}^{-1}\otimes\nabla h(\bm{z})) ^\top{\Delta^\prime}\\
=&\Vert  \bm{w}^{-1}\otimes\nabla h(\bm{z})\Vert_\ast.\end{aligned}\end{equation*}



In this way, it follows from the Taylor expansion shown in Proposition \ref{prop1} and the uniform continuity of $\nabla h$ that
\[\mathbb{E}_{\bm{Z}\sim P}\left[\max _{\Vert \bm{w}\otimes\Delta\Vert\leq \delta}h(\bm{Z}+\Delta)\right]=\mathbb{E}_ {\bm{Z}\sim P}\left[ h(\bm{Z})\right]+\delta \mathbb{E}_{\bm{Z}\sim P}\left[ \Vert \bm{w}^{-1}\otimes\nabla h(\bm{Z})\Vert_\ast\right]+o(\delta)\]as $\delta\to 0$.\end{proof}

\subsection{Proof of Theorem \ref{unbiasedness}}
\subsubsection{A Lemma}
A precise regularization effect for adaptive adversarial training could be obtained immediately from Lemma \ref{reglemma1} and Proposition \ref{prop2}, and we state the result in the following lemma.
\begin{lemma}\label{regadaptivelemma}If the function $L(f,y)$ is twice differentiable w.r.t. the first argument, and $L^{\prime\prime}(f,y)$ is uniformly continuous w.r.t. the first argument, for $\bm{\omega}\in\mathbb{R}^d$ satisfying $[\bm{\omega}]_i\not = 0, 1\leq i\leq d$, we have that 
\begin{equation*}
\begin{aligned}
&\mathbb{E}_{\mathbb{P}_n}\left[\max _{\Vert\bm{w}\otimes\Delta\Vert_{\infty}\leq \delta}L(\langle \bm{X}+\Delta,\beta\rangle,Y)\right]\\
=&\mathbb{E}_{\mathbb{P}_n}\left[ L(\langle\bm{X},\beta\rangle,Y)\right]+\delta \Vert\bm{w}^{-1}\otimes\beta\Vert_1\mathbb{E}_{\mathbb{P}_n}\left[ \left\vert L^\prime(\langle\bm{X},\beta\rangle,Y)\right\vert\right]\\
&+\frac{1}{2}\delta^2 \Vert\bm{w}^{-1}\otimes\beta\Vert_1^2 \mathbb{E}_{\mathbb{P}_n}[L^{\prime\prime}(\langle \bm{X},\beta\rangle,Y)]+
    o(\delta^2),
\end{aligned}\end{equation*}
as $\delta\to 0$.

\end{lemma}
\subsubsection{Proof}
We begin to prove Theorem \ref{unbiasedness}.
\begin{proof}

\textbf{CASE 1: $\bm{1/2<\gamma<1}$}

Let \[\Phi_n(\beta)=\mathbb{E}_{\mathbb{P}_n}\left[\max _{\Vert \widehat{\beta}^n\otimes\Delta\Vert_\infty\leq \delta_n}L(\langle\bm{X}+\Delta,\beta\rangle,Y)\right].\]

We have that
\begin{equation}
\begin{aligned}\label{defofv3}
V_n(\bm{u})\overset{\vartriangle}{=}& n\left(\Phi_n(\beta^\ast+\frac{1}{\sqrt{n}}\bm{u})-\Phi_n(\beta^\ast)\right)\\
\overset{(a)}{=}&n\left( \mathbb{E}_{\mathbb{P}_n}\left[ L(\langle\bm{X},\beta^\ast+\frac{1}{\sqrt{n}}\bm{u}\rangle,Y)\right]-\mathbb{E}_{\mathbb{P}_n}\left[ L(\langle\bm{X},\beta^\ast\rangle,Y)\right]\right)\\
&+\eta n^{1-\gamma}\left(\left\Vert (\widehat{\beta}^n)^{-1}\otimes\left(\beta^\ast+\frac{1}{\sqrt{n}}\bm{u}\right)\right\Vert_1\mathbb{E}_{\mathbb{P}_n}\left[ \left\vert L^\prime(\langle\bm{X},\beta^\ast+\frac{1}{\sqrt{n}}\bm{u}\rangle,Y)\right\vert\right]\right.\\
&\left.-\left\Vert (\widehat{\beta}^n)^{-1}\otimes\beta^\ast\right\Vert_1\mathbb{E}_{\mathbb{P}_n}\left[ \left\vert L^\prime(\langle\bm{X},\beta^\ast\rangle,Y)\right\vert\right]\right)\\
&+\eta^2 \left(\frac{1}{2}\left\Vert(\widehat{\beta}^n)^{-1}\otimes\left(\beta^\ast+\frac{1}{\sqrt{n}}\bm{u}\right)\right\Vert_1^2 \mathbb{E}_{\mathbb{P}_n}\left[L^{\prime\prime}(\langle \bm{X},\beta^\ast+\frac{1}{\sqrt{n}}\bm{u}\rangle,Y)\right]\right.\\
&\left.-\frac{1}{2}\left\Vert(\widehat{\beta}^n)^{-1}\otimes\beta^\ast\right\Vert_1^2 \mathbb{E}_{\mathbb{P}_n}\left[L^{\prime\prime}(\langle \bm{X},\beta^\ast\rangle,Y)\right]\right)+o(1),
\end{aligned}
\end{equation}
as $n\to\infty$, where (a) comes from Lemma \ref{regadaptivelemma}.

\textbf{We analyze the terms in \eqref{defofv3} as follows.}

\textbf{For the first term in  \eqref{defofv3}}, due to the asymptotic theory of empirical risk minimization, we have that
\begin{equation}\label{firsterm3}n\left( \mathbb{E}_{\mathbb{P}_n}\left[ L(\langle\bm{X},\beta^\ast+\frac{1}{\sqrt{n}}\bm{u}\rangle,Y)\right]-\mathbb{E}_{\mathbb{P}_n}\left[ L(\langle\bm{X},\beta^\ast\rangle,Y)\right]\right)\Rightarrow -\mathbf{G}^\top \bm{u}+\frac{1}{2}\bm{u}^\top H\bm{u}, \end{equation}
as $n\to\infty$, where $\mathbf{G}\sim \mathcal{N}(\bm{0},\Sigma)$ with covariance matrix $\Sigma=\Cov_{P_\ast}\left(\nabla_\beta L(\langle \bm{X},\beta^\ast\rangle,Y)\right)$ and $H=
\mathbb{E}_{P_\ast}\left[ \nabla_{\beta}^2 L(\langle \bm{X},\beta^\ast\rangle,Y)\right]$.

\textbf{For the second term in \eqref{defofv3}}, we have the following decomposition
\begin{equation}\label{decompose3}
\begin{aligned}
& n^{1-\gamma}\left(\left\Vert(\widehat{\beta}^n)^{-1}\otimes\left(\beta^\ast+\frac{1}{\sqrt{n}}\bm{u}\right)\right\Vert_1\mathbb{E}_{\mathbb{P}_n}\left[ \left\vert L^\prime(\langle\bm{X},\beta^\ast+\frac{1}{\sqrt{n}}\bm{u}\rangle,Y)\right\vert\right]\right.\\
&\left.-\left\Vert(\widehat{\beta}^n)^{-1}\otimes\beta^\ast\right\Vert_1\mathbb{E}_{\mathbb{P}_n}\left[ \left\vert L^\prime(\langle\bm{X},\beta^\ast\rangle,Y)\right\vert\right]\right)\\
=&n^{1-\gamma}\left(\left\Vert(\widehat{\beta}^n)^{-1}\otimes\left(\beta^\ast+\frac{1}{\sqrt{n}}\bm{u}\right)\right\Vert_1-\left\Vert(\widehat{\beta}^n)^{-1}\otimes\beta^\ast\right\Vert_1\right)\mathbb{E}_{\mathbb{P}_n}\left[ \left\vert L^\prime(\langle\bm{X},\beta^\ast+\frac{1}{\sqrt{n}}\bm{u}\rangle,Y)\right\vert\right]\\
&+n^{1-\gamma}\left\Vert(\widehat{\beta}^n)^{-1}\otimes\beta^\ast\right\Vert_1\left(\mathbb{E}_{\mathbb{P}_n}\left[ \left\vert L^\prime(\langle\bm{X},\beta^\ast+\frac{1}{\sqrt{n}}\bm{u}\rangle,Y)\right\vert\right]-\mathbb{E}_{\mathbb{P}_n}\left[ \left\vert L^\prime(\langle\bm{X},\beta^\ast\rangle,Y)\right\vert\right]\right).
\end{aligned}
\end{equation}

Now we analyze three terms in \eqref{decompose3}.
For the second term in    \eqref{decompose3}, since $1/2-\gamma<0$, we could obtain that 
\begin{equation*}
\begin{aligned}
    n^{1/2-\gamma}\sqrt{n}\left\Vert(\widehat{\beta}^n)^{-1}\otimes\beta^\ast\right\Vert_1\left(\mathbb{E}_{\mathbb{P}_n}\left[ \left\vert L^\prime(\langle\bm{X},\beta^\ast+\frac{1}{\sqrt{n}}\bm{u}\rangle,Y)\right\vert\right]-\mathbb{E}_{\mathbb{P}_n}\left[ \left\vert L^\prime(\langle\bm{X},\beta^\ast\rangle,Y)\right\vert\right]\right)\\\to_p 0,
    \end{aligned}\end{equation*}
    as $n\to\infty$.

For the first term in    \eqref{decompose3}, we have that \begin{equation}\begin{aligned}
    &n^{1-\gamma}\left(\left\Vert(\widehat{\beta}^n)^{-1}\otimes\left(\beta^\ast+\frac{1}{\sqrt{n}}\bm{u}\right)\right\Vert_1-\left\Vert(\widehat{\beta}^n)^{-1}\otimes\beta^\ast\right\Vert_1\right)\\
    &=\sum_{j=1}^{d}n^{1-\gamma}\vert[\widehat{\beta}^n]_j\vert^{-1}\left(\left\vert[\beta^\ast]_j+\frac{1}{\sqrt{n}}[\bm{u}]_j\right\vert-\vert [\beta^\ast]_j\vert\right).\end{aligned}\end{equation}
    For $[\beta^\ast]_j\not=0$, since we have that $\vert [\widehat{\beta}^n]_j\vert^{-1}\to_p \vert[\beta^\ast]_j\vert^{-1},$
and
\[\sqrt{n}\left(\left\vert[\beta^\ast]_j+\frac{1}{\sqrt{n}}[\bm{u}]_j\right\vert-\vert [\beta^\ast]_j\vert\right)\to [\bm{u}]_j \sign([\beta^\ast]_j),\]
then it follows from $1/2-\gamma<0$ that 
\begin{equation*}\begin{aligned}
     &n^{1-\gamma}\vert[\widehat{\beta}^n]_j\vert^{-1}\left(\left\vert[\beta^\ast]_j+\frac{1}{\sqrt{n}}[\bm{u}]_j\right\vert-\vert [\beta^\ast]_j\vert\right)\\
     &=n^{1/2-\gamma}  \vert [\widehat{\beta}^n]_j\vert^{-1}\sqrt{n}\left(\left\vert[\beta^\ast]_j+\frac{1}{\sqrt{n}}[\bm{u}]_j\right\vert-\vert [\beta^\ast]_j\vert\right)\to_p 0,\end{aligned}\end{equation*}
as $n\to\infty$. For $[\beta^\ast]_j=0$, we have that 
\[\sqrt{n}\left(\left\vert[\beta^\ast]_j+\frac{1}{\sqrt{n}}[\bm{u}]_j\right\vert-\vert [\beta^\ast]_j\vert\right)\to\vert[\bm{u}]_j\vert,\]
and $\sqrt{n}[\widehat{\beta}^n]_j=O_p(1),$ as $n\to\infty$,
which comes from the asymptotic normality of $\widehat{\beta}^n$.

In this case, it follows from $1-\gamma>0$ that
\begin{equation} 
\begin{aligned}\label{twocases}
&n^{1-\gamma}\vert[\widehat{\beta}^n]_j\vert^{-1}\left(\left\vert[\beta^\ast]_j+\frac{1}{\sqrt{n}}[\bm{u}]_j\right\vert-\vert [\beta^\ast]_j\vert\right)\\
=&n^{1-\gamma}  \vert \sqrt{n}[\widehat{\beta}^n]_j\vert^{-1}\sqrt{n}\left(\left\vert[\beta^\ast]_j+\frac{1}{\sqrt{n}}[\bm{u}]_j\right\vert-\vert [\beta^\ast]_j\vert\right)
\to_p\begin{cases}
    \infty,&[\bm{u}]_j\not=0\\
    0,&[\bm{u}]_j=0
\end{cases},
\end{aligned}
\end{equation}
as $n\to\infty$.

\textbf{For the third term in    \eqref{defofv3}},  we have that that
\begin{equation}\label{thirdadp}
\begin{aligned}
   &\eta^2 \left(\frac{1}{2}\left\Vert(\widehat{\beta}^n)^{-1}\otimes\left(\beta^\ast+\frac{1}{\sqrt{n}}\bm{u}\right)\right\Vert_1^2 \mathbb{E}_{\mathbb{P}_n}\left[L^{\prime\prime}(\langle \bm{X},\beta^\ast+\frac{1}{\sqrt{n}}\bm{u}\rangle,Y)\right]\right.\\
&\left.-\frac{1}{2}\left\Vert(\widehat{\beta}^n)^{-1}\otimes\beta^\ast\right\Vert_1^2 \mathbb{E}_{\mathbb{P}_n}\left[L^{\prime\prime}(\langle \bm{X},\beta^\ast\rangle,Y)\right]\right)\to_p0,
\end{aligned}
\end{equation}
as $n\to\infty$.

Combining \eqref{firsterm3}-\eqref{thirdadp}, we could obtain that 
\[ V_n(\bm{u})\Rightarrow V(\bm{u}),\]
as $n\to\infty$, where \[ V(\bm{u})=\begin{cases}
 -[\mathbf{G}]_{\mathcal{A}}^\top [\bm{u}]_{\mathcal{A}}+\frac{1}{2}[\bm{u}]_{\mathcal{A}}^\top H_{11}[\bm{u}]_{\mathcal{A}},& \text{if } [\bm{u}]_j=0, \forall j\not\in\mathcal{A}\\
 \infty,&\text{otherwise}
\end{cases}.\]

It is straightforward to check that $\Psi_n(\beta)$ is convex following steps in Lemma \ref{convexity}. 
In this way, we further have that  $V_n(\bm{u})$ is convex.
Also, notice that the unique minimizer of $V(\bm{u})$ is $( H_{11}^{-1}[\mathbf{G}]_{\mathcal{A}},0)^\top$. 
It follows from the epi-convergence results in \citet{geyer1994asymptotics,fu2000asymptotics,zou2006adaptive} that 
 \[ \sqrt{n}\left( [\widetilde{\beta}^n]_{\mathcal{A}}-[\beta^\ast]_{\mathcal{A}}\right)\Rightarrow H_{11}^{-1}[\mathbf{G}]_{\mathcal{A}},\quad \sqrt{n}\left( [\widetilde{\beta}^n]_{\mathcal{A}^c}-[\beta^\ast]_{\mathcal{A}^c}\right)\Rightarrow \bm{0},\]
 indicating 
 \[ \sqrt{n}[\widetilde{\beta}^n]_{\mathcal{A}^c}\to_p \bm{0},\]
as $n\to\infty$, where $\mathbf{G}\sim \mathcal{N}(\bm{0},\Sigma)$ with covariance matrix $\Sigma=\Cov_{P_\ast}(\nabla_\beta L(\langle \bm{X},\beta^\ast\rangle,Y))$ and $H=\mathbb{E}_{P_\ast}\left[ \nabla_{\beta}^2 L(\langle \bm{X},\beta^\ast\rangle,Y)\right]$.

 \textbf{CASE 2: $\bm{\gamma=1/2}$}

When $\gamma=1/2$, we also focus on the function defined in    \eqref{defofv3}. Compared with the proof for the case $1/2<\gamma <1$, the analysis differs from    \eqref{decompose3}.

For the second term in \eqref{decompose3},  it follows from \eqref{usedinadaptive} and \eqref{taylorexpansion3} that
\begin{equation*}
\begin{aligned}
    &\sqrt{n}\left\Vert(\widehat{\beta}^n)^{-1}\otimes\beta^\ast\right\Vert_1\left(\mathbb{E}_{\mathbb{P}_n}\left[ \left\vert L^\prime(\langle\bm{X},\beta^\ast+\frac{1}{\sqrt{n}}\bm{u}\rangle,Y)\right\vert\right]-\mathbb{E}_{\mathbb{P}_n}\left[ \left\vert L^\prime(\langle\bm{X},\beta^\ast\rangle,Y)\right\vert\right]\right)\\
    &=\left(\sum_{j\in \mathcal{A}} \left\vert\frac{[\beta^{\ast}]_j}{[\widehat{\beta}^n]_j}\right\vert\right) \sqrt{n}\left(\mathbb{E}_{\mathbb{P}_n}\left[ \left\vert L^\prime(\langle\bm{X},\beta^\ast+\frac{1}{\sqrt{n}}\bm{u}\rangle,Y)\right\vert\right]-\mathbb{E}_{\mathbb{P}_n}\left[ \left\vert L^\prime(\langle\bm{X},\beta^\ast\rangle,Y)\right\vert\right]\right)\\
    &\to_p s \mathbb{E}_{P_\ast}\left[\sign(L^\prime (\langle \bm{X},\beta^\ast\rangle,Y)) \nabla_\beta  L^\prime (\langle \bm{X},\beta^\ast\rangle,Y)\right]^\top \bm{u},
    \end{aligned}\end{equation*}
        as $n\to\infty$,
    where $s$ is the number of nonzero components in $\beta^\ast$.

For the first term in    \eqref{decompose3}, we have that \begin{equation*}\begin{aligned}
    &\sqrt{n}\left(\left\Vert(\widehat{\beta}^n)^{-1}\otimes\left(\beta^\ast+\frac{1}{\sqrt{n}}\bm{u}\right)\right\Vert_1-\left\Vert(\widehat{\beta}^n)^{-1}\otimes\beta^\ast\right\Vert_1\right)\\
    &=\sum_{j=1}^{d}\sqrt{n}\vert[\widehat{\beta}^n]_j\vert^{-1}\left(\left\vert[\beta^\ast]_j+\frac{1}{\sqrt{n}}[\bm{u}]_j\right\vert-\vert [\beta^\ast]_j\vert\right).\end{aligned}\end{equation*}
    For $[\beta^\ast]_j\not=0$, since we have that 
\[ \vert [\widehat{\beta}^n]_j\vert^{-1}\to_p \vert[\beta^\ast]_j\vert^{-1},\]
and
\[\sqrt{n}\left(\left\vert[\beta^\ast]_j+\frac{1}{\sqrt{n}}[\bm{u}]_j\right\vert-\vert [\beta^\ast]_j\vert\right)\to [\bm{u}]_j \sign([\beta^\ast]_j),\]
    as $n\to\infty$,
then it follows from $\gamma=1/2$ that 
\begin{equation*}\begin{aligned}
     &\sqrt{n}\vert[\widehat{\beta}^n]_j\vert^{-1}\left(\left\vert[\beta^\ast]_j+\frac{1}{\sqrt{n}}[\bm{u}]_j\right\vert-\vert [\beta^\ast]_j\vert\right)\\
     &= \vert [\widehat{\beta}^n]_j\vert^{-1}\sqrt{n}\left(\left\vert[\beta^\ast]_j+\frac{1}{\sqrt{n}}[\bm{u}]_j\right\vert-\vert [\beta^\ast]_j\vert\right)\to_p \sign([\beta^\ast]_j) \vert[\beta^\ast]_j\vert^{-1}[\bm{u}]_j=[\beta^\ast]_j^{-1}[\bm{u}]_j.\end{aligned}\end{equation*}
For $[\beta^\ast]_j=0$, we have that 
\[\sqrt{n}\left(\left\vert[\beta^\ast]_j+\frac{1}{\sqrt{n}}[\bm{u}]_j\right\vert-\vert [\beta^\ast]_j\vert\right)\to\vert[\bm{u}]_j\vert,\]
\[[\beta^\ast]_j\to_p 0,\]
as $n\to\infty$.

In this case, we have that
\begin{equation} 
\begin{aligned}\label{twocasesaat}
\sqrt{n}\vert[\widehat{\beta}^n]_j\vert^{-1}\left(\left\vert[\beta^\ast]_j+\frac{1}{\sqrt{n}}[\bm{u}]_j\right\vert-\vert [\beta^\ast]_j\vert\right)
\to_p\begin{cases}
    \infty,&[\bm{u}]_j\not=0\\
    0,&[\bm{u}]_j=0
\end{cases},
\end{aligned}
\end{equation}
as $n\to\infty$.

In this way, we could conclude that 
\[ V_n(\bm{u})\Rightarrow V(\bm{u}),\]
as $n\to\infty$, where \[ V(\bm{u})=\begin{cases}
 (-[\mathbf{G}]_{\mathcal{A}}+\eta r[{\beta^\ast}^{-1}]_{\mathcal{A}}+\eta s[K]_{\mathcal{A}})^\top [\bm{u}]_{\mathcal{A}}+\frac{1}{2}[\bm{u}]_{\mathcal{A}}^\top H_{11}[\bm{u}]_{\mathcal{A}},& \text{if } [\bm{u}]_j=0, \forall j\not\in\mathcal{A}\\
 \infty,&\text{otherwise}
\end{cases}.\]
where $r=\mathbb{E}_{P_\ast}\left[ \left\vert L^\prime(\langle\bm{X},\beta^\ast\rangle,Y)\right\vert\right], K=\mathbb{E}_{P_\ast}\left[\sign(L^\prime (\langle \bm{X},\beta^\ast\rangle,Y)) \nabla_\beta  L^\prime (\langle \bm{X},\beta^\ast\rangle,Y)\right]$.

Then, we could obtain the following:
 \[\sqrt{n}\left( [\widetilde{\beta}^n]_{\mathcal{A}}-[\beta^\ast]_{\mathcal{A}}\right)\Rightarrow H_{11}^{-1}(-\eta r[{\beta^\ast}^{-1}]_{\mathcal{A}}-\eta s[K]_{\mathcal{A}}+[\mathbf{G}]_{\mathcal{A}}),\]
 \[\sqrt{n}\left( [\widetilde{\beta}^n]_{\mathcal{A}^c}-[\beta^\ast]_{\mathcal{A}^c}\right)\Rightarrow \bm{0},\]
 indicating
 \[ \sqrt{n}[\widetilde{\beta}^n]_{\mathcal{A}^c}\to_p \bm{0},\]
 as $n\to\infty$.

 \textbf{CASE 3: $\bm{\gamma<1/2}$}
 
We have that 
 \begin{equation}\label{case3split}
\begin{aligned}
V_n(\bm{u})
\overset{\vartriangle}{=}&n^{2\gamma}\left(\Phi_n(\beta^\ast+\frac{1}{n^\gamma}\bm{u})-\Phi_n(\beta^\ast)\right)\\
=&n^{2\gamma}\left( \mathbb{E}_{\mathbb{P}_n}\left[ L(\langle\bm{X},\beta^\ast+\frac{1}{n^\gamma}\bm{u}\rangle,Y)\right]-\mathbb{E}_{\mathbb{P}_n}\left[ L(\langle\bm{X},\beta^\ast\rangle,Y)\right]\right)\\
&+\eta n^\gamma\left(\left\Vert(\widehat{\beta}^n)^{-1}\otimes\left(\beta^\ast+\frac{1}{n^\gamma}\bm{u}\right)\right\Vert_1\mathbb{E}_{\mathbb{P}_n}\left[ \left\vert L^\prime(\langle\bm{X},\beta^\ast+\frac{1}{n^\gamma}\bm{u}\rangle,Y)\right\vert\right]\right.\\
&\left.-\left\Vert(\widehat{\beta}^n)^{-1}\otimes\beta^\ast\right\Vert_1\mathbb{E}_{\mathbb{P}_n}\left[ \left\vert L^\prime(\langle\bm{X},\beta^\ast\rangle,Y)\right\vert\right]\right)\\
&+\eta^2 \left(\mathbb{E}_{\mathbb{P}_n}\left[P(\bm{X},\beta^\ast+\frac{1}{n^\gamma}\bm{u},Y)\right]-\mathbb{E}_{\mathbb{P}_n}[P(\bm{X},\beta^\ast,Y)]\right)+o(1),
\end{aligned}
\end{equation}
as $n\to\infty$.

\textbf{We analyze terms in \eqref{case3split} as follows.}

\textbf{For the first term} in    \eqref{case3split}, we have that 
\begin{equation*}\begin{aligned}
    n^{2\gamma}\left( \mathbb{E}_{\mathbb{P}_n}\left[ L(\langle\bm{X},\beta^\ast+\frac{1}{n^\gamma}\bm{u}\rangle,Y)\right]-\mathbb{E}_{\mathbb{P}_n}\left[ L(\langle\bm{X},\beta^\ast\rangle,Y)\right]\right)
\Rightarrow \frac{1}{2}\bm{u}^\top H\bm{u},
\end{aligned}\end{equation*}
as $n\to\infty$.

\textbf{The second term and the third term }in    \eqref{case3split} have the same limits as in \textbf{CASE 2}, 
In this way, we could conclude that 
\[ V_n(\bm{u})\Rightarrow V(\bm{u}),\]
where \[ V(\bm{u})=\begin{cases}
 (\eta r[{\beta^\ast}^{-1}]_{\mathcal{A}}+\eta s[K]_{\mathcal{A}})^\top [\bm{u}]_{\mathcal{A}}+\frac{1}{2}[\bm{u}]_{\mathcal{A}}^\top H_{11}[\bm{u}]_{\mathcal{A}},& \text{if } [\bm{u}]_j=0, \forall j\not\in\mathcal{A}\\
 \infty,&\text{otherwise}
\end{cases}.\]
Then, we could obtain the following:
 \[n^{\gamma}\left( [\widetilde{\beta}^n]_{\mathcal{A}}-[\beta^\ast]_{\mathcal{A}}\right)\Rightarrow H_{11}^{-1}(-\eta r[{\beta^\ast}^{-1}]_{\mathcal{A}}-\eta s[K]_{\mathcal{A}}),\]\[n^{\gamma}\left( [\widetilde{\beta}^n]_{\mathcal{A}^c}-[\beta^\ast]_{\mathcal{A}^c}\right)\Rightarrow \bm{0},\]
  indicating
 \[ n^{\gamma}[\widetilde{\beta}^n]_{\mathcal{A}^c}\to_p \bm{0},\]
 as $n\to\infty$.
\end{proof}

\subsection{Proof of Proposition \ref{inconsistency}}
\begin{proof}
$\mathcal{A}_n=\mathcal{A}$ implies $[\beta^n]_j=0, \forall j\not \in \mathcal{A}$, indicating \[P_\ast(\mathcal{A}_n=\mathcal{A})\leq P_\ast(\sqrt{n}[\beta^n]_j=0, \forall j \not\in \mathcal{A}).\] 

If we denote $\bm{u}^\ast$ as the minimizer of the function $V(\bm{u})$ defined in Theorem \ref{asyminfty}, then Theorem \ref{asyminfty} implies that \[\lim\inf_n P_\ast(\sqrt{n}[\beta^n]_j=0, \forall j \not\in \mathcal{A})\leq P_\ast([\bm{u}^\ast]_j=0,\forall j\not\in \mathcal{A}).\]

Thus, it suffices to show $c=P_\ast([\bm{u}^\ast]_j=0,\forall j\not\in \mathcal{A})<1$.

\textbf{CASE 1: $\bm{\gamma=1/2}$}

If $[\bm{u}^\ast]_j=0$ for all $j\not\in\mathcal{A}$,  it follows from the first-order condition that 
\[ -[\mathbf{G}]_{\mathcal{A}}+\eta\Vert\beta^\ast\Vert_1 [K]_{\mathcal{A}}+H_{11} [\bm{u}^\ast]_{\mathcal{A}}+\eta r \sign([\beta^\ast]_{\mathcal{A}})=\bm{0}, \]
and 
\[-\eta r \bm{1}\leq-[\mathbf{G}]_{\mathcal{A}^c}+\eta \Vert\beta^\ast\Vert_1[K]_{\mathcal{A}^c}+ H_{21} [\bm{u}^\ast]_{\mathcal{A}}\leq \eta r \bm{1}.\]

In this way, we have \begin{equation*}\begin{aligned}
    c&= P_\ast\left( -\eta r \bm{1}\leq-[\mathbf{G}]_{\mathcal{A}^c}+\eta\Vert\beta^\ast\Vert_1 [K]_{\mathcal{A}^c}\right.\\
    &\left.+ H_{21}H_{11}^{-1}([\mathbf{G}]_{\mathcal{A}}-\eta\Vert\beta^\ast\Vert_1 [K]_{\mathcal{A}}-\eta r \sign([\beta^\ast]_{\mathcal{A}}))  \leq \eta r \bm{1}\right)
    <1.\end{aligned}\end{equation*}

\textbf{CASE 2: $\bm{\gamma>1/2}$}

    It is straightforward to see $\bm{u}^\ast=H^{-1}\mathbf{G}$, so $c=0$.

\textbf{CASE 3: $\bm{0<\gamma<1/2}$}

It is straightforward to see $\bm{u}^\ast$ is a nonrandom quantity, so $c=0$.
  
\end{proof}

\subsection{Proof of Theorem \ref{consistency}}
\begin{proof}
\textbf{CASE 1: $\bm{1/2<\gamma<1}$}

We consider the following function
\begin{equation}\label{definitionofgamma}\Gamma(\bm{u})=\mathbb{E}_{\mathbb{P}_n}\left[\max _{\Vert\widehat{\beta}^n\otimes\Delta\Vert_\infty\leq \delta_n}L(\langle\bm{X}+\Delta,\beta^\ast+\frac{1}{\sqrt{n}}\bm{u}\rangle,Y)\right],\end{equation}
where $\Gamma(\bm{u})$ is minimized at $\sqrt{n}(\widetilde{\beta}^n-\beta^\ast)$.

It follows from Lemma \ref{reglemma1} that 
\begin{equation*}\begin{aligned}
    \Gamma(\bm{u}) =& \mathbb{E}_{\mathbb{P}_n}\left[ L(\langle\bm{X},\beta^\ast+\frac{1}{\sqrt{n}}\bm{u}\rangle,Y)\right]\\
&+\eta n^{-\gamma}\left\Vert(\widehat{\beta}^n)^{-1}\otimes\left(\beta^\ast+\frac{1}{\sqrt{n}}\bm{u}\right)\right\Vert_1\mathbb{E}_{\mathbb{P}_n}\left[ \left\vert L^\prime(\langle\bm{X},\beta^\ast+\frac{1}{\sqrt{n}}\bm{u}\rangle,Y)\right\vert\right]\\
&+\frac{1}{2}\eta^2 n^{-2\gamma}\left\Vert(\widehat{\beta}^n)^{-1}\otimes\left(\beta^\ast+\frac{1}{\sqrt{n}}\bm{u}\right)\right\Vert_1^2 \mathbb{E}_{\mathbb{P}_n}\left[L^{\prime\prime}(\langle\bm{X},\beta^\ast+\frac{1}{\sqrt{n}}\bm{u}\rangle,Y)\right]+o(n^{-2\gamma}), \end{aligned}\end{equation*}
as $n\to\infty$.

If $j\in\mathcal{A}$, Theorem \ref{unbiasedness} implies that $[\widetilde{\beta}^n]_j\to_p[\beta^\ast]_j$ as $n\to\infty$, indicating that $P_\ast(j\in \widetilde{\mathcal{A}}_n)\to 1$.

Thus, it suffices to show: if $j\not\in\mathcal{A}$, then 
$P_\ast(j\in\widetilde{\mathcal{A}}_n)\to 0$.

If $j\not\in\mathcal{A}, j\in\widetilde{\mathcal{A}}_n$, it follows from the the first-order condition of $\Gamma(\bm{u})$ that 
\begin{equation}\begin{aligned} \label{KKT}&\sqrt{n}\mathbb{E}_{\mathbb{P}_n}\left[ L^\prime(\langle\bm{X},\widetilde{\beta}^n\rangle,Y) [\bm{X}]_j\right]\\
&+\eta n^{1-\gamma}\left(\sum_{k\in\widetilde{\mathcal{A}}_n } \left\vert\frac{[\widetilde{\beta}^n]_k}{\sqrt{n} [\widehat{\beta}^n]_k}\right\vert\mathbb{E}_{\mathbb{P}_n} \left[ \sign(L^\prime (\langle \bm{X},\widetilde{\beta}^n\rangle,Y))  L^{\prime\prime} (\langle \bm{X},\widetilde{\beta}^n\rangle,Y)[\bm{X}]_j\right]\right.\\
&\left.+ \frac{1}{\vert\sqrt{n}[\widehat{\beta}^n]_j\vert}\sign ([\widetilde{\beta}^n]_j)
\mathbb{E}_{\mathbb{P}_n}\left[\left\vert L^\prime(\langle\bm{X},\widetilde{\beta}^n\rangle,Y)\right\vert\right]\right)+o(1)=0,\end{aligned}\end{equation}
as $n\to\infty$.

\textbf{We analyze the terms in  \eqref{KKT} as follows.}

\textbf{For the first term in  \eqref{KKT}}, we have that
\begin{equation}\begin{aligned}\label{expansionaa}
&\mathbb{E}_{\mathbb{P}_n}\left[ L^\prime(\langle\bm{X},\widetilde{\beta}^n\rangle,Y) [\bm{X}]_j\right]\\
=&\mathbb{E}_{\mathbb{P}_n}\left[ L^\prime(\langle\bm{X},\beta^\ast\rangle,Y) [\bm{X}]_j\right]+(\widetilde{\beta}^n-\beta^\ast)^\top\mathbb{E}_{\mathbb{P}_n}\left[ L^{\prime\prime}(\langle\bm{X},\beta^\ast\rangle,Y) [\bm{X}]_j\bm{X}\right]+o_p(\Vert \widetilde{\beta}^n-\beta^\ast\Vert_2),
\end{aligned}\end{equation}
as $n\to\infty$. It follows from the central limit theorem that
\begin{equation}\label{centrallimit}\sqrt{n}\mathbb{E}_{\mathbb{P}_n}\left[ L^\prime(\langle\bm{X},\beta^\ast\rangle,Y) [\bm{X}]_j\right]\Rightarrow \mathcal{N}\left(0, \Cov_{P_\ast} (L^\prime(\langle\bm{X},\beta^\ast\rangle,Y) [\bm{X}]_j)\right),\end{equation}
as $n\to\infty$. Due to  Theorem \ref{unbiasedness}, we know that $\sqrt{n}(\widetilde{\beta}^n-\beta^\ast)$ converges to some distribution, and \[\mathbb{E}_{\mathbb{P}_n}\left[ L^{\prime\prime}(\langle\bm{X},\beta^\ast\rangle,Y) [\bm{X}]_j\bm{X}\right]\to_p\mathbb{E}_{P_\ast}\left[ L^{\prime\prime}(\langle\bm{X},\beta^\ast\rangle,Y) [\bm{X}]_j\bm{X}\right],\] 
as $n\to\infty$. In this way, we have that $\sqrt{n}(\widetilde{\beta}^n-\beta^\ast)^\top\mathbb{E}_{\mathbb{P}_n}\left[ L^{\prime\prime}(\langle\bm{X},\beta^\ast\rangle,Y) [\bm{X}]_j\bm{X}\right]$ converges to some distribution as $n\to\infty$. Then, notice that $o_p(\sqrt{n}\Vert \widetilde{\beta}^n-\beta^\ast\Vert_2)=O_p(1)$. In conclusion, the first term in    \eqref{KKT} is bounded in probability.

\textbf{For second term in    \eqref{KKT}}, we have that $\gamma<1$ and $\sqrt{n}[\widehat{\beta}^n]_j=O_p(1)$. In this way, the second term is not bounded in probability.

Thus, 
\[P_\ast(j\in\widetilde{\mathcal{A}}_n)\leq P_\ast( \text{the equation }\eqref{KKT} \text{ holds})\to 0.\]

In this way, we have that 
\[ \lim_n P_\ast(\widetilde{\mathcal{A}}_n=\mathcal{A})=1\]

\textbf{CASE 2: $\bm{\gamma=1/2}$}

We also focus on the function $\Gamma(\bm{u})$ defined in  \eqref{definitionofgamma}.
Notice we have that
\begin{equation*}\begin{aligned}
    \Gamma(\bm{u}) =& \mathbb{E}_{\mathbb{P}_n}\left[ L(\langle\bm{X},\beta^\ast+\frac{1}{\sqrt{n}}\bm{u}\rangle,Y)\right]\\
&+\eta \frac{1}{\sqrt{n}}\left\Vert(\widehat{\beta}^n)^{-1}\otimes\left(\beta^\ast+\frac{1}{\sqrt{n}}\bm{u}\right)\right\Vert_1\mathbb{E}_{\mathbb{P}_n}\left[ \left\vert L^\prime(\langle\bm{X},\beta^\ast+\frac{1}{\sqrt{n}}\bm{u}\rangle,Y)\right\vert\right]\\
&+\frac{1}{2}\eta^2 n^{-1}\left\Vert(\widehat{\beta}^n)^{-1}\otimes\left(\beta^\ast+\frac{1}{\sqrt{n}}\bm{u}\right)\right\Vert_1^2 \mathbb{E}_{\mathbb{P}_n}\left[L^{\prime\prime}(\langle\bm{X},\beta^\ast+\frac{1}{\sqrt{n}}\bm{u}\rangle,Y)\right]+o(n^{-1}),\end{aligned}\end{equation*}
as $n\to\infty$.

If $j\in\mathcal{A}$, Theorem \ref{unbiasedness} implies that $[\widetilde{\beta}^n]_j\to_p[\beta^\ast]_j$, indicating that $P_\ast(j\in \widetilde{\mathcal{A}}_n)\to 1$.

Thus, it suffices to show: if $j\not\in\mathcal{A}$, then 
$P_\ast(j\in\widetilde{\mathcal{A}}_n)\to 0$.

Similar to \textbf{CASE 1}, when $j\not\in\mathcal{A}, j\in\widetilde{\mathcal{A}}_n$, we check the first-order condition as follows,
\begin{equation}\begin{aligned} \label{KKT2}\sqrt{n}&\mathbb{E}_{\mathbb{P}_n}\left[ L^\prime(\langle\bm{X},\widetilde{\beta}^n\rangle,Y) [\bm{X}]_j\right]\\
&+\eta\left(\sum_{k\in\widetilde{\mathcal{A}}_n } \left\vert\frac{[\widetilde{\beta}^n]_k}{ [\widehat{\beta}^n]_k}\right\vert\mathbb{E}_{\mathbb{P}_n} \left[ \sign(L^\prime (\langle \bm{X},\widetilde{\beta}^n\rangle,Y))  L^{\prime\prime} (\langle \bm{X},\widetilde{\beta}^n\rangle,Y)[\bm{X}]_j\right]\right.\\
&\left.+ \frac{1}{\vert[\widehat{\beta}^n]_j\vert} \sign ([\widetilde{\beta}^n]_j)
\mathbb{E}_{\mathbb{P}_n}\left[\left\vert L^\prime(\langle\bm{X},\widetilde{\beta}^n\rangle,Y)\right\vert\right]\right)+o(1)=0,\end{aligned}\end{equation}
as $n\to\infty$.

Since $[\widehat{\beta}^n]_j\to_p 0$ as $n\to\infty$, we know the second term of  \eqref{KKT2} is not bounded in probability. It follows the calculations in \textbf{CASE 1} that the first term is bounded in probability. Thus, the same conclusion could be made, i.e., 
\[ \lim_n P_\ast(\widetilde{\mathcal{A}}_n=\mathcal{A})=1.\]

\textbf{CASE 3: $\bm{0<\gamma<1/2}$}

We consider the following function
\begin{equation}\label{definitionofgamma2}\Gamma(\bm{u})=\mathbb{E}_{\mathbb{P}_n}\left[\max _{\Vert\widehat{\beta}^n\otimes\Delta\Vert_\infty\leq \delta_n}L(\langle\bm{X}+\Delta,\beta^\ast+\frac{1}{n^{\gamma}}\bm{u}\rangle,Y)\right],\end{equation}
where $\Gamma(\bm{u})$ is minimized at $n^{\gamma}(\widetilde{\beta}^n-\beta^\ast)$.

It follows from Lemma \ref{reglemma1} that 
\begin{equation*}\begin{aligned}
    \Gamma(\bm{u}) =& \mathbb{E}_{\mathbb{P}_n}\left[ L(\langle\bm{X},\beta^\ast+\frac{1}{n^\gamma}\bm{u}\rangle,Y)\right]\\
&+\eta n^{-\gamma}\left\Vert(\widehat{\beta}^n)^{-1}\otimes\left(\beta^\ast+\frac{1}{n^\gamma}\bm{u}\right)\right\Vert_1\mathbb{E}_{\mathbb{P}_n}\left[ \left\vert L^\prime(\langle\bm{X},\beta^\ast+\frac{1}{n^\gamma}\bm{u}\rangle,Y)\right\vert\right]\\
&+\frac{1}{2}\eta^2 n^{-2\gamma}\left\Vert(\widehat{\beta}^n)^{-1}\otimes\left(\beta^\ast+\frac{1}{n^\gamma}\bm{u}\right)\right\Vert_1^2 \mathbb{E}_{\mathbb{P}_n}\left[L^{\prime\prime}(\langle\bm{X},\beta^\ast+\frac{1}{n^\gamma}\bm{u}\rangle,Y)\right]+o(n^{-2\gamma}), \end{aligned}\end{equation*}
as $n\to\infty$.

If $j\in\mathcal{A}$, Theorem \ref{unbiasedness} implies that $[\widetilde{\beta}^n]_j\to_p[\beta^\ast]_j$, indicating that $P_\ast(j\in \widetilde{\mathcal{A}}_n)\to 1$.

Thus, it suffices to show: if $j\not\in\mathcal{A}$, then 
$P_\ast(j\in\widetilde{\mathcal{A}}_n)\to 0$.

Similar to \textbf{CASE 1} and \textbf{CASE 2}, when $j\not\in\mathcal{A}, j\in\widetilde{\mathcal{A}}_n$, we check the first-order condition as follows,

\begin{equation}\begin{aligned} \label{KKT3}n^{\gamma}&\mathbb{E}_{\mathbb{P}_n}\left[ L^\prime(\langle\bm{X},\widetilde{\beta}^n\rangle,Y) [\bm{X}]_j\right]\\
&+\eta \left(\sum_{k\in\widetilde{\mathcal{A}}_n } \left\vert\frac{[\widetilde{\beta}^n]_k}{ [\widehat{\beta}^n]_k}\right\vert\mathbb{E}_{\mathbb{P}_n} \left[ \sign(L^\prime (\langle \bm{X},\widetilde{\beta}^n\rangle,Y))  L^{\prime\prime} (\langle \bm{X},\widetilde{\beta}^n\rangle,Y)[\bm{X}]_j\right]\right.\\
&\left.+ \frac{1}{\vert[\widehat{\beta}^n]_j\vert}
\mathbb{E}_{\mathbb{P}_n}\left[\left\vert L^\prime(\langle\bm{X},\widetilde{\beta}^n\rangle,Y)\right\vert\right]\right)+o(1)=0,\end{aligned}\end{equation}
as $n\to\infty$.

Recall that we have the expansion as shown \eqref{expansionaa} for the first term in \eqref{KKT3}. It follows from \eqref{centrallimit} that $n^{\gamma}\mathbb{E}_{\mathbb{P}_n}\left[ L^\prime(\langle\bm{X},\beta^\ast\rangle,Y) [\bm{X}]_j\right]$ converges to zero in probability.
Due to  Theorem \ref{unbiasedness}, we know that $n^{\gamma}(\widetilde{\beta}^n-\beta^\ast)$ converges in distribution.
In this way, we have that $n^{\gamma}(\widetilde{\beta}^n-\beta^\ast)^\top\mathbb{E}_{\mathbb{P}_n}\left[ L^{\prime\prime}(\langle\bm{X},\beta^\ast\rangle,Y) [\bm{X}]_j\bm{X}\right]$ converges in distribution.
Then, notice that $o_p(n^{\gamma}\Vert \widetilde{\beta}^n-\beta^\ast\Vert_2)=O_p(1)$. In conclusion, the first term in \eqref{KKT3} is bounded in probability.
For the second term in    \eqref{KKT3}, the analysis is the same as that of \textbf{CASE 2}. Thus, the same conclusion could be made, i.e., 
\[ \lim_n P_\ast(\widetilde{\mathcal{A}}_n=\mathcal{A})=1.\]

\end{proof}
\subsection{Proof of Corollary \ref{asymcorolaarylinear}}
\begin{proof}
We have that
\[\Sigma=\Cov_{P_\ast}(\nabla_\beta L(\langle \bm{X},\beta^\ast\rangle,Y))=\Cov_{P_\ast}\left(\bm{X}(\langle\bm{X},\beta^\ast\rangle-Y)\right)=\sigma^2 \mathbb{E}_{P_\ast}[\bm{X}\bm{X}^\top]=\sigma^2   \bm{I}_d,\]
\[H=\mathbb{E}_{P_\ast}\left[\nabla_{\beta}^2 L(\langle \bm{X},\beta^\ast\rangle,Y)\right]=\mathbb{E}_{P_\ast}[\bm{X}\bm{X}^\top]=\bm{I}_d,\]
\[r=\mathbb{E}_{P_\ast}\left[\left\vert L^\prime (\langle \bm{X},\beta^\ast\rangle,Y)\right\vert\right]=\mathbb{E}_{P_\ast}\left[\left\vert \langle\bm{X},\beta^\ast\rangle-Y\right\vert\right]=\sigma\sqrt{\frac{2}{\pi}},\]
\begin{equation*}\begin{aligned}
    K&=\mathbb{E}_{P_\ast}\left[\sign(L^\prime (\langle \bm{X},\beta^\ast\rangle,Y)) \nabla_\beta  L^\prime (\langle \bm{X},\beta^\ast\rangle,Y)\right]\\
    &=\mathbb{E}_{P_\ast}\left[\sign(\langle\bm{X},\beta^\ast\rangle-Y) \bm{X}\right]=\bm{0}.\end{aligned}\end{equation*}

The results are obtained by plugging the expressions above into Theorem \ref{unbiasedness}.
\end{proof}
\subsection{Proof of Corollary \ref{tractable2}}
\begin{proof}
    Motivated by the proof of Theorem 4 in \citet{ribeiro2023regularization}, we define the Fenchel conjugate of $L$ as
    \[ L^\ast(u,y)\overset{\vartriangle}{=}\max_{z}\{uz-L(z,y)\}.\]

    It follows from the Fenchel-Moreau theorem that 
    \begin{equation*}
    \begin{aligned}
        \max_{\Vert\bm{w}\otimes\Delta\Vert\leq \delta} L(\langle \bm{x}+\Delta,\beta\rangle,y)&=\max_{\Vert\bm{w}\otimes\Delta\Vert\leq \delta} \max_{u}\left\{u\langle\bm{x},\beta\rangle+u\langle\Delta,\beta\rangle-L^\ast(u,y)\right\}\\
        &=\max_{u}\left\{u\langle\bm{x},\beta\rangle+\max_{\Vert\bm{w}\otimes\Delta\Vert\leq \delta}u\langle\Delta,\beta\rangle-L^\ast(u,y)\right\}\\
        &=\max_{u}\left\{u\langle\bm{x},\beta\rangle+\max_{\Vert\Delta\Vert\leq \delta}u\langle\Delta,\bm{w}^{-1}\otimes\beta\rangle-L^\ast(u,y)\right\}\\
        &=\max_{u}\left\{u\langle\bm{x},\beta\rangle+\max_{s=\left\{1,-1\right\}}\left\{su\delta\Vert\bm{w}^{-1}\otimes\beta\Vert_\ast\right\}-L^\ast(u,y)\right\}\\
        &=\max_{u}\max_{s=\left\{1,-1\right\}}\left\{u\langle\bm{x},\beta\rangle+su\delta\Vert\bm{w}^{-1}\otimes\beta\Vert_\ast-L^\ast(u,y)\right\}\\
        &=\max_{s=\left\{1,-1\right\}}\max_{u}\left\{u\left(\langle\bm{x},\beta\rangle+s\delta\Vert\bm{w}^{-1}\otimes\beta\Vert_\ast\right)-L^\ast(u,y)\right\}\\
        &=\max_{s\in\{-1,1\}}L\left(\langle \bm{x},\beta\rangle+\delta s\Vert\bm{w}^{-1}\otimes \beta\Vert_{\ast},y\right).
    \end{aligned}
    \end{equation*}
    \end{proof}
  
\section{Asymptotic Distribution of Adversarial Training Estimator under $\ell_p$-perturbation}\label{pperturbation}
This section will focus on the asymptotic distribution of adversarial training estimator under $\ell_p$-perturbation, where $p\in(1,\infty)$.
The adversarial training estimator under $\ell_p$-perturbation is defined by the following optimization problem:
\begin{equation}\label{advestimator2}\beta^n\in\arg\min_{\beta\in B}\mathbb{E}_{\mathbb{P}_n}\left[\max _{\Vert\Delta\Vert_p\leq \delta_n}L(\langle \bm{X}+\Delta,\beta\rangle,Y)\right],\end{equation}
where $p\in(1,\infty)$, $B$ is a convex compact subset of $\mathbb{R}^d$, and $\delta_n=\eta/n^\gamma$, $\eta,\gamma>0$.

We present the asymptotic distribution of the estimator $\beta^n$ along with its proof below.

\begin{theorem}[Asymptotic Distribution] \label{asymfinite}
Under Assumption \ref{assume1main}, for the adversarial training estimator $\beta^n$ defined in \eqref{advestimator2}, if the perturbation magnitude is chosen as $\delta_n=\eta/n^{\gamma}$, $\eta,\gamma>0$,
we have the following convergence as $n\to\infty$:
 \begin{itemize}[itemsep=-2ex]
        \item If $\gamma>1/2$, then 
        \begin{equation*} \sqrt{n}(\beta^n-\beta^\ast)\Rightarrow H^{-1}\mathbf{G};\end{equation*}
        \item If $\gamma=1/2$, then 
        \begin{equation*}\begin{aligned}
             \sqrt{n}(\beta^n-\beta^\ast)
             \Rightarrow H^{-1}\left(-\eta \nabla_{\beta}\left( \Vert\beta^\ast\Vert_q\mathbb{E}_{P_\ast}[\vert L^\prime (\langle \bm{X},\beta^\ast\rangle,Y)\vert]\right)+\mathbf{G}\right);\end{aligned}\end{equation*}
        \item If $0<\gamma<1/2$, then 
        \begin{equation*}\begin{aligned} n^{\gamma}(\beta^n-\beta^\ast)
        \Rightarrow&-\eta H^{-1}\nabla_{\beta}\left( \Vert\beta^\ast\Vert_q\mathbb{E}_{P_\ast}[\vert L^\prime (\langle \bm{X},\beta\rangle,Y)\vert]\right),\end{aligned}\end{equation*}
    \end{itemize}
    where $1/p+1/q=1$, $\mathbf{G}\sim \mathcal{N}(\bm{0},\Sigma)$ with the covariance matrix $\Sigma=\Cov_{P_\ast}(\nabla_\beta L(\langle \bm{X},\beta^\ast\rangle,Y))$, 
    and $H=\mathbb{E}_{P_\ast}\left[ \nabla_{\beta}^2L(\langle \bm{X},\beta^\ast\rangle,Y)\right]$. \end{theorem}
\begin{proof}
Lemma \ref{reglemma1} implies that
\begin{equation}
\begin{aligned}
&\mathbb{E}_{\mathbb{P}_n}\left[\max _{\Vert\Delta\Vert_{p}\leq \delta_n}L(\langle \bm{X}+\Delta,\beta\rangle,Y)\right]\\
=&\mathbb{E}_{\mathbb{P}_n}\left[ L(\langle\bm{X},\beta\rangle,Y)\right]+\delta_n \Vert\beta\Vert_q\mathbb{E}_{\mathbb{P}_n}\left[ \vert L^{\prime}(\langle \bm{X},\beta\rangle,Y)\vert\right]+\frac{1}{2}\delta^2_n\Vert\beta\Vert_q^2 \mathbb{E}_{\mathbb{P}_n}[L^{\prime\prime}(\langle \bm{X},\beta\rangle,Y)]+
    o(\delta^2_n),
\end{aligned}\end{equation}
as $\delta\to 0$.

The proof of Theorem \ref{asymfinite} follows a similar structure to that of Theorem \ref{asyminfty}. As an illustrative example, we present the case where $\gamma=1/2$.

We define the following function:
\[ \Psi_n(\beta)=\mathbb{E}_{\mathbb{P}_n}\left[\max _{\Vert\Delta\Vert_p\leq \delta_n}L(\langle\bm{X}+\Delta,\beta\rangle,Y)\right],\]
based on which we further have that 
\begin{equation}
\begin{aligned}\label{defofvfinte}
&V_n(\bm{u})\\
\overset{\vartriangle}{=}&n\left(\Psi_n(\beta^\ast+\frac{1}{\sqrt{n}}\bm{u})-\Psi_n(\beta^\ast)\right)\\
=&n\left( \mathbb{E}_{\mathbb{P}_n}\left[ L(\langle\bm{X},\beta^\ast+\frac{1}{\sqrt{n}}\bm{u}\rangle,Y)\right]-\mathbb{E}_{\mathbb{P}_n}\left[ L(\langle\bm{X},\beta^\ast\rangle,Y)\right]\right)\\
&+\eta\sqrt{n}\left(\left\Vert\beta^\ast+\frac{1}{\sqrt{n}}\bm{u}\right\Vert_q\mathbb{E}_{\mathbb{P}_n}\left[ \left\vert L^\prime(\langle\bm{X},\beta^\ast+\frac{1}{\sqrt{n}}\bm{u}\rangle,Y)\right\vert\right]-\left\Vert\beta^\ast\right\Vert_q\mathbb{E}_{\mathbb{P}_n}\left[ \left\vert L^\prime(\langle\bm{X},\beta^\ast\rangle,Y)\right\vert\right]\right)\\
&+\eta^2 \left(\frac{1}{2}\left\Vert\beta^\ast+\frac{1}{\sqrt{n}}\bm{u}\right\Vert_q^2 \mathbb{E}_{\mathbb{P}_n}\left[L^{\prime\prime}(\langle \bm{X},\beta^\ast+\frac{1}{\sqrt{n}}\bm{u}\rangle,Y)\right]-\frac{1}{2}\left\Vert\beta^\ast\right\Vert_q^2 \mathbb{E}_{\mathbb{P}_n}\left[L^{\prime\prime}(\langle \bm{X},\beta^\ast\rangle,Y)\right]\right)+o(1),
\end{aligned}
\end{equation}
as $n\to\infty$.

For the first term in \eqref{defofvfinte}, we have that
\begin{equation}\label{firstequation1}n\left( \mathbb{E}_{\mathbb{P}_n}\left[ L(\langle\bm{X},\beta^\ast+\frac{1}{\sqrt{n}}\bm{u}\rangle,Y)\right]-\mathbb{E}_{\mathbb{P}_n}\left[ L(\langle\bm{X},\beta^\ast\rangle,Y)\right]\right)\Rightarrow -\mathbf{G}^\top \bm{u}+\frac{1}{2}\bm{u}^\top H\bm{u}. \end{equation}

For the second term in  \eqref{defofvfinte}, we have that
\begin{equation}\label{decompose1}
\begin{aligned}
&\sqrt{n}\left(\left\Vert\beta^\ast+\frac{1}{\sqrt{n}}\bm{u}\right\Vert_q\mathbb{E}_{\mathbb{P}_n}\left[ \left\vert L^\prime(\langle\bm{X},\beta^\ast+\frac{1}{\sqrt{n}}\bm{u}\rangle,Y)\right\vert\right]-\left\Vert\beta^\ast\right\Vert_q\mathbb{E}_{\mathbb{P}_n}\left[ \left\vert L^\prime(\langle\bm{X},\beta^\ast\rangle,Y)\right\vert\right]\right)\\
=& \left(\nabla_{\beta}(\left\Vert\beta^\ast\right\Vert_q\mathbb{E}_{\mathbb{P}_n}\left[ \left\vert L^\prime(\langle\bm{X},\beta^\ast\rangle,Y)\right\vert\right])\right)^\top \bm{u}.
\end{aligned}
\end{equation}

For the third term in    \eqref{defofvfinte}, we have that
\begin{equation}\label{third1}
   \eta^2 \left(\frac{1}{2}\left\Vert\beta^\ast+\frac{1}{\sqrt{n}}\bm{u}\right\Vert_q^2 \mathbb{E}_{\mathbb{P}_n}\left[L^{\prime\prime}(\langle \bm{X},\beta^\ast+\frac{1}{\sqrt{n}}\bm{u}\rangle,Y)\right]-\frac{1}{2}\left\Vert\beta^\ast\right\Vert_q^2 \mathbb{E}_{\mathbb{P}_n}\left[L^{\prime\prime}(\langle \bm{X},\beta^\ast\rangle,Y)\right]\right)\to_p0,
\end{equation}
as $n\to\infty$.

It follows from \eqref{firstequation1}, \eqref{decompose1} and \eqref{third1} that 
\[ V_n(\bm{u})\Rightarrow V(\bm{u}),\]
as $n\to\infty$, where 
\begin{equation}
\begin{aligned}
V(\bm{u})=&(-\mathbf{G}+\eta\nabla_{\beta}(\left\Vert\beta^\ast\right\Vert_q\mathbb{E}_{\mathbb{P}_n}\left[ \left\vert L^\prime(\langle\bm{X},\beta^\ast\rangle,Y)\right\vert\right]))^\top \bm{u}+\frac{1}{2}\bm{u}^\top H\bm{u}.
\end{aligned}
\end{equation}

Then, we have that \[ \sqrt{n}\left( \beta^n-\beta^\ast\right)\Rightarrow H^{-1}\left(-\eta \nabla_{\beta}\left( \Vert\beta^\ast\Vert_q\mathbb{E}_{P_\ast}[\vert L^\prime (\langle \bm{X},\beta^\ast\rangle,Y)\vert]\right)+\mathbf{G}\right),\]
as $n\to\infty$.

When $0<\gamma<1/2$ and $\gamma>1/2$, the proof approach takes the same way, and we could obtain that 
        \begin{equation*} \sqrt{n}(\beta^n-\beta^\ast)\Rightarrow H^{-1}\mathbf{G},\quad \text{when    } \gamma>1/2,\end{equation*}
           \begin{equation*}n^{\gamma}(\beta^n-\beta^\ast)
        \Rightarrow-\eta H^{-1}\nabla_{\beta}\left( \Vert\beta^\ast\Vert_q\mathbb{E}_{P_\ast}[\vert L^\prime (\langle \bm{X},\beta\rangle,Y)\vert]\right),\quad \text{when    } 0<\gamma<1/2.\end{equation*}
\end{proof}
\end{document}